\documentclass[a4paper]{amsart}
\usepackage[utf8]{inputenc}
\usepackage{amsfonts,amsmath,amsthm,graphicx}
\usepackage[hidelinks]{hyperref}
\usepackage{array,float}
\usepackage[toc,page]{appendix}
\usepackage{caption, subcaption,comment}
\usepackage[foot]{amsaddr}
\usepackage[a4paper]{geometry}

\numberwithin{equation}{section}

\newcommand{\matdev}{\partial^{\bullet}}
\newcommand{\matdevtau}{\partial^{\bullet}_\tau}

\newcommand{\gradg}{\nabla_{\Gamma}}
\newcommand{\gradgh}{\nabla_{\Gamma_h}}
\newcommand{\lapg}{\Delta_{\Gamma}}

\newcommand{\invshn}[1]{\mathcal{G}_{S_h^{#1}}}
\newcommand{\normshn}[2]{\left\| #1 \right\|_{S_h^{#2}}}

\newcommand{\Xint}[1]{\mathchoice
	{\XXint\displaystyle\textstyle{#1}}%
	{\XXint\textstyle\scriptstyle{#1}}%
	{\XXint\scriptstyle\scriptscriptstyle{#1}}%
	{\XXint\scriptscriptstyle\scriptscriptstyle{#1}}%
	\!\int}
\newcommand{\XXint}[3]{{\setbox0=\hbox{$#1{#2#3}{\int}$ }
		\vcenter{\hbox{$#2#3$ }}\kern-.6\wd0}}

\newcommand{\dashint}{\Xint-}
\newcommand{\mval}[2]{\dashint_{#2} #1}
\DeclareMathOperator*{\esssup}{ess\,sup}

\newcommand{\Ech}{E^{\mathrm{CH}}}
\newcommand{\Echh}{E_h^{\mathrm{CH}}}

\theoremstyle{definition}%
\newtheorem{theorem}{Theorem}[section]
\newtheorem{corollary}[theorem]{Corollary}
\newtheorem{proposition}[theorem]{Proposition}
\newtheorem{lemma}[theorem]{Lemma}
\newtheorem{definition}[theorem]{Definition}

\newtheorem{remark}[theorem]{Remark}

\begin{document}
\title[Fully discrete ESFEM for the {Cahn}-{Hilliard} equation with a regular potential]{A fully discrete evolving surface finite element method for the {Cahn}-{Hilliard} equation with a regular potential}
\author{Charles M. Elliott\\
	Thomas Sales}
\address{Mathematics Institute, Zeeman Building, University of Warwick, Coventry CV4 7AL, UK}
\email{\href{mailto:c.m.elliott@warwick.ac.uk}{c.m.elliott@warwick.ac.uk}, \href{mailto:tom.sales@warwick.ac.uk}{tom.sales@warwick.ac.uk}}
\date{}

\begin{abstract}
	We study two fully discrete evolving surface finite element schemes for the Cahn-Hilliard equation on an evolving surface, given a smooth potential with polynomial growth.
	In particular we establish optimal order error bounds for a (fully implicit) backward Euler time-discretisation, and an implicit-explicit time-discretisation, with isoparametric surface finite elements discretising space.
\end{abstract}

\maketitle

\section{Introduction}
In this paper we present optimal order error bounds for fully discrete   evolving surface finite element (ESFEM) schemes for the Cahn-Hilliard equation posed on a closed, orientable evolving surface, $(\Gamma(t))_{t \in [0,T]} \subset \mathbb{R}^3$. 
The evolution of the surface is assumed to be sufficiently smooth, and known a priori.
In this setting, the Cahn-Hilliard equation, as formulated in \cite{caetano2021cahn}, is given by
\begin{align}
\begin{gathered}
	\matdev u + u (\gradg \cdot V )= \lapg w,\\
	w = - \varepsilon \lapg u + \frac{1}{\varepsilon}F'(u),
\end{gathered}
 \label{cheqn}
\end{align}
subject to the initial condition $u(0) = u_0$ for suitable initial data. 
For simplicity we do not consider arbitrary parametrisations of $\Gamma(t)$, and focus on the setting where  $\Gamma(t)$ is parametrised using the  advective tangential velocity associated to the model.
If one chooses an arbitrary parametrisation then the first equation of \eqref{cheqn} changes as in \cite{caetano2021cahn}.\\

In this paper we focus on a smooth potential  
$F \in C^2(\mathbb{R})$  such that $F(r) = F_1(r) - \frac{\theta}{2}r^2,$ and
\begin{enumerate}
	\item $F(r) \geq \beta$,
	\item $F_1 \geq 0$ is convex,
	\item $\exists q \in [1, \infty)$ such that $|F_1'(r)| \leq \alpha |r|^q + \alpha$,
\end{enumerate}
where $\alpha \geq 0$, $\beta \in \mathbb{R}$, and $q \geq 1$.
A notable example is the quartic potential, $F(r) = \frac{(1-r^2)^2}{4}$, as studied in \cite{elliott2015evolving}.
The system has an associated energy functional, known as the Ginzburg-Landau functional, given by 
\begin{align}
	\label{glfunctional}
	\Ech[u;t] := \int_{\Gamma(t)} \frac{\varepsilon |\gradg u|^2}{2} + \frac{1}{\varepsilon} F(u),
\end{align}
where the constant $\varepsilon > 0$ is proportional to the width of a diffuse interface separating phase  domains in which the solution $u$ takes values close to the minima of $F$.
 For a stationary domain it is known that the Cahn-Hilliard equation is the $H^{-1}$ gradient flow of the Ginzburg-Landau (see for instance \cite{bansch2023interfaces}).
This is not the case for an evolving surface, as remarked in \cite{elliott2015evolving}, but this functional remains useful in the the analysis nonetheless.\\

The Cahn-Hilliard equation originates from the work of Cahn and Hilliard in 1958, \cite{cahn1958free}, in modelling phase separation in a binary alloy.
From a modelling perspective, the function $u$ is understood as the concentration of the mixture. In the case of a double well potential the two minima are associated with two phases.
The Cahn-Hilliard equation naturally finds applications in metallurgy, for example in the studying phenomenon of spinodal decomposition and phase separation \cite{cahn1961spinodal,miller1995spinodal}.
Outside of this application, modified versions of the Cahn-Hilliard equation have found use in the modelling tumour growth \cite{garcke2022viscoelastic,garcke2016cahn}, and in the study of the dynamics of lipid biomembranes \cite{barrett2017finite,zhiliakov2021experimental}.
This study has typically been done on a stationary, Euclidean domain --- see for example \cite{elliott1989second,miranville2019cahn}.
The study of Cahn-Hilliard equations on evolving surfaces is motivated by applications such as \cite{eilks2008numerical,zimmermann2019isogeometric}. 
We refer to \cite{caetano2021cahn,caetano2023regularization,elliott2024navier} for recent results on the analysis.\\

The study of spatially discrete numerical schemes for solving the Cahn-Hilliard equation on evolving surfaces was initiated in 
\cite{elliott2015evolving} and extended in \cite{beschle2022stability}. Error bounds for fully discrete ESFEM schemes for second order parabolic equations  have been shown in \cite{dziuk2012fully,dziuk2012runge,lubich2013backward,frittelli2018numerical,kovacs2023maximal,kovacs2016error}.  We also note that the Cahn-Hilliard equation on a stationary surface was studied in \cite{du2011finite}, in which the authors show error bounds for a fully discrete scheme, and in \cite{li2018direct} the authors look at numerical results for a fully discrete scheme on an evolving surface.\\

To the authors' knowledge, this paper is the first to analyse fully discrete ESFEM schemes for a nonlinear fourth order problem.
The outline of this paper is as follows.
In Section \ref{section:prelim} we introduce some prerequisite results regarding evolving surfaces and ESFEM.
In Section \ref{section:fullydisc} we introduce our fully implicit scheme, and show well-posedness (and numerical stability).
Next, in Section \ref{section:error} we prove error bounds --- where we consider two cases based on the growth of $F_1$.
We then discuss in Section \ref{section:implicitexplicit} how these results adapt for an implicit-explicit time discretisation such as the convex-concave splitting introduced in \cite{blowey1992cahn,EllStu93}.
Finally in Section \ref{section:implementation} we end with some numerical results, which experimentally confirm the error bounds for linear finite elements.

\section{Preliminaries}
\label{section:prelim}
\subsection{Some geometric analysis}
Throughout, we will be considering, for a fixed $k \in \mathbb{N}$,  closed, connected, oriented $C^{k+1}$ surfaces in $\mathbb{R}^{3}$.
Given such a surface, $\Gamma$, we denote its unit oriented  normal vector at a point $x \in \Gamma$ by $\boldsymbol{\nu}(x)$.
As noted in \cite{deckelnick2005computation}, $\Gamma$ partitions $\mathbb{R}^{3}$ into two regions, an interior region which we denote $\Omega$, and an exterior region $\mathbb{R}^{3} \backslash \bar{\Omega}$.
The orientation of the normal is set by pointing $\boldsymbol{\nu}$ outwards from $\Omega$ and the oriented distance function for $\Gamma$ is defined by
$$ d(x) := \begin{cases}
	\inf_{y \in \Gamma} | x - y | , & x \in \mathbb{R}^{3} \backslash \bar{\Omega}\\
	-\inf_{y \in \Gamma} | x - y | , & x \in \bar{\Omega}.
\end{cases}  
$$
Moreover it is known, see \cite{foote1984regularity,gilbarg1977elliptic}, that, for a $C^{k+1}$ surface,  $d: \mathcal{N}(\Gamma) \rightarrow \mathbb{R}$ is a $C^{k+1}$ function, where $\mathcal{N}(\Gamma)$ is some open, tubular neighbourhood of $\Gamma$.
Given a point $x \in \mathcal{N}(\Gamma)$, we may uniquely express $x$ in Fermi coordinates as
\begin{align}
	\label{normalprojection}
	x = p(x) + d(x) \boldsymbol{\nu} (p(x)),
\end{align}
for a  unique $p(x) \in \Gamma$ where 
$p : \mathcal{N}(\Gamma) \rightarrow \Gamma$ is the normal projection operator.
This will be used later in the triangulation of the surface as a way of relating functions on a triangulated surface to the exact surface.

\begin{definition}
	Let $f : \Gamma \rightarrow \mathbb{R}$ be such that we have a differentiable extension of $f$, say $f^e$, defined on an open neighbourhood of $\Gamma$.
	We define the tangential gradient of $f$ at $x \in \Gamma$ to be
	$$ \gradg f(x) = \nabla {f}^e(x) - \left( \nabla f^e(x) \cdot \boldsymbol{\nu}(x) \right) \boldsymbol{\nu}(x) .$$
	We may express this componentwise as
	$$ \gradg f = \left( \underline{D}_1 f,..., \underline{D}_{n+1} f \right).$$
	It can be shown that this expression is independent of the choice of extension, ${f}^e$.\\
	
	We then define the Laplace-Beltrami operator of $f$ at $x \in \Gamma$ to be
	$$ \Delta_{\Gamma} f = \gradg \cdot \gradg f = \sum_{i=1}^{n+1} \underline{D}_i \underline{D}_i f.$$
\end{definition}


\subsubsection{Sobolev spaces}
As in \cite{deckelnick2005computation,gilbarg1977elliptic}, we define Sobolev spaces on $\Gamma$ as follows.

\begin{definition}
	
	For $p \in [1, \infty]$, the Sobolev space $H^{1,p}(\Gamma)$ is  defined by
	$$ H^{1,p}(\Gamma) := \{ f \in L^p(\Gamma) \ | \ \underline{D}_i f \in L^p(\Gamma), \ i=1,...,n+1 \},$$
	and higher order spaces ($m \in \mathbb{N}$) are defined recursively by
	$$ H^{m,p}(\Gamma) := \{ f \in H^{m-1,p}(\Gamma) \ | \ \underline{D}_i f \in H^{m-1,p}(\Gamma), \ i = 1,...,n+1  \}, $$
	where $H^{0,p}(\Gamma) := L^p(\Gamma)$.
	These Sobolev spaces are known to be Banach spaces when equipped with norm,
	$$
	\| f \|_{H^{m,p}(\Gamma)} := \begin{cases}
		\left( \sum_{|\alpha| = 0}^{m}  \| \underline{D}^{\alpha} f \|_{L^p(\Gamma)}^p  \right)^{\frac{1}{p}}, & p \in [1, \infty),\\
		\max_{|\alpha|=1,...,m} \|\underline{D}^\alpha f \|_{L^\infty(\Gamma)}, & p = \infty,
	\end{cases}$$
	where we consider all weak derivatives of order $|\alpha|$.
We use shorthand notation, $H^{m}(\Gamma) := H^{m,2}(\Gamma)$, for the case $p = 2$.
The spaces $H^m(\Gamma)$ are in fact Hilbert spaces, where one can obtain the appropriate inner product by polarisation.
\end{definition}
We refer the reader to \cite{aubin2012nonlinear,hebey2000nonlinear} for further details on Sobolev spaces on manifolds, such as proofs of Sobolev embeddings.\\

Next we introduce some notation which will be used throughout.
For a $\mathcal{H}^2 -$measurable set, $X \subset \mathbb{R}^{3}$, we denote the $\mathcal{H}^2$ measure of $X$ by
$$ |X| := \mathcal{H}^2(X),$$
and for $f \in L^1(X)$ we denote the mean value of $f$ on $X$ by
$$ \mval{f}{X} := \frac{1}{|X|} \int_X f .$$

\subsection{Evolving surfaces}
We now generalise to surfaces which evolve in time.
\begin{definition}[$C^{k+1}$ evolving surface]
	Let $\Gamma_0 \subset \mathbb{R}^3$ be a closed, connected, orientable $C^{k+1}$ surface and let $\Phi :\Gamma_0 \times [0,T] \rightarrow \mathbb{R}^3$ be a smooth map such that
	\begin{enumerate}
		\item For all $t \in [0,T]$,
		$$\Phi(\cdot,t): \Gamma_0 \rightarrow \Phi(\Gamma_0,t) =: \Gamma(t)$$
		is a $C^{k+1}$ diffeomorphism.
		\item $\Phi(\cdot,0) = \mathrm{id}_{\Gamma_0}$.
	\end{enumerate}
	Then we call the family $(\Gamma(t))_{t \in [0,T]}$ a $C^{k+1}$ evolving surface.
\end{definition}
It follows that $\Gamma(t)$ is closed, connected and orientable for all $t$.
We define the space-time surface to be the set
$$\mathcal{G}_T = \bigcup_{t \in [0,T]} \Gamma(t) \times \{ t \},$$
and given $(x,t) \in \mathcal{G}_T$ we denote the unit normal by $\boldsymbol{\nu}(x,t)$.
We will assume throughout that there is a velocity field $V \in C^1([0,T];C^{k+1}(\mathbb{R}^3;\mathbb{R}^3))$ such that for $t \in [0,T]$ and $x \in \Gamma_0$,
\begin{align*}
	\frac{d}{dt} \Phi(x,t) &= V(\Phi(x,t),t),\\
	\Phi(x,0) &= x.
\end{align*}
By compactness of $\mathcal{G}_T$, and assumed smoothness of $V$, there is a constant $C_V$ independent of $t$ such that
$$\|V(t)\|_{C^{k+1}(\Gamma(t))} \leq C_V ,$$
for all $t \in [0,T]$.

\subsubsection{Time-dependent Lebesgue/Bochner spaces}

Next we introduce a way of relating functions on the evolving surface back to the initial surface, which will be necessary for defining the evolving function spaces.
Let $t \in [0,T]$, $\eta \in H^{m,p}(\Gamma_0)$ and $\zeta \in H^{m,p}(\Gamma(t))$ for some $m = 0,...,k+1$, and $p\in [1,\infty]$.
We define the pushforward of $\eta$ by
$$ \Phi_t \eta = \eta(\Phi(\cdot,t)) \in H^{m,p}(\Gamma(t)),$$
and the pullback of $\zeta$ by
$$ \Phi_{-t} \zeta = \zeta(\Phi(\cdot,t)^{-1}) \in H^{m,p}(\Gamma_0).$$
It can be shown that the pairs $(H^{m,p}(\Gamma(t)), \Phi_t)$ are compatible in the sense of \cite{alphonse2023function,alphonse2015abstract} for $m = 0,...,k+1$.
Compatibility of these spaces allows one to obtain Sobolev inequalities independent of $t$.
With these definitions, we can define time-dependent Bochner spaces.

\begin{definition}
	In the following we let $X(t)$ denote a time-dependent Banach space, for instance $H^{m,p}(\Gamma(t))$.
	The space $L^2_X$ consists of (equivalence classes of) functions
	\begin{gather*}
		u:[0,T] \rightarrow \bigcup_{t \in [0,T]} X(t) \times \{ t \},\\
		t \mapsto (\bar{u}(t), t),
	\end{gather*}
	such that $ \Phi_{-(\cdot)} \bar{u}(\cdot) \in L^2(0,T;X(0))$.
	We identify $u$ with $\bar{u}$.
    This space is equipped with norm,
    \[ \|u\|_{L^2_X} = \left( \int_0^T \|u(t)\|_{X(t)}^2 \right)^\frac{1}{2}, \]
	and if we have further that $X(t)$ are a family of Hilbert spaces then this norm is induced by the inner product
	\[(u,v)_{L^2_X} = \int_0^T (u(t),v(t))_{X(t)},\]
	for $u,v \in L^2_X$.
	As justified in \cite{alphonse2015abstract}, we make the identification $(L^2_{X})^* \cong L^2_{X^*}$, and for $X = H^1$ we write $L^2_{H^{-1}} := (L^2_{H^1})^*$.\\
	
	One can similarly define $L^p_X$ for $p \in [1, \infty]$, which is equipped with a norm
	\[\|u\|_{L^p_X} := \begin{cases}
		\left( \int_0^T \|u(t)\|^p_{X(t)} \right)^{\frac{1}{p}}, & p \in [1,\infty),\\
		\underset{t \in [0,T]}{\esssup} \|u(t)\|_{X(t)}, & p = \infty.
	\end{cases}\]
	We refer the reader to \cite{alphonse2023function} for further details.
\end{definition}

\begin{definition}[Strong material derivative]
	Let $u : \mathcal{G}_T \rightarrow \mathbb{R}$ be a $C^1_{C^1}$ function.
    We define the strong material time derivative by
\[\matdev u = \Phi_t \left( \frac{d}{dt} \Phi_{-t} u \right).\]
\end{definition}

Details on the material derivative are found in \cite{alphonse2015abstract,alphonse2023function}.
As in the stationary setting, this can be generalised to define a weak material derivative.

\begin{definition}[Weak material derivative]
	Let $u \in L^2_{H^1}$.
	A function $v \in L^2_{H^{-1}}$ is said to be the weak material time derivative of $u$ if for all $\eta \in \mathcal{D}_{H^1}(0,T)$ we have
    \begin{multline*}
        \int_0^T \langle v(t), \eta(t) \rangle_{H^{-1}(\Gamma(t)) \times H^1(\Gamma(t))} = - \int_0^T \left(u(t), \matdev \eta(t) \right)_{L^2(\Gamma(t))}\\
        -\int_0^T \int_{\Gamma(t)} u(t) \eta(t) \gradg \cdot V(t),
    \end{multline*}
	where
    \[\mathcal{D}_{H^1}(0,T) := \left\{ u \in L^2_{H^1} \mid \Phi_{-t} u(t) \in C^{\infty}_c(0,T;H^1(\Gamma_0)) \right\}.\]
    We abuse notation and write $v = \matdev u$.
\end{definition}

Clearly if $u \in L^2_{H^1}$ has a strong material time derivative it has a weak material time derivative, and the two coincide.\\

We introduce the following  bilinear forms to be used throughout,
\begin{align*}
	m_*(t;\hat{\eta}, \zeta) &:= \langle \hat{\eta}, \zeta \rangle_{H^{-1}(\Gamma(t)) \times H^1(\Gamma(t))},\\
	m(t;\eta,\zeta) &:=  \int_{\Gamma(t)} \eta \zeta, \\
	g(t;\eta,\zeta)  &:=  \int_{\Gamma(t)} \eta \zeta \gradg \cdot V(t),\\
	a(t;\eta,\zeta) &:= \int_{\Gamma(t)} \gradg \eta \cdot \gradg \zeta,
\end{align*}
where $\eta, \zeta \in H^1(\Gamma(t))$, $\hat{\eta} \in H^{-1}(\Gamma(t))$.
The argument in $t$ will often be omitted, as above.
For weakly differentiable functions we have the following transport theorem.
\begin{proposition}[\cite{dziuk2013finite}, Lemma 5.2]
	\label{transport2}
	Let $\eta, \zeta \in H^1_{H^{-1}} \cap L^2_{L^2}.$
	Then $t \mapsto m(\eta(t), \zeta(t))$ is absolutely continuous and such that
	$$ \frac{d}{dt} m(\eta, \zeta) = m_*(\matdev \eta, \zeta) + m_*(\eta, \matdev \zeta) + g(\eta, \zeta).$$
	Moreover, if $\gradg \matdev \eta, \gradg \matdev \zeta \in L^2_{L^2(\Gamma)}$ then $t \mapsto a_S(\eta(t), \zeta(t))$ is absolutely continuous and such that
	$$ \frac{d}{dt} a(\eta, \zeta) = a(\matdev \eta, \zeta) + a(\eta, \matdev \zeta) + b(\eta, \zeta).$$
	Here
	$$ b(\eta, \zeta) = \int_{\Gamma(t)} \mathbf{B}(V) \gradg \eta \cdot \gradg \zeta , $$
	where
	$$ \mathbf{B}(V) = \left( (\gradg \cdot V)\mathrm{id} - (\gradg V + (\gradg V)^T) \right).$$
\end{proposition}

Using this notation the weak formulation of \eqref{cheqn} is as follows.
Find $u \in H^1_{H^{-1}} \cap L^\infty_{H^1}, w \in L^2_{H^1}$ such that
\begin{gather}
	m_*(u,\eta) + g(u, \eta) + a(w,\eta) = 0, \label{cheqn1}\\
	m(w,\eta) = \varepsilon a(u, \eta) + \frac{1}{\varepsilon}m(F'(u),\eta), \label{cheqn2}
\end{gather}
for all $\eta \in H^1(\Gamma(t))$, almost all $t \in [0,T]$ and such that $u(0) = u_0$ for some initial data $u_0$.
We refer the reader to \cite{caetano2021cahn} for the corresponding analysis.
\subsection{Triangulated surfaces}
\subsubsection{Construction}
Let $\Gamma \subset \mathbb{R}^3$ be a closed, oriented $C^{k+1}$ surface.
We introduce a discretised version of this surface, denoted $\Gamma_h$, which we call an triangulated (or interpolated) surface.

\begin{definition}
	We let $(x_i)_{i=1,...,N_h} \subset \Gamma$ be a collection of nodes used to define a set of triangles\footnote{Here we mean ``triangles'' that are in fact curved for when we consider $k > 1$ - see \cite{elliott2021unified}.
	For $k = 1$ these are indeed triangles in the usual sense.} $\mathcal{T}_h$.
	The triangulated surface, $\Gamma_h$, is defined by an admissible subdivision of triangles, $\mathcal{T}_h$, such that
	$$\bigcup_{K \in \mathcal{T}_h} K = \Gamma_h.$$
	If $K_1, K_2 \in \mathcal{T}_h$ are distinct, then we have $K_1^{\circ} \cap K_2^{\circ} = \emptyset$, and if $ \bar{K}_1 \cap \bar{K}_2 \neq \emptyset$ then this intersection is either a node of the triangulation, or an edge connecting two adjacent nodes.\\
	
	For $K \in \mathcal{T}_h$ we define following quantities
	\begin{gather*}
    h_K := \mathrm{diam}(\tilde{K}),\\
    \rho_K := \sup\{ \mathrm{diam}(B) \ | \ B \text{ is a }2\text{-dimensional ball contained in } \tilde{K} \},
    \end{gather*}
	where $\tilde{K}$ is the affine part of $K$, see \cite{elliott2021unified}.
	We assume that the subdivision $\mathcal{T}_h$ is quasi-uniform, that is, there exists $\rho > 0$ such that for all $h \in (0,h_0)$
    \[\min \{ \rho_K \ | \ K \in \mathcal{T}_h \} \geq \rho h .\]
\end{definition}

Throughout this paper we work with polynomial Lagrange finite elements --- that is our degrees of freedom are given by the point evaluations at some set points.
We will denote the set of shape functions as
\[S_h := \left\{ \phi_h \in C(\Gamma_h) \mid \phi_h |_{K} \circ F_{K}^{-1} \text{ is a polynomial of degree } k, \ K \in \mathcal{T}_h \right\},\]
where $F_K$ is some appropriate reference map to a reference element, we refer to \cite{elliott2021unified} for details.
The effect of the reference mapping here is so that the functions $\phi_h |_{K} \circ F_{K}^{-1}$ are polynomials over the ``flat triangles''.
Here the polynomial degree $k$ is fixed as the same degree of our geometric approximation, $\Gamma_h$, of $\Gamma$.
This choice of matching the degree of the shape functions to the degree of the geometric approximations is known as isoparametric finite elements.
Here the normal $\boldsymbol{\nu}_h$ is defined piecewise on each element of $\Gamma_h$.
As such this gives rise to a discrete tangential gradient, $\gradgh$, defined element-wise on $\Gamma_h$.\\

We note that (locally) one can view $\Gamma$ as the graph of a $C^{k+1}$ function, $g: U \subset \mathbb{R}^2 \rightarrow \mathbb{R}$.
Hence the construction of $\Gamma_h$ can be understood locally as being the graph of a polynomial interpolant of $g$.
We omit further details on construction of higher order polynomial approximations of $\Gamma$ but refer the reader to \cite{bernardi1989optimal,demlow2009higher,elliott2021unified}.

\subsubsection{Lifts and interpolation}
Next we relate functions on $\Gamma_h$ and $\Gamma$ by defining lifts.
We will assume that our triangulated surface $\Gamma_h$ is such that $\Gamma_h \subset \mathcal{N}$ for $\mathcal{N}(\Gamma)$ a sufficiently small tubular neighbourhood of $\Gamma$.
This is possible in practise by considering a sufficiently fine triangulation.
This allows us to define lifts of functions.
\begin{definition}
	For a function $\eta_h : \Gamma_h \rightarrow \mathbb{R}$ we implicitly define the lift operation on $\eta_h$ by
	$$\eta_h^{\ell}(p(x)) := \eta_h(x),$$
	where $p$ is the normal projection operator \eqref{normalprojection}.
 Similarly, for $\eta : \Gamma \rightarrow \mathbb{R}$ we define the inverse lift by
 \[\eta^{-\ell}(x) = \eta(p(x)) . \]
	
\end{definition}

In \cite{elliott2021unified} the following result concerning lifts of functions is proven.
\begin{lemma}
	There exists constants $C_1, C_2$, independent of $h$, such that for $\eta_h \in H^1(\Gamma_h)$
	\begin{gather}
		C_1 \| \eta_h^{\ell} \|_{L^2(\Gamma)} \leq \| \eta_h \|_{L^2(\Gamma_h)} \leq C_2 \| \eta_h^{\ell} \|_{L^2(\Gamma)}, \label{lift1}\\
		C_1 \| \gradg \eta_h^{\ell} \|_{L^2(\Gamma)} \leq \| \gradgh \eta_h \|_{L^2(\Gamma_h)} \leq C_2 \| \gradg \eta_h^{\ell} \|_{L^2(\Gamma)}. \label{lift2}
	\end{gather}
\end{lemma}

This shows that there exist constants $C_1, C_2$ independent of $h$ such that
\[ C_1 |\Gamma| \leq |\Gamma_h| \leq C_2 |\Gamma| .\] 
Another useful consequence of the stability of the lift is that it allows one to obtain a Poincaré inequality on $\Gamma_h$ independent of $h$.\\

We note that we consider Sobolev functions on $\Gamma_h$ in the sense of ``broken Sobolev spaces''.
That is to say, we understand the norm $\|\cdot\|_{H^{1,p}(\Gamma_h)}$ to be
\[ \|\eta_h \|_{H^{1,p}(\Gamma_h)} = \begin{cases}
	\left( \sum_{K \in {\mathcal{T}_h}} \|\eta_h\|_{H^{1,p}(K)}^p \right)^\frac{1}{p}, & p \in [1,\infty),\\
	\max_{K \in \mathcal{T}_h} \|\eta_h\|_{H^{1,\infty}(K)}, & p = \infty,
\end{cases} \]
where $\|\cdot\|_{H^{1,p}(K)}$ is understood in the usual sense.
We refer the reader to \cite{brenner2008mathematical,elliott2021unified} for further details.

\begin{definition}
	We say that our triangulation is exact if the lifted triangles $K^{\ell} := \{ x^{\ell} \mid x \in \Gamma_h \}$ form a conforming subdivision of $\Gamma$.
\end{definition}

\subsubsection{Evolving triangulated surfaces}
We consider an evolving surface, $(\Gamma(t))_{t \in [0,T]}$, and construct an evolving triangulated surface as follows.
Firstly, we construct an admissible triangulation,  $\mathcal{T}_h(0)$, of $\Gamma_0$, with nodes $(x_{i,0})_{i=1,...,N_h}$, as in the beginning of this section.
We denote this triangulated surface as $\Gamma_{h}(0)$.
The nodes of $\Gamma_{h}(0)$ then evolve in time according to the ODE,
$$ \frac{d}{dt} x_i(t) = V(x_i(t),t), \qquad x_i(0) = x_{i,0},$$
where $V$ is the velocity field associated with the evolution of $\Gamma(t)$.
By assumptions on the smoothness of $V$ and the Picard-Lindel\"off theorem we find that there exists a unique nodal evolution $x_i(t)$, given $x_{i,0}$.
This induces a triangulation $\mathcal{T}_h(t)$, where $K(0) \in \mathcal{T}_h(0)$ gives rise to a triangle $ K(t) \in \mathcal{T}_h(t)$ where the nodes have evolved by the above ODE.
The triangulated surface $\Gamma_h(t)$ is then defined as
$$ \Gamma_h(t) := \bigcup_{K(t) \in \mathcal{T}_h(t)} K(t),$$
which will be admissible by construction of $\Gamma_{h}(0)$.
Here the $h$ parameter is now defined to be
$$ h:= \sup_{t \in [0,T]} \max_{K(t) \in \mathcal{T}_h(t)} h_{K(t)}. $$
\begin{definition}
\begin{itemize}
    \item The evolving triangulated surface, $\Gamma_h(t)$, is said to be uniformly quasi-uniform if there exists $\rho > 0$ such that for all $t \in [0,T]$, and $h \in (0, h_0)$ we have
    \[\min\{ \rho_{K(t)} \ | \ K(t) \in \mathcal{T}_h(t) \} \geq \rho h .\]
    \item We say the triangulation is an exact evolving triangulation if for all $t \in [0,T]$
    \[\bigcup_{K(t) \in \mathcal{T}_h(t)} K^{\ell}(t) = \Gamma(t).\]
\end{itemize}
\end{definition}
In our analysis we consider an evolving triangulated surface, $\Gamma_h(t)$ which is uniformly quasi-uniform, and exact all $t \in [0,T]$.
We also assume that for each $t \in [0,T]$ that $\Gamma_h(t) \subset \mathcal{N}(\Gamma(t))$ so that we may define lift at all time $t \in [0,T]$.
We denote the discrete spacetime surface as
\[\mathcal{G}_{h,T} := \bigcup_{t \in [0,T]} \Gamma_h(t) \times \{t\}.\]

We note that as the domain is evolving, the set of basis functions also evolves in time.
As such, we denote the set of basis functions at time $t$ to be
$$ S_h(t) = \left\{ \phi_h \in C(\Gamma_h(t)) \mid \phi_h |_{K(t)} \circ F_K(t)^{-1} \text{ is a polynomial of degree } k,  K(t) \in \mathcal{T}_h(t) \right\},$$
where the reference map $F_K$ now has some time dependence.
This definition allows one to characterise the velocity of the surface $\Gamma_h(t)$, as an arbitrary point $x(t) \in \Gamma_h(t)$ will evolve according to the discrete velocity, $V_h$, given by
\[\frac{d}{dt} x(t) = V_h(x(t),t) := \sum_{i=1}^{N_h} \dot{x}_i(t) \phi_i(x(t),t) = \sum_{i=1}^{N_h} V(x_i(t),t) \phi_i(x(t),t),\]
where $\phi_i(t)$ is the `$i$'th nodal basis function of $\Gamma_h(t)$.
From this we observe that $V_h$ is the Lagrange interpolant of $V$.\\

The evolution of the nodes by $V$ induces map $\Phi^h:\Gamma_h(0) \rightarrow \Gamma_h(t)$, from which we may define a discrete (strong) material time derivative, given by
\[\matdev_h \eta_h := \Phi^h_t \left( \frac{d}{dt} \Phi^h_{-t} \eta_h \right) ,\]
for a sufficiently smooth $\eta_h$.
Here $\Phi^h_t, \Phi^h_{-t}$ are the pushforward/pullback maps respectively, defined similarly to the continuous surface.
One can similarly define a weak discrete material derivative in the standard way.\\

An important consequence of this is the transport property of basis functions, that is
\[\matdev_h \phi_i = 0\]
for all of the nodal basis functions of $\Gamma_h(t)$.
This follows due to the fact that $\phi_i(t) = \Phi^h_t \phi_i(0)$ for each basis function.
This is an important property, which is exploited in implementing evolving surface finite element schemes, as it eliminates any velocity terms in the formulation.\\

We now state a discrete analogue of the transport theorem, Proposition \ref{transport2}.
Here we denote the (time-dependent) bilinear forms by
\begin{align*}
	m_h(t;\eta_h, \zeta_h) &:= \int_{\Gamma_h(t)} \eta_h \zeta_h,\\
	a_h(t;\eta_h, \zeta_h) &:= \int_{\Gamma_h(t)} \gradgh \eta_h \cdot \gradgh \zeta_h,\\
	g_h(t;\eta_h, \zeta_h) &:= \int_{\Gamma_h(t)} \eta_h \zeta_h \gradgh \cdot V_h,
\end{align*}
where we typically omit the $t$ argument, and it will be clear from context at what time the bilinear forms are understood.
We then have a discrete transport theorem for these bilinear forms.
\begin{proposition}
	\label{transport3}
	Let $\eta_h, \zeta_h \in S_h^T$, then we have
	\begin{gather*}
		\frac{d}{dt} m_h(\eta_h, \zeta_h) = m_h(\matdev_h \eta_h, \zeta_h) + m_h(\eta_h, \matdev_h \zeta_h) + g_h(\eta_h, \zeta_h),\\
		\frac{d}{dt} a_h(\eta_h, \zeta_h) = a_h(\matdev_h \eta_h, \zeta_h) + a_h(\eta_h, \matdev_h \zeta_h) + b_h(\eta_h, \zeta_h).
	\end{gather*}
	Here
	$$ b_h(\eta_h, \zeta_h) = \int_{\Gamma_h(t)} \mathbf{B}_h(V_h) \gradgh \eta_h \cdot \gradgh \zeta_h , $$
	where
	$$ \mathbf{B}_h(V_h) = \left( (\gradgh \cdot V_h)\mathrm{id} - (\gradgh V_h + (\gradgh V_h)^T) \right).$$
\end{proposition}

We will not take note of the domain when discussing evolving Bochner spaces.
For example, we will write $L^2_{H^1}$, to be understood as either $L^2_{H^1}$ on $(\Gamma(t))_{t \in [0,T]}$ or $(\Gamma_h(t))_{t \in [0,T]}$, as it will be clear which we mean from context.

\subsubsection{Geometric perturbation estimates and the Ritz projection}
Here we state several results which are crucial to the numerical analysis of surface PDEs.
We state the general results for isoparametric surface finite elements of order $k$.
Firstly, as noted in \cite{elliott2021unified}, one can obtain another material derivative by (inverse) lifting a function onto $\Gamma_h(t)$, differentiating, and lifting back onto $\Gamma(t)$.
This lifted material derivative is useful as it is intermediate between $\matdev_h$ and $\matdev$.
\begin{definition}
	Let $\eta \in C^1_{L^2}$, then we define the (strong) lifted material derivative of $\eta$ as
	$$ \matdev_\ell \eta = \left( \matdev_h \eta^{-\ell} \right)^\ell.$$
\end{definition}
This allows one to obtain alternate versions of the transport theorem, Proposition \ref{transport2}.
\begin{proposition}
	\label{transport4}
	Let $\eta, \zeta \in C^1_{L^2}$, then
	\begin{align*}
		\frac{d}{dt} m(t; \eta, \zeta) = m(t; \matdev_\ell, \zeta) + m(t; \eta, \matdev_\ell \zeta) + g_\ell (t; \eta, \zeta),
	\end{align*}
	where
	$$g_\ell (t; \eta, \zeta) = \int_{\Gamma(t)} \eta \zeta (\gradg \cdot V_h^\ell).$$
	Similarly, if $\eta, \zeta \in C^1_{H^1(\Gamma)}$, then
	\begin{align*}
		\frac{d}{dt} a(t; \eta, \zeta) = a(t; \matdev_\ell, \zeta) + a(t; \eta, \matdev_\ell \zeta) + b_\ell (t; \eta, \zeta),
	\end{align*}
	where
	$$ b_\ell(\eta, \zeta) = \int_{\Gamma(t)} \mathbf{B}(V_h^\ell) \gradg \eta \cdot \gradg \zeta , $$
	and $\mathbf{B}(V_h^\ell)$ is as in Proposition \ref{transport2}.
\end{proposition}
This allows one to define a weak lifted material derivative in the usual way.
We can then related $\matdev_\ell$ and $\matdev$ through the following result.
\begin{lemma}[\cite{elliott2021unified}, Lemma 9.25]
	Let $\eta \in H^1_{H^1}$.
	Then we have
	\begin{align}
		\|\matdev \eta - \matdev_\ell \eta\|_{L^2(\Gamma(t))} \leq C h^{k+1} \|\eta\|_{H^1(\Gamma(t))}.\label{derivativedifference1}
	\end{align}
	If we have that $\eta \in H^1_{H^2}$, then
	\begin{align}
		\|\gradg(\matdev \eta - \matdev_\ell \eta)\|_{L^2(\Gamma(t))} \leq C h^k \|\eta\|_{H^2(\Gamma(t))}.\label{derivativedifference2}
	\end{align}
\end{lemma}
We now state some results which allow us to compare bilinear forms on $\Gamma(t)$ and $\Gamma_h(t)$.
The following results are proven in \cite{elliott2021unified}.
\begin{lemma}
	Let $\eta_h, \zeta_h \in H^1(\Gamma_h(t))$, and $h$ be sufficiently small.
	Then there exists a constant $C > 0$, independent of $t,h$, such that
	\begin{align}
		\left| m \left(t; \eta_h^{\ell}, \zeta_h^{\ell} \right) - m_h \left(t; \eta_h, \zeta_h \right) \right| &\leq C h^{k+1} \| \eta_h \|_{L^2(\Gamma_h(t))} \| \zeta_h \|_{L^2(\Gamma_h(t))},\label{perturb1}\\
		\left| a \left(t; \eta_h^{\ell}, \zeta_h^{\ell} \right) - a_h \left(t; \eta_h, \zeta_h \right) \right| &\leq C h^{k+1} \| \gradgh \eta_h \|_{L^2(\Gamma_h(t))} \| \gradgh \zeta_h \|_{L^2(\Gamma_h(t))},\label{perturb2}\\
		\left| g_\ell \left(t; \eta_h^{\ell}, \zeta_h^{\ell} \right) - g_h \left(t; \eta_h, \zeta_h \right) \right| &\leq C h^{k+1} \| \eta_h \|_{L^2(\Gamma_h(t))} \| \zeta_h \|_{L^2(\Gamma_h(t))},\label{perturb3}\\
		\left| b_\ell \left(t; \eta_h^{\ell}, \zeta_h^{\ell} \right) - b_h \left(t; \eta_h, \zeta_h \right) \right| &\leq C h^{k+1} \| \gradgh \eta_h \|_{L^2(\Gamma_h(t))} \| \gradgh \zeta_h \|_{L^2(\Gamma_h(t))},\label{perturb4}\\
		\left| g \left(t; \eta_h^{\ell}, \zeta_h^{\ell} \right) - g_\ell \left(t; \eta_h^{\ell}, \zeta_h^{\ell} \right) \right| &\leq C h^k \| \eta_h \|_{H^1(\Gamma_h(t))} \| \zeta_h \|_{H^1(\Gamma_h(t))},\label{perturb5}\\
		\left| b \left(t; \eta_h^{\ell}, \zeta_h^{\ell} \right) - b_\ell \left(t; \eta_h^{\ell}, \zeta_h^{\ell} \right) \right| &\leq C h^k \| \eta_h \|_{H^1(\Gamma_h(t))} \| \zeta_h \|_{H^1(\Gamma_h(t))}.\label{perturb6}
	\end{align}
\end{lemma}

Next we introduce a projection onto the shape functions which is useful in the error analysis for surface finite elements.
\begin{definition}
	For $z \in H^1(\Gamma(t))$ we define the Ritz projection\footnote{Some authors define the Ritz projection differently, see \cite{elliott2015evolving} Remark 3.4.}, $\Pi_h z \in S_h(t)$, to be the unique solution of
	\begin{equation}
		\label{ritz}
		a_h(\Pi_h z , \phi_h) = a(z , \phi_h^{\ell}),
	\end{equation}
	for all $\phi_h \in S_h(t)$, subject to the condition
	$$ \int_{\Gamma_h(t)} \Pi_h z = \int_{\Gamma(t)} z.$$
	We denote the lift of the Ritz projection by
	$\pi_h z = (\Pi_h z)^{\ell}.$
\end{definition}
One has the following bounds for the Ritz projection, for which we refer the reader to \cite{elliott2015evolving,elliott2021unified}.

\begin{lemma}
	For $z \in H^1(\Gamma(t))$ we have the following,
	\begin{gather}
		\label{ritz1}
		\| \pi_h z \|_{H^1(\Gamma(t))} \leq C \| z \|_{H^1(\Gamma(t))},\\
		\label{ritz2}
		\| \pi_h z - z \|_{L^2(\Gamma(t))} \leq C h \|z\|_{H^1(\Gamma(t))}.
	\end{gather}
	Moreover, if $z \in H^{2}(\Gamma(t))$ then
	\begin{gather}
		\|\Pi_h z\|_{L^\infty(\Gamma_h(t))} = \|\pi_h z\|_{L^\infty(\Gamma(t))} \leq C \|z\|_{H^2(\Gamma(t))}. \label{ritz4}
	\end{gather}
	If further still we have $z \in H^{k+1}(\Gamma(t))$ then
	\begin{gather}
		\label{ritz3}
		\| \pi_h z - z \|_{L^2(\Gamma(t))} + h \| \gradg (\pi_h z - z) \|_{L^2(\Gamma(t))} \leq C h^{k+1} \|z \|_{H^{k+1}(\Gamma(t))}.
	\end{gather}
\end{lemma}

\begin{lemma}
	For $z : \mathcal{G}_T \rightarrow \mathbb{R}$ with $z, \matdev{z} \in H^2(\Gamma(t))$ then $\matdev_h \Pi_h z \in S_h(t)$ exists and
	\begin{equation}
		\label{ritzddtnorm}
		\| \matdev_h \Pi_h z\|_{H^1(\Gamma_h(t))} \leq C\left( \|z\|_{H^2(\Gamma(t))} + \|\matdev z\|_{H^2(\Gamma(t))} \right).
	\end{equation}
	If we have further regularity, so that $z, \matdev{z} \in H^{k+1}(\Gamma(t))$
	\begin{multline}
		\label{ritzddt}
		\left \|  \matdev_\ell(\pi_h z - z) \right\|_{L^2(\Gamma(t))} + h \left \| \gradg  \matdev_\ell(\pi_h z - z)\right\|_{L^2(\Gamma(t))}\\
  \leq C h^{k+1} \left( \|z\|_{H^{k+1}(\Gamma(t))} + \left\| \matdev z \right\|_{H^{k+1}(\Gamma(t))} \right),
	\end{multline}
	where $C$ is a constant independent of $h,t$.
\end{lemma}

\section{A fully implicit scheme}
\label{section:fullydisc}
\subsection{Scheme and main results}
We begin by introducing some notation to be used throughout.
	\begin{enumerate}
 \item We define $\Gamma_h^n := \Gamma_h(t_n)$ for the time discrete surfaces, and $S_h^n := S_h(t_n)$ for the corresponding space of shape functions.
    \item For functions $\phi_h^{n-1} \in S_h^{n-1}, \phi_h^n \in S_h^n$ we define functions $\overline{\phi_h^{n-1}} \in S_h^n$, and $\underline{\phi_h^n} \in S_h^{n-1}$ to be the function in $S_h^n$ (respectively $S_h^{n-1}$) with the same nodal value as $\phi_h^{n-1}$ (respectively $\phi_h^n$).
	\item We define the fully discrete material time derivative for a sequence of functions $(\phi_h^n)_n$, where $\phi_h^n \in S_h^n$ by
    \[\matdevtau \phi_h^n := \frac{1}{\tau}\left( \phi_h^n - \overline{\phi_h^{n-1}} \right) \in S_h^n.\]
	\end{enumerate}
It is worth noting that in \cite{dziuk2012fully} the authors define the discrete time derivative so that it belongs to $S_h^{n-1}$, which in our notation is $\underline{\matdevtau \phi_h^n}$.\\

We consider the following numerical scheme:
Find given some initial data $U_h^0 = U_{h,0} \in S_h^0$, find $U_h^n, W_h^n \in S_h^n$ solving
\begin{gather}
	\begin{aligned}
    \frac{1}{\tau} \left(m_h(t_n;U_h^n, \phi_h^n) - m_h(t_{n-1};U_h^{n-1}, \phi_h^{n-1}) \right) + a_h(t_n;W_h^n,\phi_h^n)\\ = m_h(t_{n-1};U_h^{n-1}, \underline{\matdevtau \phi_h^n}),
    \end{aligned}\label{fdiscfecheqn1}\\ 
	m_h(t_n;W_h^n, \phi_h^n)  =\varepsilon a_{h}(t_n;U_h^n,\phi_h^n) + \frac{1}{\varepsilon} m_h(t_n;F'(U_h^n), \phi_h^n),\label{fdiscfecheqn2}
\end{gather}
for all $\phi_h^{n-1}\in S_h^{n-1}, \phi_h^{n}\in S_h^{n}$.
As usual we will often omit the time argument, as it will be clear from the arguments in the bilinear form.
Here the initial data $U_{h,0} \in S_h^0$ is an approximation of some sufficiently regular function, $u_0$ on $\Gamma(0)$.
More specifically, we assume that there is some $u_0 \in H^{k+1}(\Gamma(0))$ such that
\begin{align}
	\|u_0 - U_{h,0}^\ell\|_{L^2(\Gamma(0))} + h\|\gradg(u_0 - U_{h,0}^\ell)\|_{L^2(\Gamma(0))}\leq Ch^{k+1} \|u_0\|_{H^{k+1}(\Gamma(0))}.
	\label{discinitialdata}
\end{align}
Examples of such initial data are the Lagrange interpolant and the Ritz projection of $u_0$.\\

This backward Euler time discretisation is analogous to the second order linear problem studied in \cite{dziuk2012fully}, and indeed we have written \eqref{fdiscfecheqn1} in this way to be more in line with how their numerical scheme is posed. 
However, \eqref{fdiscfecheqn1} is in fact independent of choice of $\phi_h^{n-1}$, and noting that $\underline{\phi_h^{n}} = \phi_h^{n-1} - \tau \underline{\matdevtau \phi_h^n}$, we may rewrite \eqref{fdiscfecheqn1} as
\begin{align}
	\frac{1}{\tau} \left(m_h(t_n; U_h^n, \phi_h^n) - m_h(t_{n-1};U_h^{n-1}, \underline{\phi_h^n}) \right) + a_h(t_n;W_h^n,\phi_h^n) = 0 \label{fdiscfecheqn3}.
\end{align}
This form is more natural to the problem and will be used throughout.

\subsection{Well-posedness}
We begin this subsection by showing uniqueness, given some restriction on the timestep size.
\begin{lemma}
	\label{implicit uniqueness}
	Let $\tau < \frac{ 4 \varepsilon^3}{\theta^2}$, and $U_{h,0} \in S_h^0$.
	Then if the system \eqref{fdiscfecheqn1}, \eqref{fdiscfecheqn2} admits a solution pair $(U_h^n, W_h^n)$, it is unique.
\end{lemma}
\begin{proof}
	Assume there are two solution pairs $(U_{h,i}^{n},W_{h,i}^{n})$ ($i = 1,2$) of \eqref{fdiscfecheqn1},\eqref{fdiscfecheqn2}.
	We define the following differences
	\begin{gather*}
		\widehat{U}_h^n = U_{h,1}^{n} - U_{h,2}^{n},\\
		\widehat{W}_h^n = W_{h,1}^{n} - W_{h,2}^{n},
	\end{gather*}
	which solve
	\begin{gather}
		\label{fdiscunique1}
		m_h(\widehat{U}_h^n, \phi_h^n) + \tau a_h(\widehat{W}_h^n , \phi_h^n) = 0,\\
		m_h(\widehat{W}_h^n, \phi_h^n) = \varepsilon a_h(\widehat{U}_h^n , \phi_h^n) + \frac{1}{\varepsilon} m_h(F'(U_{h,1}^n) - F'(U_{h,2}^{n}),\phi_h^n),	\label{fdiscunique2}
	\end{gather}
	for all $\phi_h^n \in S_h^n$.
	Testing \eqref{fdiscunique1} with $\widehat{W}_h^n$ and \eqref{fdiscunique2} with $\widehat{U}_h^n$ and subtracting the resulting equations yields
	$$ \tau \| \gradgh \widehat{W}_h^n \|_{L^2(\Gamma_h^n)}^2 + \varepsilon \| \gradgh \widehat{U}_h^n \|_{L^2(\Gamma_h^n)}^2 = -  \frac{1}{\varepsilon} m_h(F'(U_{h,1}^{n}) - F'(U_{h,2}^{n}),\widehat{U}_h^n).$$
	Convexity of $F_1$ implies $F_1'$ is monotonic and hence
	$$ m_h(F'(U_{h,1}^{n}) - F'(U_{h,2}^{n}),\widehat{U}_h^n) \geq -\frac{\theta}{\varepsilon}\|\widehat{U}_h^n \|_{L^2(\Gamma_h^n)},$$
	which yields
	$$ \tau \| \gradgh \widehat{W}_h^n \|_{L^2(\Gamma_h^n)}^2 + \varepsilon \| \gradgh \widehat{U}_h^n \|_{L^2(\Gamma_h^n)}^2 \leq \frac{\theta}{\varepsilon}\|\widehat{U}_h^n\|_{L^2(\Gamma_h^n)}^2.$$
	Next we test \eqref{fdiscunique1} with $\widehat{U}_h^n$ and use Young's inequality to see
	$$ \frac{\theta}{\varepsilon}\|\widehat{U}_h^n\|_{L^2(\Gamma_h^n)}^2 = -\frac{\tau \theta}{\varepsilon}a_h(\widehat{W}_h^n, \widehat{U}_h^n) \leq \tau \|\gradgh \widehat{W}_h^n\|_{L^2(\Gamma_h^n)}^2 + \frac{\tau\theta^2}{4 \varepsilon^2}\|\gradgh \widehat{W}_h^n\|_{L^2(\Gamma_h^n)}^2$$
	Thus, using the assumption on $\tau$, we see
	$$ 0 \leq \left( \varepsilon - \frac{\tau \theta^2}{4 \varepsilon^2} \right) \| \gradgh \widehat{U}_h^n \|_{L^2(\Gamma_h)}^2 \leq 0 ,$$
	and hence $ \| \gradgh \widehat{U}_h^n \|_{L^2(\Gamma_h^n)} = 0 $.
	Observing that $\mval{\widehat{U}_h^n}{\Gamma_h^n} = 0$ we may use Poincar\'e's inequality to conclude $\widehat{U}_h^n = 0$.
	From this it is clear that $\widehat{W}_h^n = 0$ also. \qed
\end{proof}

\begin{remark}
	The condition on the timestep size $\tau < \frac{4 \varepsilon^3}{\theta^2}$ coincides with the uniqueness condition in \cite{copetti1992numerical}, where they have considered $\gamma = \varepsilon^2$, and a time scaling $t \mapsto \varepsilon t$.
    We emphasize that this is only a smallness condition on the timestep size, and not a Courant-Friedrichs-Lewy (CFL) condition.
\end{remark}

Next we recall some results from \cite{dziuk2012fully} about functions valued on different time discrete surfaces.

\begin{lemma}[\cite{dziuk2012fully}, Lemma 3.6]
	For $\phi_h^n \in S_h^n$, $\tau$ sufficiently small and $t \in [t_{n-1}, t_n]$ there exists a constant $C$ independent of $t, \tau, h$ such that
	\begin{align}
		\| \underline{\phi_h^n}(t) \|_{L^2(\Gamma_h(t))} &\leq C \| \phi_h^n \|_{L^2(\Gamma_h^n)}, \label{timenorm1}\\
		\| \gradgh \underline{\phi_h^n}(t) \|_{L^2(\Gamma_h(t))} &\leq C \| \gradgh \phi_h^n \|_{L^2(\Gamma_h^n)}, \label{timenorm2}
	\end{align}
	where $\underline{\phi_h^n}(t)$ is the function on $\Gamma_h(t)$ with the same nodal values as $\phi_h^n$.
	In this notation $\underline{\phi_h^n} = \underline{\phi_h^n}(t_{n-1})$.
\end{lemma}

\begin{lemma}
	For $\zeta_h^n, \eta_h^n \in S_h^n$ and sufficiently small $\tau$ we have
	\begin{align}
		|m_h(\zeta_h^n, \eta_h^n) - m_h(\underline{\zeta_h^n}, \underline{\eta_h^n})| &\leq C \tau \| \zeta_h^n \|_{L^2(\Gamma_h^n)} \| \eta_h^n \|_{L^2(\Gamma_h^n)}, \label{timeperturb1}\\
		|a_h(\zeta_h^n, \eta_h^n) - a_h(\underline{\zeta_h^n}, \underline{\eta_h^n})| &\leq  C \tau \|\gradgh \zeta_h^n \|_{L^2(\Gamma_h^n)} \|\gradgh \eta_h^n \|_{L^2(\Gamma_h^n)}, \label{timeperturb2}
	\end{align}
	where $C$ denotes a constant independent of $\tau, h$.
\end{lemma}
\begin{proof}
	\eqref{timeperturb1} is shown in \cite{dziuk2012fully} by combining Lemma 3.6 and Lemma 3.7, and the previous result.
	The proof for \eqref{timeperturb2} follows by the same logic. \qed
\end{proof}
\begin{corollary}
	Let $\tau$ be sufficiently small, then for $\phi_h^n \in S_h^n$ we have that
	\begin{align}
		\|\phi_h^n\|_{L^2(\Gamma_h^n)} &\leq C \|\underline{\phi_h^n}\|_{L^2(\Gamma_h^{n-1})}, \label{timenorm3}\\
		\|\gradgh \phi_h^n\|_{L^2(\Gamma_h^n)} &\leq C \|\gradgh \underline{\phi_h^n}\|_{L^2(\Gamma_h^{n-1})}, \label{timenorm4},
	\end{align}
	where $C$ denotes a constant independent of $\tau, h$.
\end{corollary}
\begin{proof}
	For \eqref{timenorm3} this follows by writing
	$$\|\phi_h^n\|_{L^2(\Gamma_h^n)}^2 = \|\underline{\phi_h^n}\|_{L^2(\Gamma_h^{n-1})}^2 + \left[ \|\phi_h^n\|_{L^2(\Gamma_h^n)}^2 - \|\underline{\phi_h^n}\|_{L^2(\Gamma_h^{n-1})}^2 \right],$$
	and bounding the term in square brackets by using \eqref{timeperturb1} with a sufficiently small $\tau$.
	\eqref{timenorm4} follows similarly. \qed
\end{proof}

Now we define a time projection to be used in showing existence of a solution.
\begin{definition}
	For $z_h^{n-1} \in S_h^{n-1}$ we define $z_{h,+}^{n-1} \in S_h^n$ to be the unique solution of
	\begin{align}
		m_h(t_n;z_{h,+}^{n-1}, \phi_h^n) = m_h(t_{n-1};z_{h}^{n-1}, \underline{\phi}_h^n), \label{timeproj}
	\end{align}
	for all $\phi_h^n \in S_h^n$.
\end{definition}

This is clearly well defined by the Lax-Milgram theorem.
This time projection has the following properties.
\begin{lemma}
	For $z_h^{n-1} \in S_h^{n-1}$ and $z_{h,+}^{n-1}$ as defined above we have:
	\begin{gather}
			\|z_{h,+}^{n-1}\|_{L^2(\Gamma_h^n)} \leq C \|z_h^{n-1}\|_{L^2(\Gamma_h^{n-1})}, \label{timeproj1}\\
			\|\overline{z_h^{n-1}} -  z_{h,+}^{n-1}\|_{L^2(\Gamma_h^n)} \leq C{\tau} \|z_h^{n-1}\|_{L^2(\Gamma_h^{n-1})}, \label{timeproj2}
		\end{gather}
	where $C$ denotes a constant independent of $\tau,h$.
\end{lemma}
\begin{proof}
	To show \eqref{timeproj1} we test \eqref{timeproj} with $z_{h,+}^{n-1}$ and use \eqref{timenorm1} to see that
	$$ \|z_{h,+}^{n-1}\|_{L^2(\Gamma_h^n)}^2 \leq \|\underline{z_{h,+}^{n-1}}\|_{L^2(\Gamma_h^{n-1})} \|z_h^{n-1}\|_{L^2(\Gamma_h^{n-1})} \leq C\|z_{h,+}^{n-1}\|_{L^2(\Gamma_h^n)} \|z_h^{n-1}\|_{L^2(\Gamma_h^{n-1})},$$
	from which the result clearly holds.
 To show \eqref{timeproj2} we note that
	$$ m_h(t_n;\overline{z_h^{n-1}} -  z_{h,+}^{n-1}, \phi_h^n) = m_h(t_n;\overline{z_h^{n-1}},\phi_h^n) - m_h(t_{n-1};z_h^{n-1}, \underline{\phi_h^n}),$$ and hence taking $\phi_h^n = \overline{z_h^{n-1}} -  z_{h,+}^{n-1}$ and using \eqref{timeperturb1} one obtains the result. \qed
\end{proof}
Similarly we require an appropriate notion of a discrete inverse Laplacian.
Given $z_h^n \in S_h^n$ such that $m_h(z_h^n,1)= 0$, we define the discrete inverse Laplacian to be the unique solution, $\invshn{n} z_h^n$, of
\[ a_h(\invshn{n} z_h^n, \phi_h^n) = m_h(z_h^n, \phi_h^n), \]
such that $m_h(\invshn{n} z_h^n, 1) = 0$.
From Poincar\'e's inequality it is clear to see that this is well defined, and moreover that
\[ \normshn{z_h^n}{n} := \| \gradgh \invshn{n} z_h^n \|_{L^2(\Gamma_h^n)} \leq C \| z_h^n \|_{L^2(\Gamma_h^n)}, \]
for a constant $C$ independent of $n, h$.\\

By using the projection \eqref{timeproj} and the discrete inverse Laplacian we find that \eqref{fdiscfecheqn1} becomes
$$ \frac{1}{\tau}a_h(\invshn{n}(U_h^n - U_{h,+}^{n-1}, \phi_h^n)) + a_h(W_h^n, \phi_h^n) = 0.$$
This implies that $\frac{1}{\tau} \invshn{n}(U_h^n - U_{h,+}^{n-1}) = \lambda^n - W_h^n$ for some constant $\lambda^n$ depending on $n$.
By using \eqref{fdiscfecheqn2} it is clear that
$$ \lambda^n = \frac{1}{\varepsilon} \mval{F_1'(U_h^n)}{\Gamma_h^n} - \frac{\theta}{\varepsilon}\mval{U_h^n}{\Gamma_h^n}.$$
Hence we find that \eqref{fdiscfecheqn1}, \eqref{fdiscfecheqn2} may be written as a single equation
\begin{multline}
	\label{discdecoupledeqn}
	\varepsilon a_h \left( U_h^{n}, \phi_h^n \right) + \frac{\theta}{2 \varepsilon}m_h\left(F_1'\left(U_h^n\right), \phi_h^n - \mval{\phi_h^n}{\Gamma_h^n}\right) - \frac{\theta}{\varepsilon}m_h\left(U_h^n, \phi_h^n - \mval{\phi_h^n}{\Gamma_h^n}\right)\\
	+ \frac{1}{\tau} m_h \left( \invshn{n}(U_h^n - U_{h,+}^{n-1}), \phi_h^n \right) = 0.
\end{multline}

Motivated by this observation, we define a functional, $J_h^n : D^n \rightarrow \mathbb{R}$, for
\[D^n := \left\{ z_h^n \in S_h^n \mid m_h(z_h^n,1) = m_h(U_h^{n-1},1) \right\}\]
given by
\[J_h^n (z_h^n) := \frac{1}{\varepsilon} m_h \left( F(z_h^n), 1 \right) + \frac{\varepsilon}{2} \|\gradgh z_h^n \|_{L^2(\Gamma_h)}^2 + \frac{1}{2 \tau} \normshn{z_h^n - U_{h,+}^{n-1}}{n}^2. \]
$D^n$ is clearly a non-empty, affine subspace of $S_h$, and thus finite dimensional.
Our existence result now follows by using similar arguments to Theorem 2.1 of \cite{copetti1992numerical}.

\begin{lemma}
	\label{fullyimp existence}
	Let $U_{h,0} \in S_h^0$.
	Then there exists a solution pair $(U_h^{n},W_h^{n})$ solving \eqref{fdiscfecheqn1},\eqref{fdiscfecheqn2}.
\end{lemma}
\begin{proof}
	We firstly note that for $z_h^n \in D^n$,
	$$ J_h^n(z_h^n) \geq \beta |\Gamma_h^n| , $$
	where $\beta$ is the lower bound on $F$.
	Hence $\inf_{z_h^n \in D^n} J_h^{n}(z_h^n) > -\infty$, and we can consider a minimising sequence $(z_{h,m}^n)_m \subset D^n$.
	It is straightforward to see such a sequence is bounded in $H^1(\Gamma_h^n)$.
    Since $D^n$ is finite dimensional there exists a subsequence of $(z_{h,m}^n)_m$ (for which we still keep the subscript $m$) converging to a minimiser, $U^n \in D^n$, of $J_h^{n}$. 
	By continuity of $J_h^{n}$ we have that
	$$ \lim_{m \rightarrow \infty} J_h^{n}(z_{h,m}^n) = J_h^{n}(U^n).$$
	Moreover, by considering the Euler-Lagrange equation associated to this minimisation problem, we see that $U^n$ solves, for all $\phi_h^n \in S_h^n$,
\begin{multline*}
    \varepsilon a_h(U^n, \phi_h^n) + \frac{1}{\varepsilon}m_h(F_1'(U^n), \phi_h^n) - \frac{\theta}{\varepsilon} m_h(U^n, \phi_h^n) + \frac{1}{\tau} m_h(\invshn{n} (U^n - U_{h,+}^{n-1}), \phi_h^n)\\
    - \lambda^n m_h(1, \phi_h^n) =0,
\end{multline*}
    where
    \[\lambda^n = \frac{1}{\varepsilon} \mval{F_1'(U^n)}{\Gamma_h^n} - \frac{\theta}{\varepsilon}\mval{U^n}{\Gamma_h^n}\]
	is a Lagrange multiplier associated to the constraint defining $D^n$.
	This is precisely the decoupled equation \eqref{discdecoupledeqn}, and hence setting
	$$U_h^{n} := U^n, \quad W_h^n := \lambda^n - \frac{1}{\tau} \invshn{n}(U^n - U_{h,+}^{n-1})$$ we obtain a solution pair. \qed
\end{proof}

\subsection{Stability}
We now discuss difficulties in showing the scheme \eqref{fdiscfecheqn1}, \eqref{fdiscfecheqn2} is numerically stable.
As is standard in the analysis of the Cahn-Hilliard equation this is expressed in terms of the (discrete) Ginzburg-Landau functional
$$ \Echh[U_h;t] = \int_{\Gamma_h(t)} \frac{\varepsilon}{2} |\gradgh U_h|^2 + \frac{1}{\varepsilon} F(U_h).$$
We will omit the $t$ argument, as it will be clear from context.

\begin{lemma}
	\label{implicitstability}
	Let $u_0 \in H^1(\Gamma_0)$, $U_{h,0} \in S_h^0$ be such that \eqref{discinitialdata} holds,
	and $(U_h^n, W_h^n) \in S_h^n \times S_h^n$ be the unique pair solving \eqref{fdiscfecheqn1}, \eqref{fdiscfecheqn2}.
	Then for $1 \leq N \leq \lfloor \frac{T}{\tau} \rfloor$,
	\begin{multline}
		\widetilde{\Echh}[U_h^N] + \frac{\tau}{2} \sum_{n=1}^N \| \gradgh W_h^n \|_{L^2(\Gamma_h^n)}^2 + \frac{\tau^2}{2}  \sum_{n=1}^N \| \gradgh \matdevtau U_h^n \|_{L^2(\Gamma_h^n)}^2 \leq CT\\
  +C \tau \widetilde{\Echh}[U_{h,0}]+ C\tau  \sum_{n=1}^{N} \left( \widetilde{\Echh}[U_h^{n}] + \left(\widetilde{\Echh}[U_h^n]\right)^q \right)
		\label{stabilitybound}
	\end{multline}
	where $C$ denotes a constant independent of $h,\tau$ but may depend on $\varepsilon, \theta, V,T, u_0$.
\end{lemma}
\begin{proof}
	Throughout we define $\widetilde{\Echh} = \Echh + C$, where $C$ is sufficiently large so that 
	$$  \widetilde{\Echh}[z_h^n] \geq \frac{\varepsilon}{2}\|\gradgh z_h^n \|_{L^2(\Gamma_h^n)}^2 \geq 0,$$
	for all $z_h^n \in S_h^n$.
	This is possible by the assumptions on $F$, and moreover we can choose $C$ independent of $h, \tau$ by the assumptions on $V$.
	We write the scheme \eqref{fdiscfecheqn1}, \eqref{fdiscfecheqn2} in matrix form as
	\begin{gather}
		M^n \frac{(U^n - U^{n-1})}{\tau} + \frac{(M^n - M^{n-1})}{\tau} U^{n-1} + A^n W^n = 0,\label{matrixeqn1}\\
		M^n W^n = \varepsilon A^n U^n + \frac{1}{\varepsilon}\mathcal{F}_1^{n}(U^n) - \frac{\theta}{\varepsilon}M^n U^n, \label{matrixeqn2}
	\end{gather}
	where $U^n, W^n$ denote the vector of nodal values of $U_h^n, W_h^n$ respectively, and
	\begin{align*}
		(M^n)_{ij} &= m_h(t_n; \phi_i^n, \phi_j^n),\\
		(A^n)_{ij} &= a_h(t_n; \phi_i^n, \phi_j^n),\\
		\mathcal{F}_1^{n}(U^n)_j &= m_h(t_n; F_1'(U_h^n), \phi_j^n),
	\end{align*}
	for $\phi_i^n$ the `$i$'th nodal basis function in $S_h^n$.
        We note that the only nonlinear term here is due to the convex part of the potential, $F_1$, as we have assumed that the concave part of the potential is given by a quadratic.\\
        
	We dot \eqref{matrixeqn1} with $W^n$ and \eqref{matrixeqn2} with $\frac{1}{\tau}(U^n - U^{n-1})$.
	This yields
	\begin{multline}
		U^n \cdot A^n \frac{(U^n - U^{n-1})}{\tau} + \mathcal{F}_1^n(U^n)\cdot \frac{(U^n - U^{n-1})}{\tau\varepsilon} + W^n \cdot A^n W^n\\
  = \frac{\theta}{\tau\varepsilon} U^n \cdot M^n(U^n-U^{n-1})- W^n \cdot \frac{(M^n - M^{n-1})}{\tau} U^{n-1}. \label{discstab1}
	\end{multline}
	It is then straightforward to show that
	\begin{align}
    \begin{aligned}
		U^n \cdot A^n \frac{(U^n - U^{n-1})}{\tau} &= \frac{1}{2 \tau} \left[ U^n \cdot A^{n} U^n - U^{n-1} \cdot A^{n-1} U^{n-1}\right]\\
  &+\frac{1}{2\tau} \left(U^n - U^{n-1}\right) \cdot A^n \left(U^n - U^{n-1}\right)\\
		&+ \frac{1}{2\tau} \left[ U^{n-1} \cdot A^{n-1} U^{n-1} - U^{n-1} \cdot A^n U^{n-1}  \right]
  \end{aligned}\label{discstab2}
	\end{align}
	This can be understood as a fully discrete form of the transport theorem, see \cite{dziuk2012fully} Lemma 3.5.
	By using the convexity of $F_1$, we see that
	$$ \mathcal{F}_1^n(U^n)\cdot (U^n - U^{n-1}) = m_h(F_1'(U_h^n), U_h^n - \overline{U_h^{n-1}}) \geq m_h(F_1(U_h^n) - F_1(\overline{U_h^{n-1}}), 1).$$
	We then rewrite this final term as
	\begin{multline*}
	    m_h(F_1(U_h^n) - F_1(\overline{U_h^{n-1}}), 1) = m_h(F_1(U_h^n), 1) - m_h(F_1(U_h^{n-1}), 1)\\
     + \left[m_h(F_1(U_h^{n-1}), 1) - m_h(F_1(\overline{U_h^{n-1}}), 1) \right],
	\end{multline*}
	where the term in square brackets may be written as
    \begin{align*}
        m_h(F_1(U_h^{n-1}), 1) - m_h(F_1(\overline{U_h^{n-1}}), 1) &= -\int_{t_{n-1}}^{t_n} \frac{d}{dt} m_h(F_1(\overline{U_h^{n-1}}(t)),1)\\
        &= -\int_{t_{n-1}}^{t_n} g_h(F_1(\overline{U_h^{n-1}}(t)),1) .
    \end{align*}
	Hence by using the smoothness of $V_h$ and arguing along the lines of \eqref{timenorm1} one obtains that
	\begin{align}
		\frac{1}{\tau\varepsilon}\left| m_h(F_1(U_h^{n-1}), 1) - m_h(F_1(\overline{U_h^{n-1}}), 1) \right| \leq C \widetilde{\Echh}[U_h^{n-1}], \label{discstab3}
	\end{align}
	for some constant $C$ independent of $\tau, h$.\\
	
	Next we notice that one can write
	\begin{align}
 \begin{aligned}
		\frac{\theta}{\tau\varepsilon} U^n \cdot M^n(U^n-U^{n-1}) &= \frac{\theta}{2 \tau \varepsilon} \left( U^n \cdot M^n U^n - U^{n-1} \cdot M^{n-1} U^{n-1} \right)\\
  &+ \frac{\theta}{2 \tau \varepsilon}(U^n - U^{n-1}) \cdot M^n (U^n - U^{n-1})\\
&+ \frac{\theta}{2 \tau \varepsilon} \left[U^{n-1} \cdot M^{n-1} U^{n-1} - U^{n-1} \cdot M^{n} U^{n-1} \right],
  \end{aligned}\label{discstab4}
	\end{align}
	where we note that the first term on the right hand side corresponds to the quadratic part of the potential.
	One dots \eqref{matrixeqn1} with $\frac{\theta}{2\varepsilon}(U^{n} - U^{n-1})$ to see that
    \begin{multline*}
        \frac{\theta}{2 \tau \varepsilon}(U^n - U^{n-1}) \cdot M^n (U^n - U^{n-1}) = (U^n-U^{n-1})\cdot\frac{\theta(M^n - M^{n-1})}{2\varepsilon\tau} U^{n-1}\\
        + \frac{\theta}{2\varepsilon}(U^n-U^{n-1}) \cdot A^n W^n,
    \end{multline*}
	and by using \eqref{timeperturb1}, \eqref{timenorm3}, \eqref{timenorm4} and a Young's inequality this yields
	\begin{align}
		\begin{split}
			\left| \frac{\theta}{2 \tau \varepsilon}(U^n - U^{n-1}) \cdot M^n (U^n - U^{n-1}) \right| &\leq \frac{C \theta}{\varepsilon} \left( \|U_h^{n-1}\|_{L^2(\Gamma_h^{n-1})}^2 + \|U_h^{n}\|_{L^2(\Gamma_h^{n})}^2\right)\\
			&+ \frac{C\theta^2 }{\varepsilon^2}\left(\|\gradgh U_h^{n-1}\|_{L^2(\Gamma_{h}^{n-1})}^2 + \|\gradgh U_h^n\|_{L^2(\Gamma_{h}^n)}^2\right)\\
			&+\frac{1}{4} \|\gradgh W_h^{n}\|_{L^2(\Gamma_h^{n})}^2, \label{discstab5}
		\end{split}
	\end{align}
	where $C$ is independent of $\varepsilon, \theta$.\\
	
	Similarly, by using \eqref{timeperturb1}, \eqref{timenorm3} we find that
	\begin{align}
		\left| W^n \cdot \frac{(M^n - M^{n-1})}{\tau} U^{n-1} \right| \leq C\|W_h^n\|_{L^2(\Gamma_h^n)}\|U_h^{n-1}\|_{L^2(\Gamma_h^{n-1})}.
	\end{align}
	Now by testing \eqref{fdiscfecheqn2} with $W_h^n$ and using Young's inequality we find that
	\begin{multline}
		\| W_h^n \|_{L^2(\Gamma_h^n)}^2 \leq \frac{1}{2}\| W_h^n \|_{L^2(\Gamma_h^n)}^2 + \frac{1}{8}\| \gradgh W_h^n \|_{L^2(\Gamma_h^n)}^2 + C\|F_1'(U_h^n)\|_{L^2(\Gamma_h^n)}^2\\
		+ C\|U_h^n\|_{L^2(\Gamma_h^n)}^2 + C \|\gradgh U_h^n\|_{L^2(\Gamma_h^n)}^2, \label{implicitwbound}
	\end{multline}
	where the constants $C$ may depend on $\varepsilon, \theta$.
	Next we recall the condition on the potential, 
	$$ |F_1'(r)| \leq \alpha |r|^q + \alpha,$$
	for constants $\alpha \in \mathbb{R}^+$, hence
	$$ \|F_1'(U_h^n)\|_{L^2(\Gamma_h^n)}^2 \leq C|\Gamma_h^n| + C\|U_h^n\|_{L^{2q}(\Gamma_h^n)}^{2q} \leq C|\Gamma_h^n| + C\|U_h^n\|_{H^1(\Gamma_h^n)}^{2q}, $$
	where we have used the Sobolev embedding $H^1(\Gamma_h^n) \hookrightarrow L^{2q}(\Gamma_h^n)$ (and lifts if necessary so that the constants are independent of $h$).
	By using lifts and the assumptions on $V$ we can then bound $|\Gamma_h^n| \leq C$ independent of $h,\tau$, and we can use Poincar\'e's inequality to see that
	\[\|U_h^n\|_{H^1(\Gamma_h^n)} \leq C \left( \left|\mval{U_{h,0}}{\Gamma_h^0}\right| + \|\gradgh U_h^n\|_{L^2(\Gamma_h^n)}\right).\]
	Using lifts and \eqref{discinitialdata} it is straightforward to show that $|\mval{U_{h,0}}{\Gamma_h^0}| \leq C$ for a constant $C$ depending on $u_0$ but independent of $h$.
	Similarly we note that by construction,
	$$ \|\gradgh U_h^n\|_{L^2(\Gamma_h^n)}^2 \leq C \widetilde{\Echh}[U_h^n].$$
	It then follows that
	\begin{align}
		\| W_h^n \|_{L^2(\Gamma_h^n)}^2 \leq C + \frac{1}{4}\| \gradgh W_h^n \|_{L^2(\Gamma_h^n)}^2 + C\|\gradgh U_h^n\|_{L^2(\Gamma_h^n)}^2 + C \left(\widetilde{\Echh}[U_h^n]\right)^q, \label{discstab6}
	\end{align}
	where we have again used Poincar\'e's inequality.
	Finally, combining \eqref{discstab1}, \eqref{discstab2}, \eqref{discstab3}, \eqref{discstab4}, \eqref{discstab5}, \eqref{discstab6}, using \eqref{timeperturb1}, \eqref{timeperturb2} and Poincar\'e's inequality where necessary we find that, after multiplying by $\tau$,
	\begin{multline}
		\widetilde{\Echh}[U_h^n] - \widetilde{\Echh}[U_h^{n-1}] + \frac{\tau}{2} \| \gradgh W_h^n \|_{L^2(\Gamma_h^n)}^2 + \frac{\tau}{2} \| \gradgh \matdevtau U_h^n \|_{L^2(\Gamma_h^n)}^2 \leq C\tau\\
		+ C\tau \left( \widetilde{\Echh}[U_h^{n-1}] + \widetilde{\Echh}[U_h^{n}] \right) + C \tau \left(\widetilde{\Echh}[U_h^n]\right)^q, \label{discstab7}
	\end{multline}
	where the constants $C$ may depend on $\varepsilon, \theta$.
	Now summing \eqref{discstab7} over $n = 1,...,N \leq \lfloor \frac{T}{\tau} \rfloor$ we conclude that \eqref{stabilitybound} holds. \qed
\end{proof}

This proof mirrors that of \cite{caetano2021cahn} Proposition 4.5, and they are in a position where they can use a generalised Gr\"onwall inequality to obtain a bound.
In our case we cannot do something analogous, as both sides of the inequality \eqref{stabilitybound} depend on the final time $N$, whereas the usual discrete analogues of the Gr\"onwall inequality (for example the Bihari-type inequalities in \cite{wu2013nonlinear}) only go up to time $N-1$ on the righthand side.
In fact the inequality we would require likely does not hold, as the above inequality remains true for arbitrarily large $\widetilde{\Echh}[U_h^N]$.\\

If $q=1$ then we can take $\tau$ sufficiently small so that we can apply the usual discrete Gr\"onwall inequality to \eqref{stabilitybound} to obtain a bound in terms of $\widetilde{\Echh}[U_{h,0}]$ and $T$.
Moreover, it is clear that $\widetilde{\Echh}[U_{h,0}]$ can be bounded independent of $h$ by \eqref{discinitialdata}.
Stability is not surprising in this case as we are considering a nonlinearity with linear growth, which is then similar to the linear case.
As such we firstly consider the error analysis for a quadratic potential, $q=1$, and return to the general setting, $q > 1$, later.
Lastly, we note that upon inspection of this proof it is clear that this issue arises solely because of the evolution of the domain, and for a stationary surface one can neglect these considerations, as $M^n = M^{n-1}$.\\

Regardless of this stability issue for general $q$, we end this section with a result on the dependence on initial data.
For this we introduce a new time projection, similar to \eqref{timeproj}.

\begin{definition}
	Given $z_h^{n-1} \in S_h^{n-1}$ we define $z_{h,++}^{n-1} \in S_h^n$ to be the unique solution of
	\begin{align}
		\begin{split}
			a_h(t_n; z_{h,++}^{n-1}, \phi_h^n) &= a_h(t_{n-1};z_{h}^{n-1}, \underline{\phi_h^n}),\\
			\int_{\Gamma_h^n} z_{h,++}^{n-1} &= \int_{\Gamma_h^{n-1}}z_{h}^{n-1},
		\end{split}
		\label{timeproj3}
	\end{align}
	for all $\phi_h^n \in S_h^n$.
\end{definition}
One can show that
\begin{align}
	\|\gradgh z_{h,++}^{n-1}\|_{L^2(\Gamma_h^{n})} &\leq C \|\gradgh z_{h}^{n-1}\|_{L^2(\Gamma_h^{n-1})},\label{timeproj4}\\
	\| \gradgh(z_{h,++}^{n-1} - \overline{z_{h}^{n-1}}) \|_{L^2(\Gamma_h^n)} &\leq C \tau \|\gradgh z_{h}^{n-1}\|_{L^2(\Gamma_h^{n-1})} \label{timeproj5},
\end{align}
by similar means as the proofs of \eqref{timeproj1}, \eqref{timeproj2}.
We do not do these calculations here.
The importance of this projection comes from the observation that
\[ \invshn{n} z_{h,+}^{n-1} = (\invshn{n-1} z_h^{n-1})_{++}. \]
To see this we compute
\begin{align*}
	a_h(t_n;(\invshn{n-1} z_h^{n-1})_{++}, \phi_h^n) &= a_h(t_{n-1};\invshn{n-1} z_h^{n-1}, \underline{\phi_h^n}) = m_h(t_{n-1};z_h^{n-1}, \underline{\phi_h^n})\\
 &= m_h(t_n;z_{h,+}^{n-1}, \phi_h^n) = a_h(t_{n};\invshn{n} z_{h,+}^{n-1}, \phi_h^n),
\end{align*}
and use the fact that
\[ \mval{(\invshn{n-1} z_h^{n-1})_{++}}{\Gamma_h^n} = 0 =  \mval{\invshn{n} z_{h,+}^{n-1}}{\Gamma_h^n}. \]
With this observation, we now show continuous dependence on the initial data.

\begin{proposition}
	\label{continuous dependence}
	Let $U_{h,0}^1, U_{h,0}^2 \in S_h^0$ be such that $\mval{U_{h,0}^1}{\Gamma_h(0)} = \mval{U_{h,0}^2}{\Gamma_h(0)}$.
	Denoting the solution of \eqref{fdiscfecheqn1},\eqref{fdiscfecheqn2} with initial data $U_{h,0}^i$ by $(U_{h}^{n,i},W_{h}^{n,i})$ we have that for $1\leq N \leq N_T$,
	\begin{multline}
		\label{smoothstability}
		\normshn{U_h^{N,1} -U_h^{N,2} }{N}^2 + \frac{\varepsilon^4 \tau}{\varepsilon^3 - \tau \theta^2} \sum_{n=1}^N  \| \gradgh(U_h^{N,1} -U_h^{N,2}) \|_{L^2(\Gamma_h^n)}^2\\
  \leq e^{C t_N} \left(\frac{ \varepsilon^3}{ \varepsilon^3 - \tau \theta^2}\right) \normshn{U_{h,0}^1 - U_{h,0}^2}{0}^2,
	\end{multline}
	provided that $\tau < \frac{ \varepsilon^3}{\theta^2}$.
	
\end{proposition}
\begin{proof}
	Firstly we define some shorthand notation
	\begin{gather*}
		\widehat{U}_h^n = U_{h,1}^n - U_{h,2}^n,\\
		\widehat{W}_h^n = W_{h,1}^n - W_{h,2}^n,
	\end{gather*}
	and note that from \eqref{fdiscfecheqn1}, \eqref{fdiscfecheqn2}
	\begin{gather}
		\frac{1}{\tau} \left( m_h(\widehat{U}_h^n, \phi_h^n) - m_h(\widehat{U}_h^{n-1}, \underline{\phi_h^n}) \right) + a_h(\widehat{W}_h^n,\phi_h^n) = 0, \label{smoothstabeqn1}\\
		m_h(\widehat{W}_h^n,\phi_h^n) = \varepsilon a_h(\widehat{U}_h^n, \phi_h^n) + \frac{1}{\varepsilon} m_h(F_1'(U_{h,1}^n) - F_1'(U_{h,2}^n), \phi_h^n) - \frac{\theta}{\varepsilon} m_h(\widehat{U}_h^n, \phi_h^n), \label{smoothstabeqn2}
	\end{gather}
	holds for all $\phi_h^n \in S_h^n$.\\
	
	We test \eqref{smoothstabeqn1} with $\invshn{n} \widehat{U}_h^n$ for
	\[ m_h(\widehat{U}_h^n, \invshn{n} \widehat{U}_h^n) - m_h(\widehat{U}_h^{n-1}, \underline{\invshn{n} \widehat{U}_h^n}) + \tau a_h(\widehat{W}_h^n,\invshn{n} \widehat{U}_h^n) = 0, \]
	and we note that
    \begin{multline*}
        a_h(\widehat{W}_h^n,\invshn{n} \widehat{U}_h^n) = m_h(\widehat{W}_h^n, \widehat{U}_h^n) = \varepsilon a_h(\widehat{U}_h^n, \widehat{U}_h^n) + \frac{1}{\varepsilon} m_h(F_1'(U_{h,1}^n) - F_1'(U_{h,2}^n), \widehat{U}_h^n)\\
        - \frac{\theta}{\varepsilon} m_h(\widehat{U}_h^n, \widehat{U}_h^n),
    \end{multline*}
	and so one finds that
	\begin{multline}
		m_h(\widehat{U}_h^n, \invshn{n} \widehat{U}_h^n) - m_h(\widehat{U}_h^{n-1}, \underline{\invshn{n} \widehat{U}_h^n}) + \varepsilon \tau a_h(\widehat{U}_h^n, \widehat{U}_h^n) + \frac{\tau}{\varepsilon} m_h(F_1'(U_{h,1}^n) - F_1'(U_{h,2}^n), \widehat{U}_h^n)\\
		= \frac{\tau\theta}{\varepsilon} m_h(\widehat{U}_h^n, \widehat{U}_h^n). \label{stability1}
	\end{multline}
	From the convexity of $F_1(\cdot)$ one has that $m_h(F_1'(U_{h,1}^n) - F_1'(U_{h,2}^n), \widehat{U}_h^n) \geq 0$.
	The difficulty now is in rewriting the term on the previous surface.
	For this we use \eqref{timeproj} so that
    \begin{align}
        m_h(\widehat{U}_h^n, \invshn{n} \widehat{U}_h^n) - m_h(\widehat{U}_h^{n-1}, \underline{\invshn{n} \widehat{U}_h^n}) &= m_h(\widehat{U}_h^n - \widehat{U}_{h,+}^{n-1}, \invshn{n} \widehat{U}_h^n) \notag\\
        &= a_h(\invshn{n}(\widehat{U}_h^n - \widehat{U}_{h,+}^{n-1}), \invshn{n} \widehat{U}_h^n), \label{invlap difference1}
    \end{align}
	where we notice that
	\begin{align}
    \invshn{n}(\widehat{U}_h^n - \widehat{U}_{h,+}^{n-1}) = \invshn{n}\widehat{U}_h^n - \invshn{n}\widehat{U}_{h,+}^{n-1}. \label{invlap difference2} \end{align}
	Likewise it is a straightforward calculation to verify that
	\[ a_h(\invshn{n}(\widehat{U}_h^n - \widehat{U}_{h,+}^{n-1}), \invshn{n} \widehat{U}_h^n) = \frac{1}{2} \left( \normshn{\widehat{U}_h^n}{n}^2 - \normshn{\widehat{U}_{h,+}^{n-1}}{n}^2 + \normshn{\widehat{U}_h^n - \widehat{U}_{h,+}^{n-1}}{n}^2 \right), \]
	and so \eqref{stability1} becomes
	\begin{multline}
		\normshn{\widehat{U}_h^n}{n}^2 - \normshn{\widehat{U}_{h}^{n-1}}{n-1}^2+ \normshn{\widehat{U}_h^n - \widehat{U}_{h,+}^{n-1}}{n}^2 + 2\varepsilon\tau \|\gradgh \widehat{U}_h^n \|_{L^2(\Gamma_h^n)}^2 \leq \frac{2 \tau \theta}{\varepsilon} \|\widehat{U}_h^n \|_{L^2(\Gamma_h^n)}^2\\
		+ \left( \normshn{\widehat{U}_{h,+}^{n-1}}{n}^2- \normshn{\widehat{U}_{h}^{n-1}}{n-1}^2 \right). \label{stability2}
	\end{multline}
	
	We now show a bound for this final term.
	To begin we calculate
	\begin{align*}
		\normshn{\widehat{U}_{h,+}^{n-1}}{n}^2- \normshn{\widehat{U}_{h}^{n-1}}{n-1}^2 &= m_h(\widehat{U}_{h,+}^{n-1}, \invshn{n}\widehat{U}_{h,+}^{n-1}) - m_h(\widehat{U}_{h}^{n-1}, \invshn{n-1}\widehat{U}_{h}^{n-1}) \\
		&= m_h(\widehat{U}_{h,+}^{n-1}, (\invshn{n-1}\widehat{U}_{h}^{n-1})_{++}) - m_h(\widehat{U}_{h}^{n-1}, \invshn{n-1}\widehat{U}_{h}^{n-1}) \\
		&= m_h(\widehat{U}_{h}^{n-1},\underline{(\invshn{n-1}\widehat{U}_{h}^{n-1})_{++}} -  \invshn{n-1}\widehat{U}_{h}^{n-1})\\
		&= a_h(\invshn{n-1}\widehat{U}_{h}^{n-1},\underline{(\invshn{n-1}\widehat{U}_{h}^{n-1})_{++}} -  \invshn{n-1}\widehat{U}_{h}^{n-1}),
	\end{align*}
	from which \eqref{timeproj5} yields that
	\begin{align}
		\left|\normshn{\widehat{U}_{h,+}^{n-1}}{n}^2- \normshn{\widehat{U}_{h}^{n-1}}{n-1}^2\right| \leq C \tau \normshn{\widehat{U}_{h}^{n-1}}{n-1}^2. \label{stability3}
	\end{align}
	Lastly we use the definition of $\invshn{n}$ and Young's inequality to see that
	\begin{align}
		\frac{2 \tau \theta}{\varepsilon} \|\widehat{U}_h^n \|_{L^2(\Gamma_h^n)}^2 = \frac{2 \tau \theta}{\varepsilon} a_h(\invshn{n}\widehat{U}_h^n,\widehat{U}_h^n) \leq  \varepsilon \tau \|\gradgh \widehat{U}_h^n\|_{L^2(\Gamma_h^n)}^2 + \frac{\tau \theta^2}{ \varepsilon^3}\normshn{\widehat{U}_h^n}{n}^2. \label{stability4}
	\end{align}
	Now using \eqref{stability3}, \eqref{stability4} in \eqref{stability2} and summing over $n$ one obtains
	\begin{multline*}
		\normshn{\widehat{U}_h^N}{N}^2 - \normshn{\widehat{U}_h^0}{0}^2 + \varepsilon \tau \sum_{n=1}^N \|\gradgh \widehat{U}_h^n\|_{L^2(\Gamma_h^n)}^2 + \sum_{n=1}^N \normshn{\widehat{U}_h^n - \widehat{U}_{h,+}^{n-1}}{n}^2\\
  \leq \frac{\tau \theta^2}{ \varepsilon^3}\normshn{\widehat{U}_h^N}{N}^2+ C \tau \sum_{n=0}^{N-1} \normshn{\widehat{U}_h^n}{n}^2.
	\end{multline*}
	The result then follows from our assumption that $\tau < \frac{ \varepsilon^3}{\theta^2}$ and a discrete Gr\"onwall inequality. \qed
\end{proof}

\section{Error analysis}
\label{section:error}
Throughout, we assume that the initial data $u_0 \in H^{k+1}(\Gamma(0))$
 and  the following bounds on the solution of the continuous problem \eqref{cheqn1}, \eqref{cheqn2}
\begin{multline}
	\label{smoothassumptions1}
	\sup_{t \in [0,T]} \left( \|u\|_{H^{k+1}(\Gamma(t))} +  \|w\|_{H^{k+1}(\Gamma(t))}\right)\\
 + \int_0^T \|\matdev u\|_{H^{k+1}(\Gamma(t))}^2 + \|(\matdev)^2 u\|_{L^2(\Gamma(t))}^2 \, dt  < \infty.
\end{multline}
Some of these bounds are established in \cite{caetano2021cahn} for $k=1$.
We note that establishing these bounds for $k > 1$ will require $\Gamma(t)$ to be sufficiently smooth (at least $C^{k+1}$) as well as smoothness and polynomial growth assumptions on the potential $F$.
These bounds imply that \eqref{cheqn1}, \eqref{cheqn2} holds for all $t \in [0,T]$ and as such we can consider them at $t = t_n$.
We recall that by our assumptions that the velocity field $V$ is such that
\begin{align}
	\sup_{t \in [0,T]} \left( \|V\|_{L^{\infty}(\Gamma(t))} + \|\matdev V\|_{L^{\infty}(\Gamma(t))} \right) < \infty. \label{smoothassumptions2}
\end{align}

As in \cite{dziuk2012fully,elliott2015evolving} we look at the splitting
\begin{align*}
	(u^n)^{-\ell} - U_h^n &= \underbrace{(u^n)^{-\ell} - \Pi_h u^n}_{=: \rho_u^n} + \underbrace{\Pi_h u^n - U_h^n}_{=: \sigma_u^n},\\
	(w^n)^{-\ell} - W_h^n &= \underbrace{(w^n)^{-\ell} - \Pi_h w^n}_{=: \rho_w^n} + \underbrace{\Pi_h w^n - U_h^n}_{=: \sigma_w^n},
\end{align*}
where $u^n = u(t_n), w^n = w(t_n).$
Note that \eqref{smoothassumptions1} implies that $\Pi_h u^n, \Pi_h w^n$ are defined for all $n$.
We recall by \eqref{ritz3} that $\rho_u^n, \rho_w^n$ are such that
\begin{multline*}
    \max_{n=1,...,N_T} \|\rho_u^n\|_{L^2(\Gamma_h^n)}^2 + \tau \sum_{n=1}^{N_T}  \|\rho_w^n\|_{L^2(\Gamma_h^n)}^2\\
    \leq Ch^{2(k+1)} \sup_{t \in [0,T]}\left( \|u\|_{H^{k+1}(\Gamma(t))}^2 + \|w\|_{H^{k+1}(\Gamma(t))}^2\right),
\end{multline*}
and hence it remains to show bounds on $\sigma_u^n$ and $\sigma_w^n$.\\

We begin by finding equations of the form \eqref{fdiscfecheqn1}, \eqref{fdiscfecheqn2} satisfied by $\sigma_u^n, \sigma_w^n$. They will involve consistency errors $E_i:S_h^n\rightarrow \mathbb R, i=1,2,3,4$ defined below.
Firstly, by the definition of the Ritz projection \eqref{ritz} we find
\begin{multline*}
	\frac{1}{\tau} \left( m_h(\Pi_h u^n, \phi_h^n) - m_h(\Pi_h u^{n-1}, \underline{\phi_h^n}) \right) + a_h(\Pi_h w^n , \phi_h^n)\\
	= \frac{1}{\tau} \left( m_h(\Pi_h u^n, \phi_h^n) - m_h(\Pi_h u^{n-1}, \underline{\phi_h^n}) \right) + a(w^n, (\phi_h^n)^\ell).
\end{multline*}
We then subtract \eqref{fdiscfecheqn1} and note that by \eqref{cheqn1}
$$ a(w^n, (\phi_h^n)^\ell) = -m(\matdev u^n, (\phi_h^n)^\ell) - g(u^n, (\phi_h^n)^\ell),$$
to obtain 
\begin{align}
	\frac{1}{\tau} \left( m_h(\sigma_u^n, \phi_h^n) - m_h(\sigma_u^{n-1}, \underline{\phi_h^n}) \right) + a_h(\sigma_w^n , \phi_h^n) = E_1(\phi_h^n),
	\label{smootherroreqn1}
\end{align}
where
\begin{align*}
	E_1(\phi_h^n) := \frac{1}{\tau} \left( m_h(\Pi_h u^n, \phi_h^n) - m_h(\Pi_h u^{n-1}, \underline{\phi_h^n}) \right) -m(\matdev u^n, (\phi_h^n)^\ell)- g(u^n, (\phi_h^n)^\ell).
\end{align*}
Similarly we observe that
\begin{multline*}
	m_h(\Pi_h w^n, \phi_h^n) - \varepsilon a_h(\Pi_h u^n, \phi_h^n) - \frac{1}{\varepsilon}m_h(F_1'(\Pi_h u^n), \phi_h^n) + \frac{\theta}{\varepsilon} m_h(\Pi_h u^n, \phi_h^n)
	\\ = m_h(\Pi_h w^n, \phi_h^n) - \varepsilon a(u^n, (\phi_h^n)^\ell) - \frac{1}{\varepsilon}m_h(F_1'(\Pi_h u^n), \phi_h^n) + \frac{\theta}{\varepsilon} m_h(\Pi_h u^n, \phi_h^n),
\end{multline*}
and hence by subtracting \eqref{fdiscfecheqn2}, and using \eqref{cheqn2} we find
\begin{multline}
	m_h(\sigma_w^n, \phi_h^n) - \varepsilon a_h(\sigma_u^n, \phi_h^n) - \frac{1}{\varepsilon}m_h(F_1'(\Pi_h u^n) - F_1'(U_h^n), \phi_h^n) + \frac{\theta}{\varepsilon} m_h(\sigma_u^n, \phi_h^n)\\
	= E_2(\phi_h^n) + E_3(\phi_h^n) + E_4(\phi_h^n),
	\label{smootherroreqn2}
\end{multline}
where
\begin{gather*}
	E_2(\phi_h^n) := m_h(\Pi_h w^n, \phi_h^n) - m(w^n, (\phi_h^n)^\ell),\\
	E_3(\phi_h^n) := \frac{1}{\varepsilon} \left[m(F_1'(u^n), (\phi_h^n)^\ell )  - m_h(F_1'(\Pi_h u^n), \phi_h^n)\right],\\
	E_4(\phi_h^n) := \frac{\theta}{\varepsilon} \left[ m_h(\Pi_h u^n, \phi_h^n) - m(u^n, (\phi_h^n)^\ell) \right].
\end{gather*}

\subsection{Quadratic potential}

By the assumptions on $F_1$ we find $F_1'$ is Lipschitz continuous.
We denote the Lipschitz constant by $L_F$, which will be important in the later analysis.

\begin{lemma}[\bf Consistency errors]
	\label{smootherrorlemma}
	For $E_1,...,E_4$ as defined above we have that
	\begin{align*}
		|E_1(\phi_h^n)| &\leq \frac{Ch^{k+1}}{\sqrt{\tau}} \left( \int_{t_{n-1}}^{t_n}  \|u\|_{H^{k+1}(\Gamma(t))}^2 + \|\matdev u \|_{H^{k+1}(\Gamma(t))}^2 \,dt \right)^{\frac{1}{2}}\|\phi_h^n \|_{L^2(\Gamma_h^n)}\\
	&+ C \sqrt{\tau} \left(\int_{t_{n-1}}^{t_n} \|\matdev u\|_{L^2(\Gamma(t))}^2 + \|(\matdev)^2 u\|_{L^2(\Gamma(t))}^2 \, dt \right)^{\frac{1}{2}}\|\phi_h^n \|_{L^2(\Gamma_h^n)}\\
 &+ Ch^{k+1}\|u^n\|_{L^2(\Gamma(t^n))}\| \phi_h^n\|_{H^1(\Gamma_h^n)},
	\end{align*}
	\begin{gather*}
		|E_2(\phi_h^n)| \leq  Ch^{k+1} \|w^n\|_{H^{k+1}(\Gamma(t_n))} \|\phi_h^n \|_{L^2(\Gamma_h^n)},\\
		|E_3(\phi_h^n)| \leq  \frac{CL_F h^{k+1}}{\varepsilon}\left( 1 + \|u^n\|_{H^{k+1}(\Gamma(t_n))} \right)\|\phi_h^n\|_{L^2(\Gamma_h^n)},\\
		|E_4(\phi_h^n)| \leq  \frac{C\theta h^{k+1}}{\varepsilon} \|u^n\|_{H^{k+1}(\Gamma(t_n))} \|\phi_h^n \|_{L^2(\Gamma_h^n)},
	\end{gather*}
	and $C$ denotes some constant independent of $h, \tau$.
\end{lemma}
\begin{proof}
	We firstly bound $E_2,E_3,E_4$ as they are the simplest.
	The bound for $E_2$ is straightforward, and we write
	\begin{multline*}
		|E_2(\phi_h^n)| \leq |m_h(\Pi_h w^n, \phi_h^n) - m(\pi_h w^n, (\phi_h^n)^\ell)| + |m(\pi_h w^n - w^n, (\phi_h^n)^\ell)|\\
		\leq Ch^{k+1} \|w^n\|_{H^{k+1}(\Gamma(t_n))} \|\phi_h^n \|_{L^2(\Gamma_h^n)},
	\end{multline*}
	where we have used \eqref{perturb1}, \eqref{ritz3} and the stability of the lift.
	The bound for $E_4$ follows identically, and the bound on $E_3$ is similar, where we also use the Lipschitz continuity of $F_1'$.
	We see that
	\begin{align*}
		|E_3(\phi_h^n)| &\leq \frac{1}{\varepsilon} \left| m_h(F_1'(\Pi_h u^n), \phi_h^n) -  m(F_1'(\pi_h u^n), (\phi_h^n)^\ell )\right| + \frac{1}{\varepsilon}\left|  m(F_1'(u^n) - F_1'(\pi_h u^n), (\phi_h^n)^\ell ) \right|\\
		&\leq \frac{CL_F h^{k+1}}{\varepsilon} \left( 1 + \|u^n\|_{H^{k+1}(\Gamma(t_n))} \right) \|\phi_h^n \|_{L^2(\Gamma_h^n)},
	\end{align*}
	where we have used the Lipschitz continuity of $F_1'$, \eqref{perturb1}, \eqref{ritz3}, and the stability of the lift.\\
	
	Lastly we bound $E_1$.
	We firstly notice that by comparing Proposition \ref{transport2} and Proposition \ref{transport4} that
	\begin{align}
		m(\matdev u^n, (\phi_h^n)^\ell) + g(u^n, (\phi_h^n)^\ell) =  m(\matdev_\ell u^n, (\phi_h^n)^\ell) + g_\ell(u^n, (\phi_h^n)^\ell) + m(u^n, (\matdev_\ell - \matdev)(\phi_h^n)^\ell). \label{smootherror1}
	\end{align}
	Using this, we write $E_1$ in an integral form as
	\begin{align*}
		E_1(\phi_h^n) &= \frac{1}{\tau} \int_{t_{n-1}}^{t_n} \frac{d}{dt}m_h(\Pi_h u, \underline{\phi_h^n}(t)) - m(\matdev_\ell u^n, (\phi_h^n)^\ell) - g_\ell(u^n, (\phi_h^n)^\ell) \, dt + m(u^n, (\matdev_\ell - \matdev)(\phi_h^n)^\ell)\\
		&=: I_1 + I_2.
	\end{align*}
	Next by using Proposition \ref{transport3} we find
	\begin{align}
		\begin{split}
			I_1 &= \frac{1}{\tau} \int_{t_{n-1}}^{t_n} \left[ m_h(\matdev_h \Pi_h u, \underline{\phi_h^n}(t)) - m(\matdev_\ell \pi_h u,(\underline{\phi_h^n}(t))^\ell ) \right] + \left[m(\matdev_\ell \pi_h u - \matdev_\ell u,(\underline{\phi_h^n}(t))^\ell ) \right] \, dt\\
			&+ \frac{1}{\tau} \int_{t_{n-1}}^{t_n} \left[ g_h(\Pi_h u, \underline{\phi_h^n}(t)) - g_\ell(\pi_h u, (\underline{\phi_h^n}(t))^\ell) \right] + \left[ g_\ell(\pi_h u - u, (\underline{\phi_h^n}(t))^\ell) ) \right] \, dt\\
			& + \frac{1}{\tau} \int_{t_{n-1}}^{t_n} \left[ m(\matdev_\ell u,(\underline{\phi_h^n}(t))^\ell ) - m(\matdev_\ell u^n, ({\phi_h^n})^\ell) + g_\ell(u, (\underline{\phi_h^n}(t))^\ell) - g_\ell(u^n, ({\phi_h^n})^\ell) \right] \, dt.
		\end{split}
		\label{smootherror2}
	\end{align}
	We then write this last term as
	\begin{multline*}
		\frac{1}{\tau} \int_{t_{n-1}}^{t_n} \left[ m(\matdev_\ell u,(\underline{\phi_h^n}(t))^\ell ) - m(\matdev_\ell u^n, ({\phi_h^n})^\ell) + g_\ell(u, (\underline{\phi_h^n}(t))^\ell) - g_\ell(u^n, ({\phi_h^n})^\ell) \right] \, dt\\
		= - \frac{1}{\tau} \int_{t_{n-1}}^{t_n} \int_t^{t_n} \frac{d}{ds} \left( m(\matdev u,(\underline{\phi_h^n}(s))^\ell ) + g(u,(\underline{\phi_h^n}(s))^\ell) \right) \,ds \, dt
	\end{multline*}
	where we have again used \eqref{smootherror1}.
	By using Proposition \ref{transport2} we find that
	\begin{multline*}
		\int_{t_{n-1}}^{t_n} \int_t^{t_n} \frac{d}{ds} \left( m(\matdev u,(\underline{\phi_h^n}(s))^\ell ) + g(u,(\underline{\phi_h^n}(s))^\ell) \right) \,ds \, dt\\
  = \int_{t_{n-1}}^{t_n} \int_t^{t_n} m\left( (\matdev)^2 u,(\underline{\phi_h^n}(s))^\ell \right) + 2g(\matdev u, (\underline{\phi_h^n}(s))^\ell) + m(\matdev u, (\underline{\phi_h^n}(s))^\ell \matdev V) \, ds \,dt.
	\end{multline*}
	Now by using H\"older's inequality, \eqref{timenorm1}, and the assumptions \eqref{smoothassumptions1}, \eqref{smoothassumptions2} we find that
	\begin{align*}
		\left| \frac{1}{\tau} \int_{t_{n-1}}^{t_n} \left[ m(\matdev_\ell u,(\underline{\phi_h^n}(t))^\ell ) - m(\matdev_\ell u^n, ({\phi_h^n})^\ell) + g_\ell(u, (\underline{\phi_h^n}(t))^\ell) - g_\ell(u^n, ({\phi_h^n})^\ell) \right] \, dt \right|\\
		\leq C \int_{t_{n-1}}^{t_n} \left( \|\matdev u\|_{L^2(\Gamma(t))} + \|(\matdev)^2 u\|_{L^2(\Gamma(t))} \right)\|\phi_h^n\|_{L^2(\Gamma_h^n)} \, dt\\
		\leq C \sqrt{\tau} \left(\int_{t_{n-1}}^{t_n} \|\matdev u\|_{L^2(\Gamma(t))}^2 + \|(\matdev)^2 u\|_{L^2(\Gamma(t))}^2 \, dt \right)^{\frac{1}{2}} \|\phi_h^n\|_{L^2(\Gamma_h^n)}.
	\end{align*}
	The bounds for the other two terms in \eqref{smootherror2} follow similarly, using \eqref{derivativedifference1}, \eqref{perturb1}, \eqref{perturb3}, \eqref{ritz3}, \eqref{ritzddt}, \eqref{timenorm1} from which we find
	\begin{align*}
		\left| \frac{1}{\tau} \int_{t_{n-1}}^{t_n} \left[ m_h(\matdev_h \Pi_h u, \underline{\phi_h^n}(t)) - m(\matdev_\ell \pi_h u,(\underline{\phi_h^n}(t))^\ell ) \right] + \left[m(\matdev_\ell \pi_h u - \matdev_\ell u,(\underline{\phi_h^n}(t))^\ell ) \right] \, dt \right| \\
		\leq \frac{Ch^{k+1}}{\sqrt{\tau}} \left( \int_{t_{n-1}}^{t_n} \left( \|u\|_{H^{k+1}(\Gamma(t))}^2 + \|\matdev u \|_{H^{k+1}(\Gamma(t))}^2 \right) \,dt \right)^{\frac{1}{2}} \|\phi_h^n\|_{L^2(\Gamma_h^n)},
	\end{align*}
	and
	\begin{align*}
		\left| \frac{1}{\tau} \int_{t_{n-1}}^{t_n} \left[ g_h(\Pi_h u, \underline{\phi_h^n}(t)) - g_\ell(\pi_h u, (\underline{\phi_h^n}(t))^\ell) \right] + \left[ g_\ell(\pi_h u - u, (\underline{\phi_h^n}(t))^\ell) ) \right] \ \, dt \right|\\
		\leq \frac{Ch^{k+1}}{\sqrt{\tau}} \left(\int_{t_{n-1}}^{t_n} \|u\|_{H^{k+1}(\Gamma(t))}^2 \right)^{\frac{1}{2}} \|\phi_h^n\|_{L^2(\Gamma_h^n)}.
	\end{align*}
	Combining these bounds in \eqref{smootherror2} gives the desired bound for $I_1$.
	We also bound $I_2$ by
	\[I_2 = m(u^n, (\matdev_\ell - \matdev)(\phi_h^n)^\ell) \leq C h^{k+1} \|u^n\|_{L^2(\Gamma(t_n))} \| \phi_h^n\|_{H^1(\Gamma_h^n)},\]
	where we have used \eqref{derivativedifference1}. \qed
\end{proof}
We note that only $E_3$ required the potential to be quadratic here as we explicitly use the Lipschitz continuity of $F_1'$.
With these bounds established it remains to show the error bound.
\begin{theorem}
	\label{smootherrortheorem}
	Given initial data $u_0 \in H^{k+1}(\Gamma(0))$, $U_{h,0} = \Pi_h u_0$, and assuming \eqref{smoothassumptions1}, \eqref{smoothassumptions2} hold and the timestep size is such that
    \[\tau \left( \varepsilon^2 + \frac{C\theta}{\varepsilon} + \frac{L_F^2 + C\theta^2}{\varepsilon^2} \right) \leq \frac{\varepsilon}{4},\]
    then we have the following:
	\begin{gather}
		\begin{aligned}
		    \varepsilon \max_{n=1,...,N_T}\|(u^n)^{-\ell} - U_h^n\|_{L^2(\Gamma_h^n)}^2 + \varepsilon h^2 \max_{n=1,...,N_T} \|\gradgh ((u^n)^{-\ell} - U_h^n)\|_{L^2(\Gamma_h^n)}^2\\
            \leq C_h h^{2({k+1})} + C_\tau \tau^2,
		\end{aligned}\label{smootherroru}\\
        \begin{aligned}
		\tau\sum_{n=1}^{N_T} \|(w^n)^{-\ell} - W_h^n\|_{L^2(\Gamma_h^n)}^2 + \tau h^2 \sum_{n=1}^{N_T}\|\gradgh ((w^n)^{-\ell} - W_h^n)\|_{L^2(\Gamma_h^n)}^2 \leq C_h h^{2({k+1})} + C_\tau \tau^2,
        \end{aligned}\label{smootherrorw}
	\end{gather}
	where
	\begin{gather*}
		C_{\tau} = C \int_0^T \|\matdev u\|_{L^2(\Gamma(t))}^2 + \|(\matdev)^2 u\|_{L^2(\Gamma(t))}^2 \, dt,\\
		C_h = C \int_0^{T} \|\matdev u \|_{H^{k+1}(\Gamma(t))}^2 \,dt + C \sup_{t \in [0,T]} \left(\|w\|_{H^{k+1}(\Gamma(t))}^2 + \frac{\theta^2 \|u\|_{H^{k+1}(\Gamma(t))}^2}{\varepsilon^2} \right) + \frac{CL_F^2}{\varepsilon^2}.
	\end{gather*}
\end{theorem}
\begin{proof}
	We firstly test \eqref{smootherroreqn1} with $\varepsilon\sigma_u^n$ and \eqref{smootherroreqn2} with $\sigma_w^n$ to obtain
	\begin{align*}
		\frac{\varepsilon}{\tau} \left( m_h(\sigma_u^n, \sigma_u^n) - m_h(\sigma_u^{n-1}, \underline{\sigma_u^n}) \right) + \varepsilon a_h(\sigma_w^n , \sigma_u^n) = \varepsilon E_1(\sigma_u^n),
	\end{align*}
	and
	\begin{multline*}
		\varepsilon a_h(\sigma_u^n, \sigma_w^n) = m_h(\sigma_w^n, \sigma_w^n) - \frac{1}{\varepsilon}m_h(F_1'(\Pi_h u^n) - F_1'(U_h^n), \sigma_w^n) + \frac{\theta}{\varepsilon} m_h(\sigma_u^n, \sigma_w^n)\\
		- E_2(\sigma_w^n) - E_3(\sigma_w^n) - E_4(\sigma_w^n).
	\end{multline*}
	Then, similar to as we have done before (see \eqref{discstab2}), we see that
	\begin{align*}
		\frac{1}{\tau} \left( m_h(\sigma_u^n, \sigma_u^n) - m_h(\sigma_u^{n-1}, \underline{\sigma_u^n}) \right) &= \frac{1}{2\tau} \left( m_h(\sigma_u^n, \sigma_u^n) - m_h(\sigma_u^{n-1}, \sigma_u^{n-1}) \right)\\
        &+ \frac{\tau}{2} m_h(\matdev_h \sigma_u^n, \matdev_h \sigma_u^n)\\
	&+ \frac{1}{2 \tau}\left[ m_h(\sigma_u^{n-1},\sigma_u^{n-1}) - m_h(\overline{\sigma_u^{n-1}},\overline{\sigma_u^{n-1}}) \right].
	\end{align*}
	Hence combining all of the above we find that
	\begin{multline}
		\frac{\varepsilon}{2\tau} \left( m_h(\sigma_u^n, \sigma_u^n) - m_h(\sigma_u^{n-1}, \sigma_u^{n-1}) \right) + \frac{\varepsilon\tau}{2} m_h(\matdevtau \sigma_u^n, \matdevtau \sigma_u^n) + m_h(\sigma_w^n, \sigma_w^n)\\
		= \frac{1}{\varepsilon}m_h(F_1'(\Pi_h u^n) - F_1'(U_h^n), \sigma_w^n) + \frac{\theta}{\varepsilon} m_h(\sigma_u^n, \sigma_w^n) + \frac{\varepsilon}{2 \tau}\left[ m_h(\sigma_u^{n-1},\sigma_u^{n-1}) - m_h(\overline{\sigma_u^{n-1}},\overline{\sigma_u^{n-1}}) \right]\\
		+ \varepsilon E_1(\sigma_u^n) + E_2(\sigma_w^n) + E_3(\sigma_w^n) + E_4(\sigma_w^n). \label{smootherror3}
	\end{multline}
	From the Lipschitz continuity of $F_1'$ we find that
	$$ \left| \frac{1}{\varepsilon}m_h(F_1'(\Pi_h u^n) - F_1'(U_h^n), \sigma_w^n) \right| \leq \frac{L_F}{\varepsilon} \|\sigma_u^n\|_{L^2(\Gamma_h^n)}\|\sigma_w^n\|_{L^2(\Gamma_h^n)},$$
	and from using \eqref{timeperturb1}, \eqref{timenorm3} we find
	$$ \frac{\varepsilon}{2 \tau}\left| m_h(\sigma_u^{n-1},\sigma_u^{n-1}) - m_h(\overline{\sigma_u^{n-1}},\overline{\sigma_u^{n-1}}) \right| \leq C \varepsilon \|\sigma_u^{n-1}\|_{L^2(\Gamma_h^{n-1})}^2 .$$
	Now by using these bounds, the bounds in Lemma \ref{smootherrorlemma}, and Young's inequality in \eqref{smootherror3} we obtain
	\begin{multline}
		\label{smootherror4}
		\frac{\varepsilon}{2 \tau} \left( \|\sigma_u^n\|_{L^2(\Gamma_h^n)}^2 - \|\sigma_u^{n-1}\|_{L^2(\Gamma_h^{n-1})}^2 \right) + \frac{\varepsilon \tau}{2 } \| \matdevtau \sigma_u^n\|_{L^2(\Gamma_h^n)}^2 + \|\sigma_w^n\|_{L^2(\Gamma_h^n)}^2 \leq C \varepsilon \|\sigma_u^{n-1}\|_{L^2(\Gamma_h^{n-1})}^2\\
		+ \frac{1}{4} \|\sigma_w^{n}\|_{L^2(\Gamma_h^{n})}^2+ \left( \varepsilon^2 + \frac{L_F^2 + C\theta^2}{\varepsilon^2} \right) \|\sigma_u^{n}\|_{L^2(\Gamma_h^{n})}^2 + \sum_{i=1}^5(\hat{E}_i^n)^2 + \varepsilon\| \gradgh \sigma_u^n\|_{L^2(\Gamma_h^n)}^2,
	\end{multline}
	where
	\begin{multline*}
		\hat{E}_1^n = \frac{Ch^{k+1}}{\sqrt{\tau}} \left( \int_{t_{n-1}}^{t_n}  \|u\|_{H^{k+1}(\Gamma(t))}^2 + \|\matdev u \|_{H^{k+1}(\Gamma(t))}^2 \,dt \right)^{\frac{1}{2}}\\
		+ C \sqrt{\tau} \left(\int_{t_{n-1}}^{t_n} \|\matdev u\|_{L^2(\Gamma(t))}^2 + \|(\matdev)^2 u\|_{L^2(\Gamma(t))}^2 \, dt \right)^{\frac{1}{2}},
	\end{multline*}
	\begin{gather*}
		\hat{E}_2^n = Ch^{k+1} \|w^n\|_{H^{k+1}(\Gamma(t_n))},\\
		\hat{E}_3^n = \frac{CL_F h^{k+1}}{\varepsilon}\left( 1 + \|u^n\|_{H^{k+1}(\Gamma(t_n))} \right),\\
		\hat{E}_4^n = \frac{C\theta h^{k+1}}{\varepsilon} \|u^n\|_{H^{k+1}(\Gamma(t_n))},\\
		\hat{E}_5^n = C h^{k+1} \|u^n\|_{L^2(\Gamma(t_n))}.
	\end{gather*}
	We want to bound the $\|\gradgh \sigma_u^n\|_{L^2(\Gamma_h^n)}$ term in \eqref{smootherror4}, and to do this we test \eqref{smootherroreqn2} with $\sigma_u^n$ for
	\begin{multline*}
		\varepsilon \|\gradgh \sigma_u^n\|_{L^2(\Gamma_h^n)}^2 = m_h(\sigma_w^n,\sigma_u^n) - \frac{1}{\varepsilon}m_h(F_1'(\Pi_h u^n) - F_1'(U_h^n), \sigma_u^n) + \frac{\theta}{\varepsilon}\|\sigma_u^n\|_{L^2(\Gamma_h^n)}^2\\
        - \sum_{i=2}^4 E_i(\sigma_u^n)
	\end{multline*}
	Monotonicity of $F_1'$, Young's inequality, and the bounds in Lemma \ref{smootherrorlemma} then yield
	\begin{align}
		\varepsilon \|\gradgh \sigma_u^n\|_{L^2(\Gamma_h^n)}^2 = \frac{1}{4}\|\sigma_w^n\|_{L^2(\Gamma_h^n)}^2 + \frac{C\theta}{\varepsilon}\|\sigma_u^n\|_{L^2(\Gamma_h^n)}^2 + \sum_{i=2}^4 (\hat{E}_i^n)^2 \label{smootherror5}
	\end{align}
	
	Using \eqref{smootherror5} in \eqref{smootherror4}, taking $\tau$ sufficiently small, in particular
	\[\tau \left( \varepsilon^2 + \frac{C\theta}{\varepsilon} + \frac{L_F^2 + C\theta^2}{\varepsilon^2} \right) \leq \frac{\varepsilon}{4} ,\]
	summing over $n=1,...,N_T$ we find
	\begin{multline*}
		\frac{\varepsilon}{4}\|\sigma_u^{N_T}\|_{L^2(\Gamma_h^{N_T})}^2 +  \frac{\varepsilon \tau^2}{2 } \sum_{n=1}^{N_T}\| \matdevtau \sigma_u^n\|_{L^2(\Gamma_h^n)}^2 + \frac{\tau}{2} \sum_{n=1}^{N_T} \|\sigma_w^n\|_{L^2(\Gamma_h^n)}^2 \leq  C \varepsilon \tau \sum_{n=1}^{N_T} \|\sigma_u^{n-1}\|_{L^2(\Gamma_h^{n-1})}^2\\
		+ C_{\tau} \tau^2 + C_h h^{2({k+1})},
	\end{multline*}
	where $C_h, C_\tau$ are as given above.
	We can then apply a discrete Gr\"onwall inequality (noting $\sigma_u^0 \equiv 0$) for the $L^2$ error bounds for $\sigma_u^n, \sigma_w^n$.
	The $L^2$ bounds in \eqref{smootherroru}, \eqref{smootherrorw} follow as by \eqref{ritz3} we also have
	\begin{gather*}
		\begin{aligned}\max_{n=1,...,N_T} \|(u^n)^{-\ell} - \Pi_h u^n\|_{L^2(\Gamma_h^n)}^2 &= \max_{n=1,...,N_T} \|\rho_u^n\|_{L^2(\Gamma_h^n)}^2\\
        &\leq Ch^{2({k+1})} \sup_{t \in [0,T]} \|u\|_{H^{k+1}(\Gamma(t)}^2,
        \end{aligned}\\
	\begin{aligned}\tau \sum_{n=1}^{N_T} \|(w^n)^{-\ell} - \Pi_h w^n\|_{L^2(\Gamma_h^n)}^2 &= \tau \sum_{n=1}^{N_T} \|\rho_w^n\|_{L^2(\Gamma_h^n)}^2\\
    &\leq Ch^{2({k+1})} \sup_{t \in [0,T]} \|w\|_{H^{k+1}(\Gamma(t)}^2.
    \end{aligned}
	\end{gather*} 
	The $H^1$ error bound follows by applying an inverse inequality (see for example \cite{brenner2008mathematical}), and the $H^1$ bound in \eqref{ritz3}. \qed
\end{proof}
\begin{remark}
	From \eqref{smoothstability} it is clear that we obtain optimal order bounds in the $H^1$ seminorm provided 
	\[\int_{\Gamma_h^0} U_{h,0} = \int_{\Gamma(0)} u_0, \qquad \| \gradgh (\Pi_h u_0 - U_{h,0}) \|_{L^2(\Gamma_h^0)} \leq Ch^k,\]
	but it is not clear that one obtains optimal order in the $L^2$ norm.
\end{remark}

\subsection{Non-quadratic potentials}
We now return to the general setting of a nonlinearity with polynomial growth, $q>1$.
As before, we assume that the solution pair $(u,w)$ solving \eqref{cheqn1}, \eqref{cheqn2} is such that \eqref{smoothassumptions1}, \eqref{smoothassumptions2} hold.
In particular, we see that $u \in H^1_{H^2}$ and by embeddings of Sobolev and evolving Sobolev-Bochner spaces (see \cite{alphonse2023function}) we see $u \in C^0_{C^0}$.
Thus we find there exists some constant $u_{\text{max}}$ such that
$$ \max_{t \in [0,T]} \max_{x \in \Gamma(t)} |u(x,t)| \leq u_{\text{max}},$$
and moreover, $u_{\text{max}}$ is determined by the choice of initial data.
Motivated by this observation, we can modify $F_1$ to be quadratic past some given value, which will allow us to show stability and perform error analysis.
We define $F_M \in C^2(\mathbb{R})$ by
$$ F_M(r) := \begin{cases}
	F_1(M) + F'_1(M) (r-M) + \frac{F_1''(M)}{2}(r-M)^2, & r > M,\\
	F_1(r), & |r| \leq M,\\
	F_1(-M) + F'_1(-M) (r+M) + \frac{F_1''(-M)}{2}(r+M)^2, & r < -M,
\end{cases}
$$
where $M > u_{\text{max}}$ is a constant to be determined.
We note that $F_M'$ is $C^1$ and Lipschitz continuous with Lipschitz constant $F_1''(M)$.
By construction we see that $F_M$ satisfies our assumptions on $F_1$ (it is convex for example), and we can apply the results for a quadratic potential.
In particular, Lemma \ref{implicitstability} (and a discrete Gr\"onwall inequality) shows we have a numerically stable scheme, and Theorem \ref{smootherrortheorem} provides an error bound for this modified potential.\\

What remains to be seen is that we can choose $M$, independent of $h,\tau$, so that the $|U_h^n| \leq M$ for all $n=1,...,N_T$.
This then shows, by considering the results on the quadratic potential, that we have a numerically stable scheme with optimal order error bounds for a general nonlinearity with polynomial growth $(q>1)$.
This is the content of the following lemma.
\begin{lemma}
	Assume \eqref{smoothassumptions1}, \eqref{smoothassumptions2} hold, $h$ is sufficiently small, and $\tau$ is chosen so that $\tau \leq C h^{1+\gamma}$ for some positive constants $C, \gamma$.
	Then there exists some $M$ independent of $\tau, h$ such that
	\[\max_{t \in [0,T]} \max_{x \in \Gamma(t)} |u(x,t)| \leq M, \qquad \max_{n = 1,...,N_T} \max_{x \in \Gamma_h(t)} |U_h^n(x,t)| \leq M,\]
	where $(U_h^n, W_h^n)$ solves \eqref{fdiscfecheqn1}, \eqref{fdiscfecheqn2} for the modified potential $F_M$.
\end{lemma}
\begin{proof}
	We firstly write
	$$ \|U_h^n\|_{L^{\infty}(\Gamma_h^n)} \leq \|\Pi_h u^n\|_{L^{\infty}(\Gamma_h^n)} + \|U_h^n - \Pi_h u^n\|_{L^{\infty}(\Gamma_h^n)}.$$
	Hence using \eqref{ritz4}, and the bound from \eqref{smootherroru} we see
	\begin{align*}
	    \|U_h^n\|_{L^{\infty}(\Gamma_h^n)} &\leq C \sup_{t \in [0,T]} \|u\|_{H^2(\Gamma(t))} + C_h h^k + \frac{C_{\tau} \tau}{h}\\
        &\leq C \sup_{t \in [0,T]} \|u\|_{H^2(\Gamma(t))} + (C_h + C_{\tau})h^\gamma,
	\end{align*}
	where we have used an inverse inequality, and $C_h, C_{\tau}$ are as in Theorem \ref{smootherrortheorem}.
	Note that in using these error bounds we have a constraint on $\tau$ given by
	$$ \tau \left( \varepsilon^2 + \frac{C\theta}{\varepsilon} +  \frac{F_1''(M)^2 + C\theta^2}{\varepsilon^2} \right) \leq \frac{\varepsilon}{4}.$$
	The result clearly holds if
	$$C \sup_{t \in [0,T]} \|u\|_{H^2(\Gamma(t))} + (C_h + C_{\tau})h \leq M,$$
	where we recall that $C_h$ depends on $F_1''(M)$.
	Hence we can fix some $M > u_{\text{max}}$ and take $h$ sufficiently small so that the above holds. \qed
\end{proof}
This result says that although we have considered a numerical scheme with a modified potential, for sufficiently small $\tau, h$ (determined by the solution pair $(u,w)$ solving \eqref{cheqn1}, \eqref{cheqn2}) the numerical solution $(U_h^n, W_h^n)$ ``doesn't see'' the modification and we can still analyse it as we did for the quadratic potential.
In particular, we have the following result.

\begin{corollary}
	Let $F_1$ satisfy the assumptions at the beginning of this chapter, with polynomial growth of $F_1'$ of order $q\geq1$.
	Then, given \eqref{smoothassumptions1}, \eqref{smoothassumptions2} hold, for sufficiently small $\tau, h$ the scheme \eqref{fdiscfecheqn1}, \eqref{fdiscfecheqn2} is numerically stable and has optimal order error bounds when we consider $U_{h,0} = \Pi_h u_0$.
	Moreover from Proposition \ref{continuous dependence} we obtain an optimal order ${H^1}$ error bound if we consider initial data such that
	\[\int_{\Gamma_h^0} U_{h,0} = \int_{\Gamma(0)} u_0, \qquad \| \gradgh (\Pi_h u_0 - U_{h,0}) \|_{L^2(\Gamma_h^0)} \leq Ch^k,\]
	for some constant $C$.
\end{corollary}

\begin{remark}
    We note that this corollary requires both the smallness condition $\tau = \mathcal{O}(\frac{\varepsilon^3}{\theta^2})$ and the CFL condition $\tau = \mathcal{O}(h^{1+\gamma})$ to hold.
\end{remark}

\section{An implicit-explicit scheme}
\label{section:implicitexplicit}
\subsection{Well-posedness}
In this section we outline how results in the previous two chapters change for an implicit-explicit scheme.
Here we consider the scheme such that one finds $U_h^n, W_h^n \in S_h^n$ such that
\begin{gather}\label{impexp eqn1}
	\frac{1}{\tau} \left(m_h(t_n;U_h^n, \phi_h^n) - m_h(t_{n-1};U_h^{n-1}, \underline{\phi_h^{n}}) \right) + a_h(t_n;W_h^n,\phi_h^n) = 0,\\ 
    \begin{aligned}
        m_h(t_n;W_h^n, \phi_h^n)  =\varepsilon a_{h}(t_n;U_h^n,\phi_h) + \frac{1}{\varepsilon} m_h(t_n;F_1'(U_h^n), \phi_h^n)\\
        - \frac{\theta}{\varepsilon} m_h(t_{n-1};U_h^{n-1}, \underline{\phi_h^n}),
    \end{aligned}\label{impexp eqn2}
\end{gather}
for all $\phi_h^{n}\in S_h^{n}$, given some initial data $U_h^0 = U_{h,0}$.
It is worth comparing this discretisation to the backwards differentiation formulae time discretisations proposed in \cite{beschle2022stability} in which the authors consider an extrapolated value in the nonlinear terms.
As before, the initial data $U_{h,0} \in S_h^0$ approximates some sufficiently regular function, $u_0$, on $\Gamma(0)$.
The benefit of this scheme is that it no longer enforces a timestep restriction for the uniqueness of the scheme.
We show this in the following lemma, which is analogous to Lemma \ref{implicit uniqueness}
\begin{lemma}
	Given some $U_{h,0} \in S_h^0$, then if the system \eqref{impexp eqn1}, \eqref{impexp eqn2} admits a solution pair $(U_h^n, W_h^n)$, it is unique.
\end{lemma}
\begin{proof}
	Assume there are two solution pairs $(U_{h,i}^{n},W_{h,i}^{n})$ ($i = 1,2$) of \eqref{impexp eqn1},\eqref{impexp eqn2}.
	We define the following differences
	\begin{gather*}
		\widehat{U}_h^n = U_{h,1}^{n} - U_{h,2}^{n},\\
		\widehat{W}_h^n = W_{h,1}^{n} - W_{h,2}^{n},
	\end{gather*}
	which solve
	\begin{gather}
		\label{impexpunique1}
		m_h(\widehat{U}_h^n, \phi_h^n) + \tau a_h(\widehat{W}_h^n , \phi_h^n) = 0,\\
		m_h(\widehat{W}_h^n, \phi_h^n) = \varepsilon a_h(\widehat{U}_h^n , \phi_h^n) + \frac{1}{\varepsilon} m_h(F_1'(U_{h,1}^n) - F_1'(U_{h,2}^{n}),\phi_h^n),	\label{impexpunique2}
	\end{gather}
	for all $\phi_h^n \in S_h^n$.
	Testing \eqref{impexpunique1} with $\widehat{W}_h^n$ and \eqref{impexpunique2} with $\widehat{U}_h^n$ and subtracting the resulting equations yields
	\[\tau \| \gradgh \widehat{W}_h^n \|_{L^2(\Gamma_h^n)}^2 + \varepsilon \| \gradgh \widehat{U}_h^n \|_{L^2(\Gamma_h^n)}^2 + \frac{1}{\varepsilon} m_h(F_1'(U_{h,1}^{n}) - F_1'(U_{h,2}^{n}),\widehat{U}_h^n)=0.\]
	As noted throughout, $F_1'$ is monotonic and hence
	\[m_h(F_1'(U_{h,1}^{n}) - F_1'(U_{h,2}^{n}),\widehat{U}_h^n) \geq 0.\]
	Thus we have
	\[\tau \| \gradgh \widehat{W}_h^n \|_{L^2(\Gamma_h^n)}^2 + \varepsilon \| \gradgh \widehat{U}_h^n \|_{L^2(\Gamma_h^n)}^2 \leq 0,\]
	and so $\gradgh \widehat{U}_h^n = \gradgh \widehat{W}_h^n = 0$.
	Observing that $\mval{\widehat{U}_h^n}{\Gamma_h^n} = 0$ we may use Poincar\'e's inequality to conclude $\widehat{U}_h^n = 0$, and $\widehat{W}_h^n = 0$ follows. \qed
\end{proof}

The existence of a solution to \eqref{impexp eqn1}, \eqref{impexp eqn2} follows similar arguments to Lemma \ref{fullyimp existence}.
We define a functional, $\widetilde{J_h^n} : D^n \rightarrow \mathbb{R}$, for
\[D^n := \left\{ z_h^n \in S_h^n \mid m_h(z_h^n,1) = m_h(U_h^{n-1},1) \right\}\]
given by
\[\widetilde{J_h^n} (z_h^n) := \frac{1}{\varepsilon} m_h \left( F_1(z_h^n), 1 \right) - \frac{\theta}{\varepsilon}m_h(U_{h,+}^{n-1}, z_h^n) + \frac{\varepsilon}{2} \|\gradgh z_h^n \|_{L^2(\Gamma_h)}^2 + \frac{1}{2 \tau} \normshn{z_h^n - U_{h,+}^{n-1}}{n}^2. \]
It then follows that if $U_h^n = {\arg \min}_{z_h^n \in D^n} \widetilde{J_h^n} (z_h^n)$ then $U_h^n$ solves
\begin{multline*}
\varepsilon a_h(U_h^n, \phi_h^n) + \frac{\theta}{2 \varepsilon}m_h(F_1'(U_h^n), \phi_h^n) - \frac{1}{\varepsilon} m_h(U_{h,+}^{n-1}, \phi_h^n) + \frac{1}{\tau} m_h(\invshn{n} (U_h^n - U_{h,+}^{n-1}), \phi_h^n)\\
- \lambda^n m_h(1, \phi_h^n) =0,
\end{multline*}
for all $\phi_h^n \in S_h^n$, where
$$ \lambda^n = \frac{1}{\varepsilon} \mval{F_1'(U_h^n)}{\Gamma_h^n} - \frac{\theta}{\varepsilon}\mval{U_h^n}{\Gamma_h^n}. $$
Then setting $W_h^n = \lambda^n - \frac{1}{\tau}\invshn{n} (U_h^n - U_{h,+}^{n-1})$, one observes that $(U_h^n, W_h^n)$ solves
\begin{gather*}
\frac{1}{\tau} m_h(U_h^n - U_{h,+}^{n-1}, \phi_h^n) + a_h(W_h^n,\phi_h^n) = 0,\\ 
m_h(W_h^n, \phi_h^n)  =\varepsilon a_{h}(U_h^n,\phi_h) + \frac{1}{\varepsilon} m_h(F_1'(U_h^n), \phi_h^n) - \frac{\theta}{\varepsilon} m_h(U_{h,+}^{n-1}, {\phi_h^n}),
\end{gather*}
for all $\phi_h^n \in S_h^n$.
Clearly this is the same as \eqref{impexp eqn1}, \eqref{impexp eqn2}, and the existence of such a $U_h^n$ follows as before.\\

The proof of (numerical) stability follows by similar arguments to Lemma \ref{implicitstability}, provided that $\tau < \frac{C\varepsilon}{\theta}$ where $C$ is the constant appearing in \eqref{timeperturb1}.
We notice that instead of \eqref{discstab1}, one obtains
\begin{multline*}
	U^n \cdot A^n \frac{(U^n - U^{n-1})}{\tau} + \mathcal{F}_1^n(U^n)\cdot \frac{(U^n - U^{n-1})}{\tau\varepsilon} + W^n \cdot A^n W^n\\
 = \frac{\theta}{\tau\varepsilon} U^{n-1} \cdot M^{n-1}(U^n-U^{n-1})- W^n \cdot \frac{(M^n - M^{n-1})}{\tau} U^{n-1},
\end{multline*}
where the only difference is the term $\frac{\theta}{\tau\varepsilon} U^{n-1} \cdot M^{n-1}(U^n-U^{n-1})$.
One then writes this as
\begin{multline*}
		\frac{\theta}{\tau\varepsilon} U^{n-1} \cdot M^{n-1}(U^n-U^{n-1}) = \frac{\theta}{2 \tau \varepsilon} \left( U^n \cdot M^n U^n - U^{n-1} \cdot M^{n-1} U^{n-1} \right)\\
		- \frac{\theta}{2 \tau \varepsilon}(U^n - U^{n-1}) \cdot M^{n-1} (U^n - U^{n-1})
		+ \frac{\theta}{2 \tau \varepsilon} \left[U^{n} \cdot M^{n-1} U^{n} - U^{n} \cdot M^{n} U^{n} \right],
\end{multline*}
where the first term corresponds to the non-coercive part of the potential, and the second term is non-positive and may be ignored.
The final term is problematic however, as we must control this using \eqref{timeperturb1}, which introduces a term $\frac{C\theta}{\varepsilon} \|U_h^n\|_{L^2(\Gamma_h^n)}^2$, which must be controlled by using the discrete Gr\"onwall inequality.
Following the proof of Lemma \ref{implicitstability} one finds that this happens by taking $\tau$ small enough so that this term can be ``absorbed into the left hand side''.
In particular this requires us to take a timestep size such that $\tau = \mathcal{O}(\frac{\varepsilon}{\theta})$.
While numerical stability introduces a timestep requirement for our implicit-explicit scheme, we observe that this requirement is significantly less restrictive than that of the fully implicit scheme, which requires $\tau = \mathcal{O}(\frac{\varepsilon^3}{\theta^2})$.
This reduced timestep requirement comes from the fact that this discretisation yields requires one to solve a system involving only monotone operators, as we have considered the concave term at the previous timestep.
We omit further details on this proof.\\

One may expect that modifying Proposition \ref{continuous dependence} will introduce some timestep restriction.
We consider this in the follows result, we show in the following result that there is continuous dependence on the initial data with no such timestep restriction.

\begin{proposition}
	\label{impexp continuous dependence}
	Let $U_{h,0}^1, U_{h,0}^2 \in S_h^0$ be such that $\mval{U_{h,0}^1}{\Gamma_h(0)} = \mval{U_{h,0}^2}{\Gamma_h(0)}$.
	Denoting the solution of \eqref{impexp eqn1},\eqref{impexp eqn2} with initial data $U_{h,0}^i$ by $(U_{h}^{n,i},W_{h}^{n,i})$ we have that for $1\leq N \leq N_T$,
	\begin{align}
		\label{impexp stability}
		\normshn{U_h^{N,1} -U_h^{N,2} }{N}^2 + \sum_{n=1}^N  \| \gradgh(U_h^{N,1} -U_h^{N,2}) \|_{L^2(\Gamma_h^n)}^2 \leq e^{C t_N}\normshn{U_{h,0}^1 - U_{h,0}^2}{0}^2.
	\end{align}
	
\end{proposition}
\begin{proof}
	As before, we define
	\begin{gather*}
		\widehat{U}_h^n = U_{h,1}^n - U_{h,2}^n,\\
		\widehat{W}_h^n = W_{h,1}^n - W_{h,2}^n,
	\end{gather*}
	and note that from \eqref{impexp eqn1}, \eqref{impexp eqn2}
	\begin{gather}
		\frac{1}{\tau} \left( m_h(\widehat{U}_h^n, \phi_h^n) - m_h(\widehat{U}_h^{n-1}, \underline{\phi_h^n}) \right) + a_h(\widehat{W}_h^n,\phi_h^n) = 0, \label{impexp stability1}\\
		m_h(\widehat{W}_h^n,\phi_h^n) = \varepsilon a_h(\widehat{U}_h^n, \phi_h^n) + \frac{1}{\varepsilon} m_h(F_1'(U_{h,1}^n) - F_1'(U_{h,2}^n), \phi_h^n) - \frac{\theta}{\varepsilon} m_h(\widehat{U}_h^{n-1}, \underline{\phi_h^n}), \label{impexp stability2}
	\end{gather}
	holds for all $\phi_h^n \in S_h^n$.\\
	
	We test \eqref{impexp stability1} with $\invshn{n} \widehat{U}_h^n$ for
	\[ m_h(\widehat{U}_h^n, \invshn{n} \widehat{U}_h^n) - m_h(\widehat{U}_h^{n-1}, \underline{\invshn{n} \widehat{U}_h^n}) + \tau a_h(\widehat{W}_h^n,\invshn{n} \widehat{U}_h^n) = 0, \]
	and observe
    \begin{multline*}
        a_h(\widehat{W}_h^n,\invshn{n} \widehat{U}_h^n) = m_h(\widehat{W}_h^n, \widehat{U}_h^n) = \varepsilon a_h(\widehat{U}_h^n, \widehat{U}_h^n) + \frac{1}{\varepsilon} m_h(F_1'(U_{h,1}^n) - F_1'(U_{h,2}^n), \widehat{U}_h^n)\\
        - \frac{\theta}{\varepsilon} m_h(\widehat{U}_h^{n-1}, \underline{\widehat{U}_h^n}), 
    \end{multline*}
	and one obtains
	\begin{multline}
		m_h(\widehat{U}_h^n, \invshn{n} \widehat{U}_h^n) - m_h(\widehat{U}_h^{n-1}, \underline{\invshn{n} \widehat{U}_h^n}) + \varepsilon \tau a_h(\widehat{U}_h^n, \widehat{U}_h^n) + \frac{\tau}{\varepsilon} m_h(F_1'(U_{h,1}^n) - F_1'(U_{h,2}^n), \widehat{U}_h^n)\\
		= \frac{\tau\theta}{\varepsilon} m_h(\widehat{U}_h^{n-1}, \underline{\widehat{U}_h^n}). \label{impexp stability3}
	\end{multline}
	Convexity of $F_1(\cdot)$ implies $m_h(F_1'(U_{h,1}^n) - F_1'(U_{h,2}^n), \widehat{U}_h^n) \geq 0$ as usual, and so we neglect this term.
	As in the proof of Proposition \ref{continuous dependence} we observe from \eqref{invlap difference1}, \eqref{invlap difference2} that
	\[ m_h(\widehat{U}_h^n, \invshn{n} \widehat{U}_h^n) - m_h(\widehat{U}_h^{n-1}, \underline{\invshn{n} \widehat{U}_h^n}) = \frac{1}{2} \left( \normshn{\widehat{U}_h^n}{n}^2 - \normshn{\widehat{U}_{h,+}^{n-1}}{n}^2 + \normshn{\widehat{U}_h^n - \widehat{U}_{h,+}^{n-1}}{n}^2 \right), \]
	and so \eqref{impexp stability3} becomes
	\begin{multline}
		\normshn{\widehat{U}_h^n}{n}^2 - \normshn{\widehat{U}_{h}^{n-1}}{n-1}^2+ \normshn{\widehat{U}_h^n - \widehat{U}_{h,+}^{n-1}}{n}^2 + 2\varepsilon\tau \|\gradgh \widehat{U}_h^n \|_{L^2(\Gamma_h^n)}^2\\
        \leq \frac{2\tau\theta}{\varepsilon} m_h(\widehat{U}_h^{n-1}, \underline{\widehat{U}_h^n}) + \left( \normshn{\widehat{U}_{h,+}^{n-1}}{n}^2- \normshn{\widehat{U}_{h}^{n-1}}{n-1}^2 \right). \label{impexp stability4}
	\end{multline}
	
	As shown in the proof of Proposition \ref{continuous dependence} we have that this final term is bounded by using
	\begin{align}
		\left|\normshn{\widehat{U}_{h,+}^{n-1}}{n}^2- \normshn{\widehat{U}_{h}^{n-1}}{n-1}^2\right| \leq C \tau \normshn{\widehat{U}_{h}^{n-1}}{n-1}^2. \label{impexp stability5}
	\end{align}
	All that remains is to bound $ m_h(\widehat{U}_h^{n-1}, \underline{\widehat{U}_h^n})$, for which we use \eqref{timenorm2}, the definition of $\invshn{n}$ and Young's inequality to see that
	\begin{align}
		\frac{2\tau\theta}{\varepsilon} m_h(\widehat{U}_h^{n-1}, \underline{\widehat{U}_h^n}) \leq \varepsilon \tau \|\gradgh \widehat{U}_h^n\|_{L^2(\Gamma_h^n)}^2 + \frac{C \tau \theta^2}{\varepsilon^3} \normshn{\widehat{U}_h^{n-1}}{n-1}
	\end{align}
	Now using \eqref{impexp stability4}, \eqref{impexp stability5} in \eqref{impexp stability3} and summing over $n$ one obtains
	\begin{multline*}
		\normshn{\widehat{U}_h^N}{N}^2 - \normshn{\widehat{U}_h^0}{0}^2 + \varepsilon \tau \sum_{n=1}^N \|\gradgh \widehat{U}_h^n\|_{L^2(\Gamma_h^n)}^2 + \sum_{n=1}^N \normshn{\widehat{U}_h^n - \widehat{U}_{h,+}^{n-1}}{n}^2\\
        \leq C \tau \sum_{n=0}^{N-1} \normshn{\widehat{U}_h^n}{n}^2,
	\end{multline*}
	where $C$ depends on $\varepsilon, \theta$.
	The result follows from a discrete Gr\"onwall inequality. \qed
\end{proof}

\subsection{Error analysis}
We use the same notation as in Section \ref{section:error}
\begin{align*}
	(u^n)^{-\ell} - U_h^n &= \underbrace{(u^n)^{-\ell} - \Pi_h u^n}_{=: \rho_u^n} + \underbrace{\Pi_h u^n - U_h^n}_{=: \sigma_u^n},\\
	(w^n)^{-\ell} - W_h^n &= \underbrace{(w^n)^{-\ell} - \Pi_h w^n}_{=: \rho_w^n} + \underbrace{\Pi_h w^n - U_h^n}_{=: \sigma_w^n},
\end{align*}
where $u^n = u(t_n), w^n = w(t_n),$ and we understand $(U_h^n, W_h^n)$ as the unique solution pair of \eqref{impexp eqn1}, \eqref{impexp eqn2}.
Note that \eqref{smoothassumptions1} implies that $\Pi_h u^n, \Pi_h w^n$ are defined for all $n$.
Recall that
\begin{multline*}
\max_{n=1,...,N_T} \|\rho_u^n\|_{L^2(\Gamma_h^n)}^2 + \tau \sum_{n=1}^{N_T}  \|\rho_w^n\|_{L^2(\Gamma_h^n)}^2\\
\leq Ch^{2(k+1)} \sup_{t \in [0,T]} \left(\|u\|_{H^{k+1}(\Gamma(t))} + \|w\|_{H^{k+1}(\Gamma(t))}^2\right),
\end{multline*}
and so our focus is in bounding $\sigma_u^n$ and $\sigma_w^n$ as before.\\

As in Section \ref{section:error} one finds that
\begin{align}
	\frac{1}{\tau} \left( m_h(\sigma_u^n, \phi_h^n) - m_h(\sigma_u^{n-1}, \underline{\phi_h^n}) \right) + a_h(\sigma_w^n , \phi_h^n) = E_1(\phi_h^n),
	\label{impexp erroreqn1}
\end{align}
where
\begin{align*}
	E_1(\phi_h^n) := \frac{1}{\tau} \left( m_h(\Pi_h u^n, \phi_h^n) - m_h(\Pi_h u^{n-1}, \underline{\phi_h^n}) \right) -m(\matdev u^n, (\phi_h^n)^\ell)- g(u^n, (\phi_h^n)^\ell),
\end{align*}
which is the same as \eqref{smootherroreqn1}.
The analogue of \eqref{smootherroreqn2} is now given by
\begin{multline}
	m_h(\sigma_w^n, \phi_h^n) - \varepsilon a_h(\sigma_u^n, \phi_h^n) - \frac{1}{\varepsilon}m_h(F_1'(\Pi_h u^n) - F_1'(U_h^n), \phi_h^n) + \frac{\theta}{\varepsilon} m_h(\sigma_u^{n-1}, \underline{\phi_h^n})\\
	= E_2(\phi_h^n) + E_3(\phi_h^n) + E_4(\phi_h^n) + E_5(\phi_h^n)
	\label{impexp erroreqn2}
\end{multline}
where
\begin{gather*}
	E_2(\phi_h^n) := m_h(\Pi_h w^n, \phi_h^n) - m(w^n, (\phi_h^n)^\ell),\\
	E_3(\phi_h^n) := \frac{1}{\varepsilon} \left[m(F_1'(u^n), (\phi_h^n)^\ell )  - m_h(F_1'(\Pi_h u^n), \phi_h^n)\right],\\
	E_4(\phi_h^n) := \frac{\theta}{\varepsilon} \left[ m_h(\Pi_h u^n, \phi_h^n) - m(u^n, (\phi_h^n)^\ell) \right],\\
	E_5(\phi_h^n) := \frac{\theta}{\varepsilon} \left[ m_h(\Pi_h u^{n-1}, \underline{\phi_h^n}) - m_h(\Pi_h u^n, \phi_h^n) \right].
\end{gather*}

Noting the consistency errors are defined only in terms of the true solution pair $(u,w)$, one finds that Lemma \ref{smootherrorlemma} still holds and so we have bounded $|E_i(\phi_h^n)|$ for $i=1,...,4$ but it remains to bound $|E_5(\phi_h^n)|$.
This is the content of the following lemma.

\begin{lemma}
	\label{impexp errorlemma}
	For $E_5$ defined as above, one has \begin{align}
		|E_5(\phi_h^n)| \leq C \sqrt{\tau} \left(\int_{t_{n-1}}^{t_n} \|u\|_{H^2(\Gamma(t))}^2 + \|\matdev u\|_{H^2(\Gamma(t))}^2 \, dt \right)^\frac{1}{2} \|{\phi_h^n}\|_{L^2(\Gamma_h^n)}.
	\end{align}
\end{lemma}
\begin{proof}
	Notice that
	\begin{align*}
		E_5(\phi_h^n) &= \int_{t_{n-1}}^{t_n}\frac{d}{dt} m_h(\Pi_h u(t), \underline{\phi_h^n}(t)) \, dt\\
  &= \int_{t_{n-1}}^{t_n} m_h(\matdev_h \Pi_h u(t), \underline{\phi_h^n}(t)) + g_h(\Pi_h u, \underline{\phi_h^n}(t)) \, dt.
	\end{align*}
	We use \eqref{ritz1}, \eqref{ritzddt}, and the bounds on $V_h$ to see that
	\begin{align*}
		|E_5(\phi_h^n)| \leq C \int_{t_{n-1}}^{t_n}  \left( \|u\|_{H^2(\Gamma(t))} + \|\matdev u\|_{H^2(\Gamma(t))} \right) \|\underline{\phi_h^n}(t)\|_{L^2(\Gamma_h(t))} \, dt.
	\end{align*}
	The result then follows from \eqref{timenorm1}, and an application of H\"older's inequality. \qed
\end{proof}

With this we can adapt Theorem \ref{smootherrortheorem} as follows.

\begin{theorem}
	\label{impexp errortheorem}
	Given initial data $u_0 \in H^{k+1}(\Gamma(0))$, $U_{h,0} = \Pi_h u_0$, and assuming \eqref{smoothassumptions1}, \eqref{smoothassumptions2} hold then for sufficiently small $\tau,h$ such that
    \[\tau \left( \varepsilon^2 + \frac{C \theta}{\varepsilon} + \frac{L_F^2 + C\theta^2}{\varepsilon^2} \right) \leq \frac{\varepsilon}{4},\]
    we have the following:
	\begin{gather*}
		\begin{aligned}
    \varepsilon \max_{n=1,...,N_T}\|(u^n)^{-\ell} - U_h^n\|_{L^2(\Gamma_h^n)}^2 + \varepsilon h^2 \max_{n=1,...,N_T} \|\gradgh ((u^n)^{-\ell} - U_h^n)\|_{L^2(\Gamma_h^n)}^2\\
    \leq C_h h^{2({k+1})} + C_\tau \tau^2,
    \end{aligned}\\
	\begin{aligned}
    \tau\sum_{n=1}^{N_T} \|(w^n)^{-\ell} - W_h^n\|_{L^2(\Gamma_h^n)}^2 + \tau h^2 \sum_{n=1}^{N_T}\|\gradgh ((w^n)^{-\ell} - W_h^n)\|_{L^2(\Gamma_h^n)}^2\\
    \leq C_h h^{2({k+1})} + C_\tau \tau^2,
    \end{aligned}
	\end{gather*}
	where $(U_h^n,W_h^n)$ is the solution pair from the implicit-explicit scheme \eqref{impexp eqn1}, \eqref{impexp eqn2}.
	\begin{gather*}
		C_{\tau} = C \int_0^T \|\matdev u\|_{L^2(\Gamma(t))}^2 + \|(\matdev)^2 u\|_{L^2(\Gamma(t))}^2 \, dt,\\
		C_h = C \int_0^{T} \|\matdev u \|_{H^{k+1}(\Gamma(t))}^2 \,dt + C \sup_{t \in [0,T]} \left(\|w\|_{H^{k+1}(\Gamma(t))}^2 + \frac{\theta^2 \|u\|_{H^{k+1}(\Gamma(t))}^2}{\varepsilon^2} \right) + \frac{CL_F^2}{\varepsilon^2}.
	\end{gather*}
\end{theorem}
\begin{proof}
	This proof is largely the same as that of Theorem \ref{smootherrortheorem}, with obvious modifications in line with the preceding lemma, and so we omit the calculations.
	We note that here we assume $F_1'$ is Lipschitz continuous, and the extension to a non-Lipschitz $F_1'$ holds before.
	As in seen in Theorem \ref{smootherrortheorem}, we take $\tau$ sufficiently small so that
	\[\tau \left( \varepsilon^2 + \frac{C \theta}{\varepsilon} + \frac{L_F^2 + C\theta^2}{\varepsilon^2} \right) \leq \frac{\varepsilon}{4},\]
	and so we still require the same timestep restriction to obtain an error bound. \qed
\end{proof}

While this proof gives an optimal order of convergence, it has the same timestep restriction as in Theorem \ref{smootherrortheorem}, which almost defeats the point of considering the implicit-explicit scheme.
One can possibly avoid this by proving a weaker error bound, similar to the arguments of \cite{barrett1995error}, which we outline here.\\

Firstly, testing \eqref{impexp erroreqn1} with $\invshn{n}\sigma_u^n$ for
\begin{multline*}
	\frac{1}{\tau} \left( a_h(\invshn{n}\sigma_u^n, \invshn{n}\sigma_u^n) -a_h(\invshn{n-1}\sigma_u^{n-1}, \invshn{n-1}\sigma_u^{n-1}) \right) + a_h(\sigma_w^n , \invshn{n}\sigma_u^n)\\
	+ \frac{1}{\tau} a_h(\invshn{n-1}\sigma_u^{n-1},\underline{\invshn{n}\sigma_u^n}- \invshn{n-1}\sigma_u^{n-1} ) = E_1(\invshn{n}\sigma_u^n),
\end{multline*}
where we have used the definition of $\invshn{n}$.
Likewise from \eqref{impexp erroreqn2} we find that
\begin{multline*}
	a_h(\sigma_w^n , \invshn{n}\sigma_u^n) = \varepsilon a_h(\sigma_u^n, \sigma_u^n) + \frac{1}{\varepsilon}m_h(F_1'(\Pi_h u^n) - F_1'(U_h^n), \sigma_u^n) - \frac{\theta}{\varepsilon} m_h(\sigma_u^{n-1}, \underline{\sigma_u^n})\\
 + \sum_{i=2}^5 E_i(\sigma_u^n).
\end{multline*}
Combining these one finds that
\begin{multline*}
	\frac{1}{\tau} \left( \normshn{\sigma_u^n}{n}^2 - \normshn{\sigma_u^{n-1}}{n-1}^2\right) + \varepsilon \|\gradgh\sigma_u^n\|_{L^2(\Gamma_h^n)}^2 + \frac{1}{\varepsilon}m_h(F_1'(\Pi_h u^n) - F_1'(U_h^n), \sigma_u^n)\\
 = \frac{\theta}{\varepsilon} m_h(\sigma_u^{n-1}, \underline{\sigma_u^n})
+ \frac{1}{\tau} a_h(\invshn{n-1}\sigma_u^{n-1},\underline{\invshn{n}\sigma_u^n}- \invshn{n-1}\sigma_u^{n-1} ) + E_1(\invshn{n}\sigma_u^n)-\sum_{i=2}^5 E_i(\sigma_u^n).
\end{multline*}
The simplification now comes from using the monotonicity of $F_1'$ so that $m_h(F_1'(\Pi_h u^n) - F_1'(U_h^n), \sigma_u^n) \geq 0$, and one should be able to remove the $L_F$ term on the timestep size restriction.
Similarly the non-coercive term now involves both $\sigma_u^n$ and $\sigma_u^{n-1}$, which means we can obtain a less strict timestep size restriction.
We do not continue this argument, which is similar to the proof of Proposition \ref{impexp continuous dependence}, but we note that the end result should be an optimal order $L^2_{H^1}$ error bound for $U_h^n$.

\section{Numerical experiments}
\label{section:implementation}
In this section we implement both the fully implicit scheme and the implicit-explicit scheme for linear finite elements using the DUNE library \cite{alkamper2014dune}.
We can express \eqref{fdiscfecheqn1}, \eqref{fdiscfecheqn2} in a block matrix form as
$$ \begin{pmatrix}
	M^n & \tau A^n\\
	-\varepsilon A^n + \frac{\theta}{\varepsilon}M^n  & M^n
\end{pmatrix}
\begin{pmatrix}
	\alpha^n\\
	\beta^n
\end{pmatrix}
- \frac{1}{\varepsilon} \begin{pmatrix}
	0\\
	\mathcal{F}_1^n(\alpha^n)
\end{pmatrix} =
\begin{pmatrix}
	M^{n-1} \alpha^{n-1}\\
	0
\end{pmatrix}
$$
where
$$ U_h^n = \sum_{j=1}^{N_h} \alpha_j^n \phi_j^n, \quad W_h^n = \sum_{j=1}^{N_h} \beta_j^n \phi_j^n,$$
\begin{gather*}
	M_{ij} = m_h(t_n;\phi_i^n,\phi_j^n),\\
	A_{ij} = a_{h}(t_n;\phi_i^n,\phi_j^n),\\
	\mathcal{F}_1^n(\alpha^n)_j = m_h \left(t_n;F_1' \left( \sum_{i=1}^{N_h} \alpha_i^n \phi_i^n \right), \phi_j^n\right),
\end{gather*}
and $\phi_i^n$ denotes the `$i$'th basis function at time $t_n$.
Similarly the block matrix form of \eqref{impexp eqn1}, \eqref{impexp eqn2} is
\[\begin{pmatrix}
	M^n & \tau A^n\\
	-\varepsilon A^n  & M^n
\end{pmatrix}
\begin{pmatrix}
	\alpha^n\\
	\beta^n
\end{pmatrix}
- \frac{1}{\varepsilon} \begin{pmatrix}
	0\\
	\mathcal{F}_1^n(\alpha^n)
\end{pmatrix} =
\begin{pmatrix}
	M^{n-1} \alpha^{n-1}\\
	-\frac{\theta}{\varepsilon}M^{n-1}\alpha^{n-1}
\end{pmatrix}
\]
Here we consider the potential given by
$$ F(z) = \frac{(1-z^2)^2}{4} = \frac{1 + z^4}{4} - \frac{z^2}{2},$$
which fits the framework we have considered.\\

All quadrature rules used are of sufficiently high order that the effects of numerical integration may be ignored (as in our preceding theory).The nonlinear system is solved using Newton's method and with a tolerance of $10^{-11}$, and the corresponding linearisation is solved directly. 

\subsection{Experimental order of convergence on an evolving sphere}

Here we consider an oscillating sphere given by the zero level set of the function
$$ \phi(x,y,z,t) = x^2 +y^2 + z^2 - 0.9 - 0.1 \cos(20 \pi t),$$
over a time interval $t \in [0,0.1]$.
We choose the initial data to be
$$ u_0(x,y,z) = 0.5 x \sin(\pi y),$$
and an interface width $\varepsilon = 0.1$.
The experimental order of convergence (EOC) is computed by solving the equation on a coarse mesh, and prolonging this coarse solution onto a finer mesh.
This process is analogous to lifting.
We obtain an error by approximating the true solution with a fine solution ($\tau = 10^{-5}, h \approx 3.827328 \cdot 10^{-2}$) of the fully implicit scheme, as the exact solution is not known, and we choose the timestep size to be $\mathcal{O}(h^2)$ to avoid a bottleneck in the error.
We see that the EOC is essentially of the order predicted by Theorem \ref{smootherrortheorem} and Theorem \ref{impexp errortheorem} for linear finite elements --- see Tables \ref{table:SmoothUTable}, \ref{table:SmoothWTable}, \ref{table:ImpExpUTable}, \ref{table:ImpExpWTable}.
We observe that the error in the fully implicit scheme is smaller  than the implicit-explicit scheme (albeit by a negligible about for $u$). The benefit of the implicit-explicit scheme, as shown in our preceding theory, is that it is stable for a smaller timestep size $\tau$.
Lastly, the example here is chosen so that we can compute a solution for a timestep size, $\tau$, which is large in comparison to the interface width, $\varepsilon$.
In practice this is not usually numerically stable and one must choose $\tau$ in accordance with the timestep size requirements for the appropriate numerical scheme.
In particular we observed stability issues when dealing with a large timestep size if the surface, $\Gamma(t)$, was shrinking sufficiently quickly --- we have thus not included these experiments.

\begin{table}[ht]
	\centering
	\begin{tabular}{ |c|c|c| } 
		\hline
		$h$ & $\|u(T) - (U_h^{N_T})^\ell\|_{L^2(\Gamma(T))}$ & EOC\\
		\hline
		$6.123724 \cdot 10^{-1}$ & $6.837856 \cdot 10^{-1}$ & - \\ 
		\hline
		$3.061862\cdot 10^{-1}$ & $2.181480 \cdot 10^{-1}$ & $1.648237$ \\ 
		\hline
		$1.530931\cdot 10^{-1}$ & $5.132094 \cdot 10^{-2}$ & $2.087688$ \\ 
		\hline
		$7.654655 \cdot 10^{-2}$ & $1.150091 \cdot 10^{-2}$ & $2.157799$ \\ 
		\hline
	\end{tabular}
	\caption{Table of EOC for $u$ on an evolving sphere for the fully implicit scheme \eqref{fdiscfecheqn1}, \eqref{fdiscfecheqn2}.}
 \label{table:SmoothUTable}
\end{table}

\begin{table}[ht]
	\centering
	\begin{tabular}{ |c|c|c| } 
		\hline
		$h$ & $\|w(T) - (W_h^{N_T})^\ell\|_{L^2(\Gamma(T))}$ & EOC\\
		\hline
		$6.123724 \cdot 10^{-1}$ & $4.020470 \cdot 10^{-1}$ & - \\ 
		\hline
		$3.061862\cdot 10^{-1}$ & $1.272730 \cdot 10^{-1}$ & $1.659437$ \\ 
		\hline
		$1.530931\cdot 10^{-1}$ & $2.823982 \cdot 10^{-2}$ & $2.172124$ \\ 
		\hline
		$7.654655 \cdot 10^{-2}$ & $5.847620 \cdot 10^{-3}$ & $2.271810$ \\ 
		\hline
	\end{tabular}
	\caption{Table of EOC for $w$ on an evolving sphere for the fully implicit scheme \eqref{fdiscfecheqn1}, \eqref{fdiscfecheqn2}.}
 \label{table:SmoothWTable}
\end{table}

\begin{table}[ht]
	\centering
	\begin{tabular}{ |c|c|c| } 
		\hline
		$h$ & $\|u(T) - (U_h^{N_T})^\ell\|_{L^2(\Gamma(T))}$ & EOC\\
		\hline
		$6.123724 \cdot 10^{-1}$ & $7.067299 \cdot 10^{-1}$ & - \\ 
		\hline
		$3.061862\cdot 10^{-1}$ & $2.252196 \cdot 10^{-1}$ & $1.6498267$ \\ 
		\hline
		$1.530931\cdot 10^{-1}$ & $5.307493 \cdot 10^{-2}$ & $2.085230$ \\ 
		\hline
		$7.654655 \cdot 10^{-2}$ & $1.193632 \cdot 10^{-2}$ & $2.156721$ \\ 
		\hline
	\end{tabular}
	\caption{Table of EOC for $u$ on an evolving sphere for the implicit-explicit scheme \eqref{impexp eqn1}, \eqref{impexp eqn2}.}
 \label{table:ImpExpUTable}
\end{table}

\begin{table}[ht]
	\centering
	\begin{tabular}{ |c|c|c| } 
		\hline
		$h$ & $\|w(T) - (W_h^{N_T})^\ell\|_{L^2(\Gamma(T))}$ & EOC\\
		\hline
		$6.123724 \cdot 10^{-1}$ & $2.369274 \cdot 10^{-1}$ & - \\ 
		\hline
		$3.061862\cdot 10^{-1}$ & $7.589871 \cdot 10^{-2}$ & $1.642298$ \\ 
		\hline
		$1.530931\cdot 10^{-1}$ & $2.436149 \cdot 10^{-2}$ & $1.639473$ \\ 
		\hline
		$7.654655 \cdot 10^{-2}$ & $6.847620 \cdot 10^{-3}$ & $1.830809$ \\ 
		\hline
	\end{tabular}
	\caption{Table of EOC for $w$ on an evolving sphere for the implicit-explicit scheme \eqref{impexp eqn1}, \eqref{impexp eqn2}.}
 \label{table:ImpExpWTable}
\end{table}

\subsection{Dynamics on an evolving torus}
\subsubsection{Constant surface area torus}
Here we consider a torus given by the zero level set of the function
$$ \phi(x,y,z,t) = \left( \sqrt{x^2 + y^2} - 0.75\left(1 + \frac{4t}{3}\right) \right)^2 + z^2 - \left(\frac{0.25}{1 + \frac{4t}{3}}\right)^2,$$
over a time interval $t \in [0,1]$.
This surface is chosen as it has a constant surface area, $|\Gamma(t)| = \frac{3 \pi^2}{4}$, but is nonstationary.
As seen in Figure \ref{fig:ConstantTorusGL}, the Ginzburg-Landau functional is non-monotonic in time (unlike the case when $\Gamma(t)$ is a stationary domain).\\

Here we have chosen the initial data to be
$$ u_0(x,y,z) = 0.5 x y \sin(10 \pi z),$$
and set $\varepsilon = 0.05$.
Here we only consider the fully implicit scheme \eqref{fdiscfecheqn1}, \eqref{fdiscfecheqn2} on a mesh consisting of $6016$ elements ($h \approx 0.136755$), and a timestep size $\tau = 5\cdot 10^{-5}$.

\begin{figure}[ht]
	\centering
	\includegraphics[width=\linewidth]{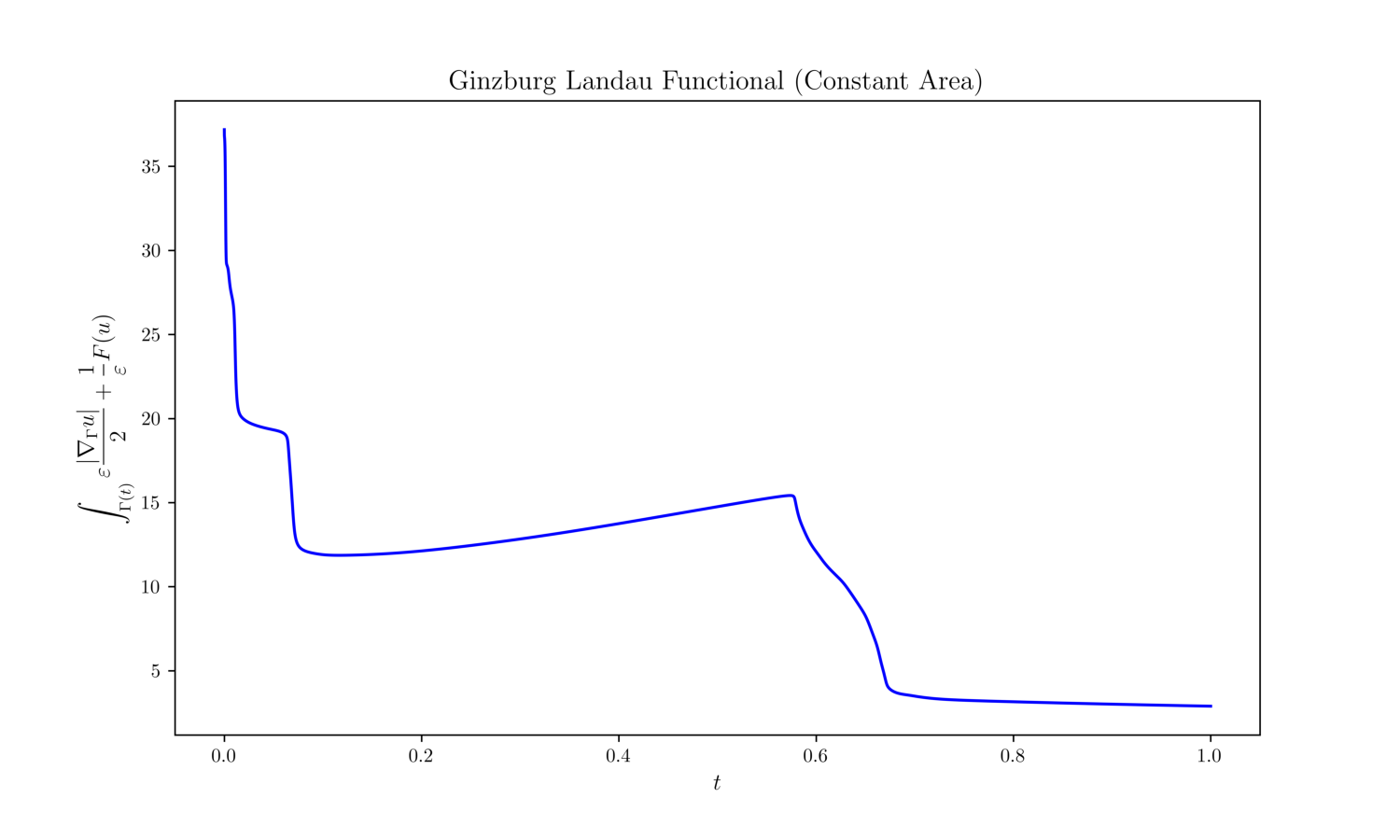}
	\caption{Plot of the Ginzburg-Landau functional (for the fully implicit scheme) on an evolving torus with constant surface area over $ t\in [0,1]$.}
	\label{fig:ConstantTorusGL}
\end{figure}

\subsubsection{Periodic torus}
Here we consider a torus given by the zero level set of the function
$$ \phi(x,y,z,t) = \left( \sqrt{x^2 + y^2} - 0.75 \right)^2 + z^2 - (0.25+0.1 \sin(20 \pi t))^2,$$
over a time interval $t \in [0,1]$.
We choose the initial data to be
$$ u_0(x,y,z) = 0.5 x y \sin(10 \pi z),$$
and set $\varepsilon = 0.05$.
Here we only consider the fully implicit scheme \eqref{fdiscfecheqn1}, \eqref{fdiscfecheqn2} on a mesh consisting of $6016$ elements ($h \approx 0.136755$), and a timestep size $\tau = 5\cdot 10^{-5}$.\\

Plots of the evolution (Figure \ref{fig:SmoothTorusEvolution}), and the Ginzburg-Landau functional (Figure \ref{fig:SmoothTorusGL}) are seen below, demonstrating the dynamics of the solution.
Again, the Ginzburg-Landau functional does not decrease monotonically, as one finds for a stationary domain.  but instead seems to converge to a periodic function.
This matches observations in  \cite{beschle2022stability,elliott2015evolving} and also 
\cite{elliott2015error} for a second order linear PDE where convergence to a periodic solution is also seen for a periodic domain.
This phenomena remains to be studied analytically for the Cahn-Hilliard equation, however there has been some work (see \cite{elliott2015time}) in this direction for a linear advection-diffusion equation.

\begin{figure}[ht]
	\centering
	\includegraphics[width=\linewidth]{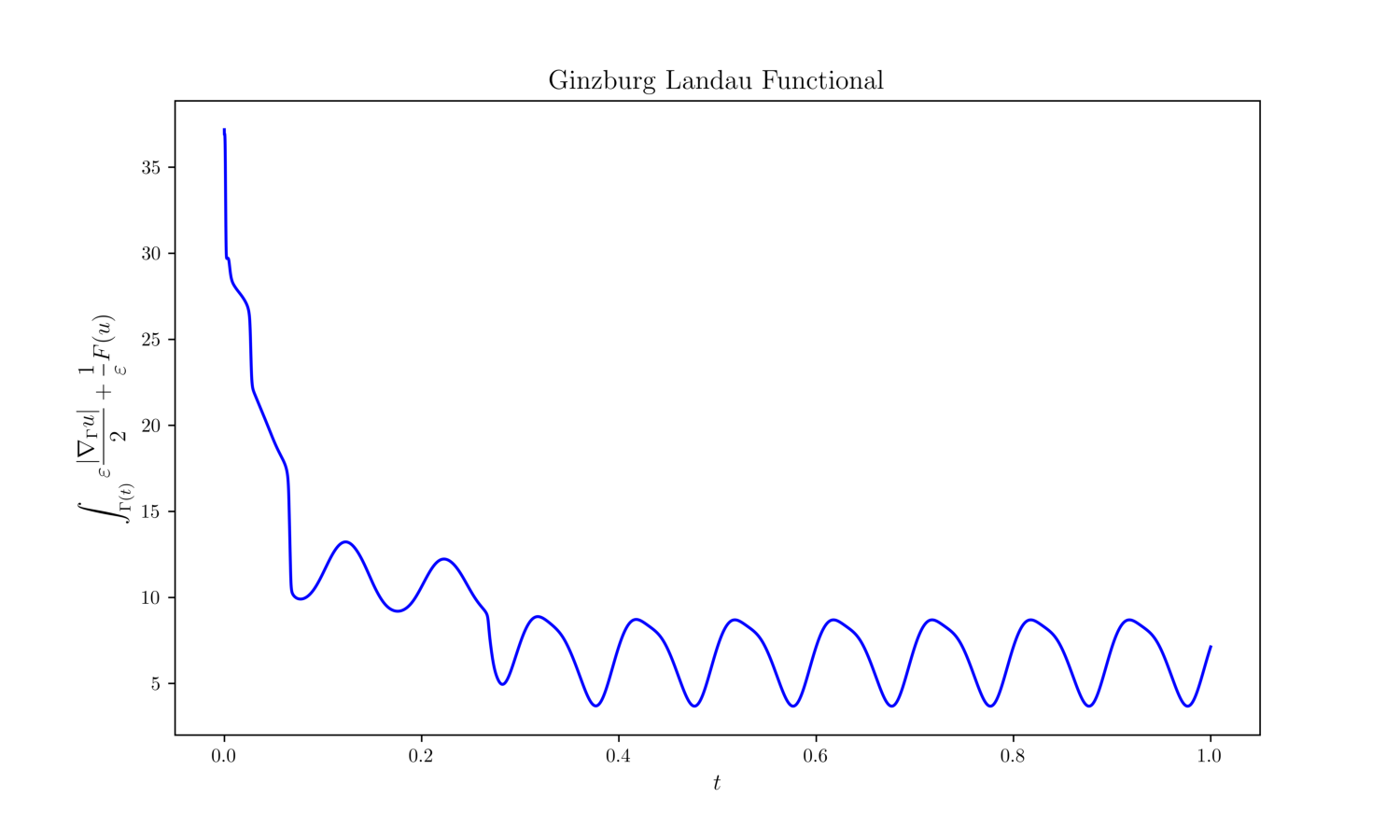}
	\caption{Plot of the Ginzburg-Landau functional (for the fully implicit scheme) on a torus with periodic evolution over $ t\in [0,1]$.}
	\label{fig:SmoothTorusGL}
\end{figure}

\begin{figure}[ht]
\centering
	\begin{subfigure}{.45\linewidth}
        \centering
		\includegraphics[width=1.2\textwidth]{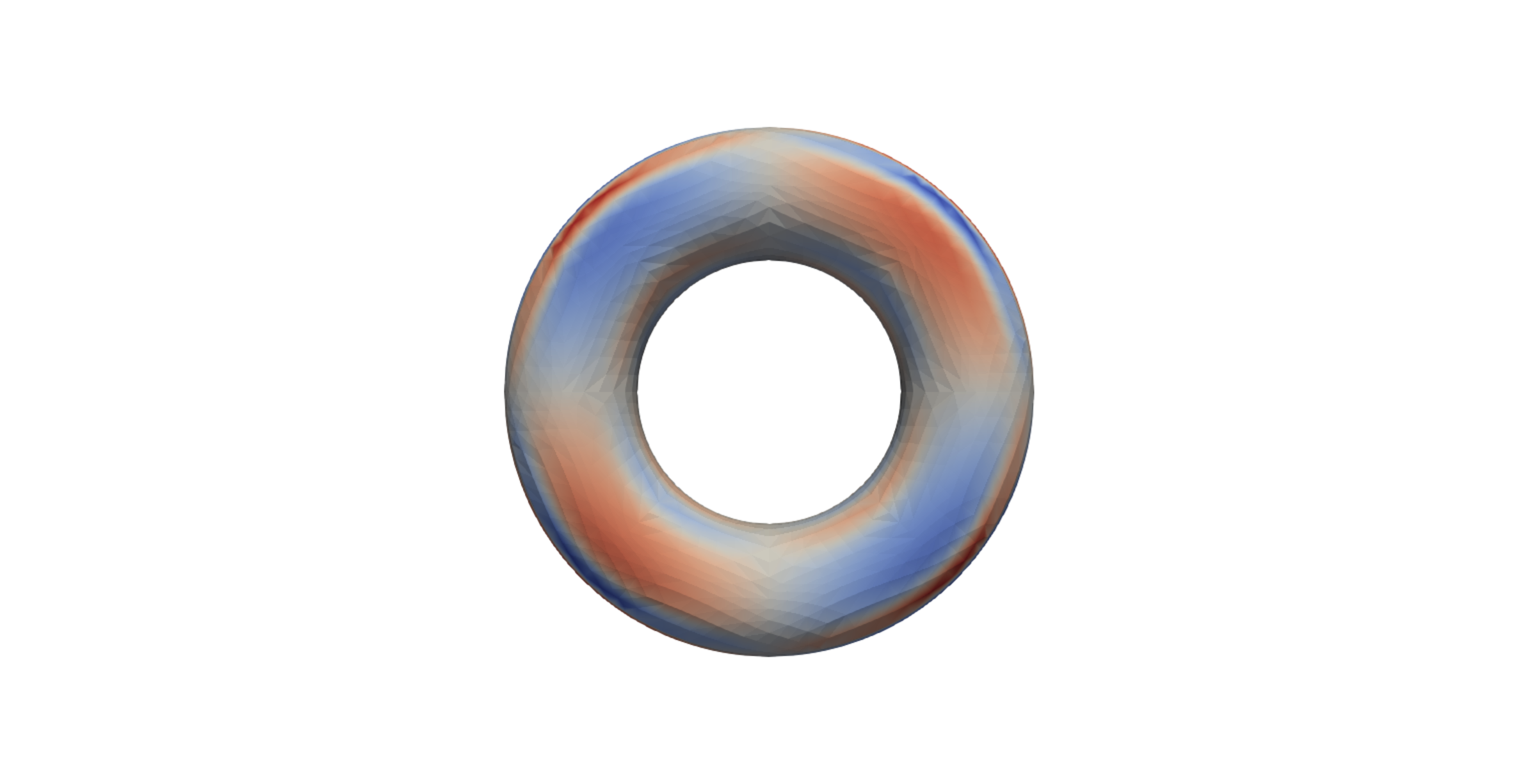}
		\caption{$t=0.$}
	\end{subfigure}%
	\begin{subfigure}{.45\linewidth}
        \centering
		\includegraphics[width=1.2\textwidth]{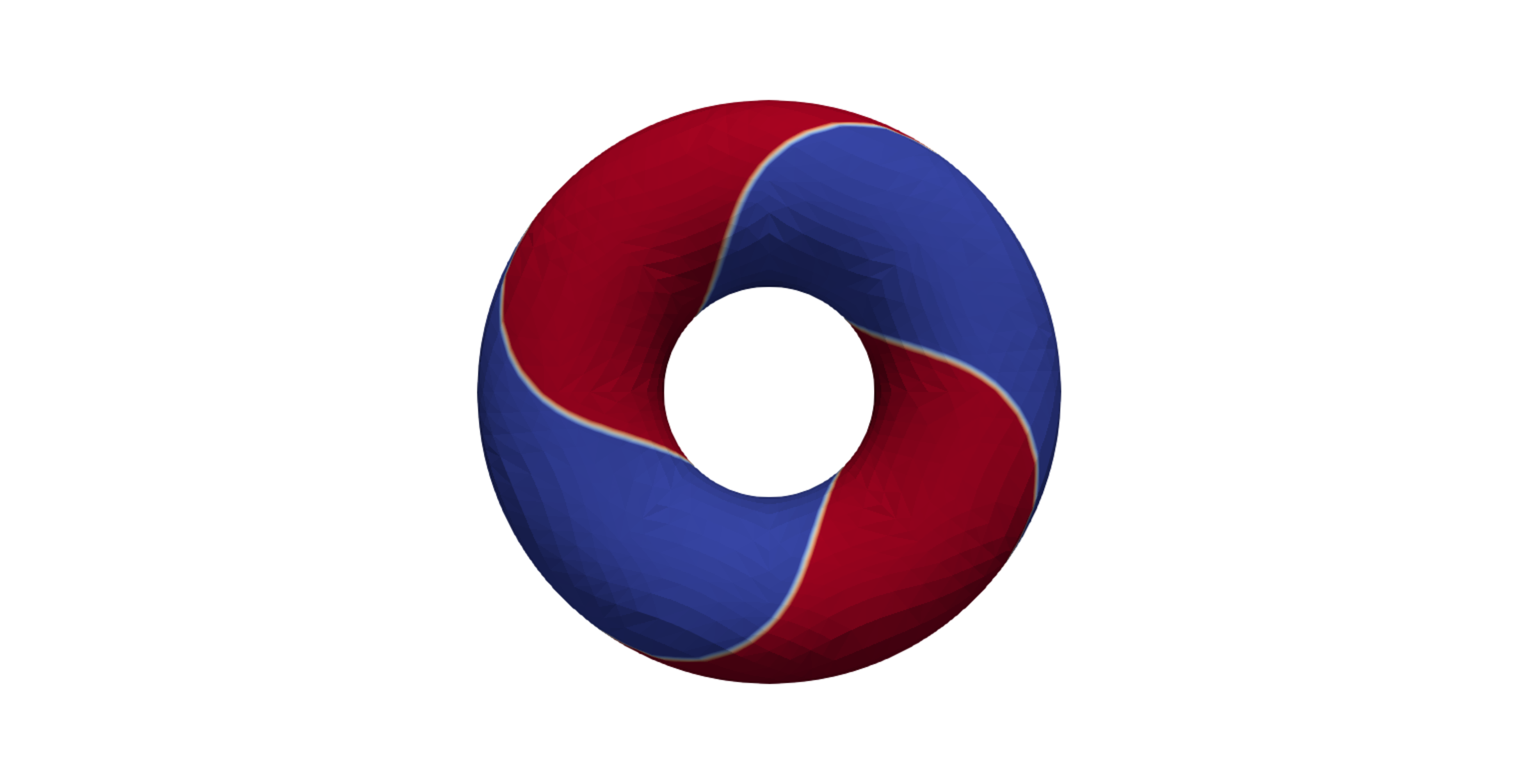}
		\caption{$t=0.125.$}
	\end{subfigure}%
	
	\begin{subfigure}{.45\linewidth}
        \centering
		\includegraphics[width=1.2\textwidth]{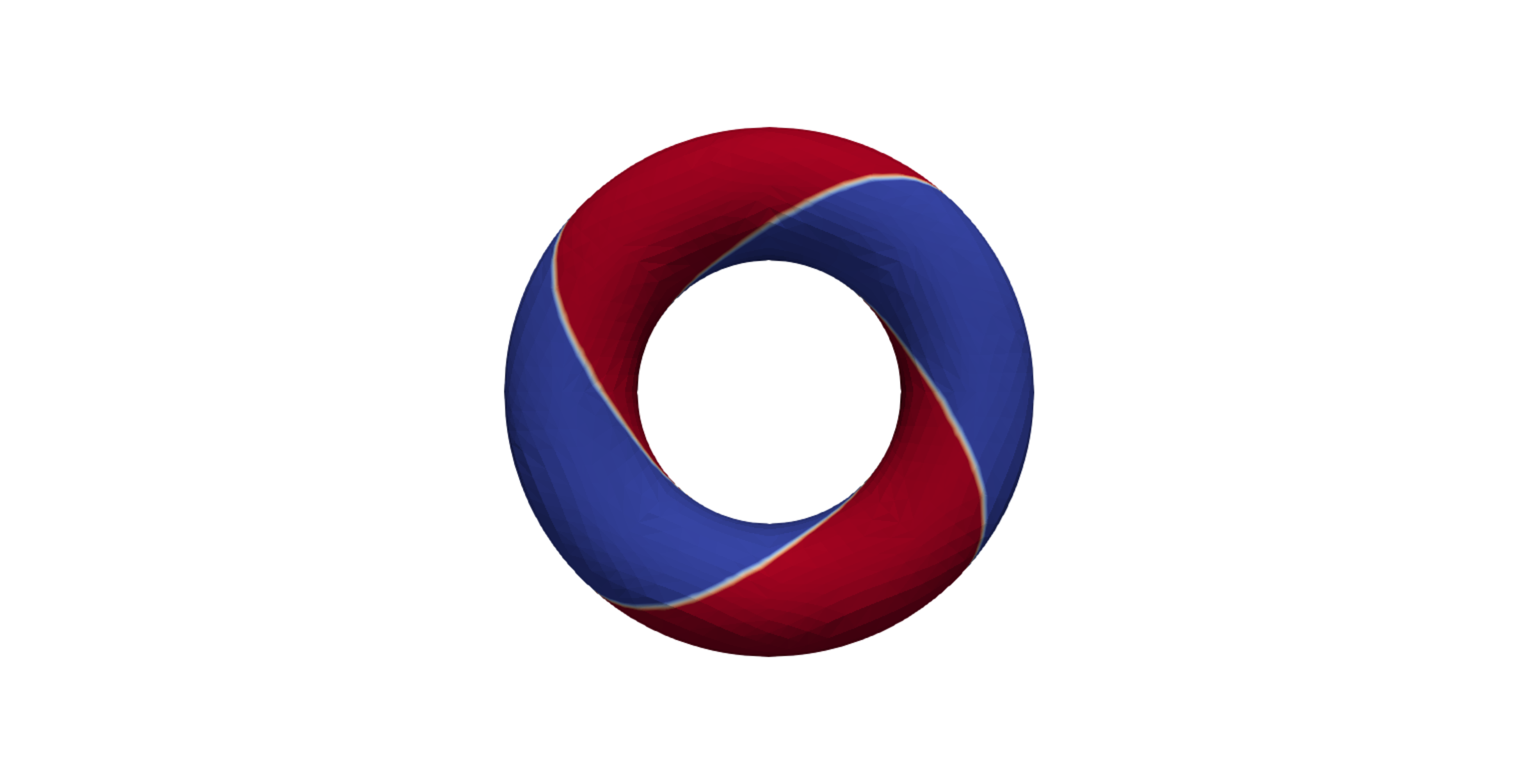}
		\caption{$t=0.25.$}
	\end{subfigure}%
	\begin{subfigure}{.45\linewidth}
        \centering
		\includegraphics[width=1.2\textwidth]{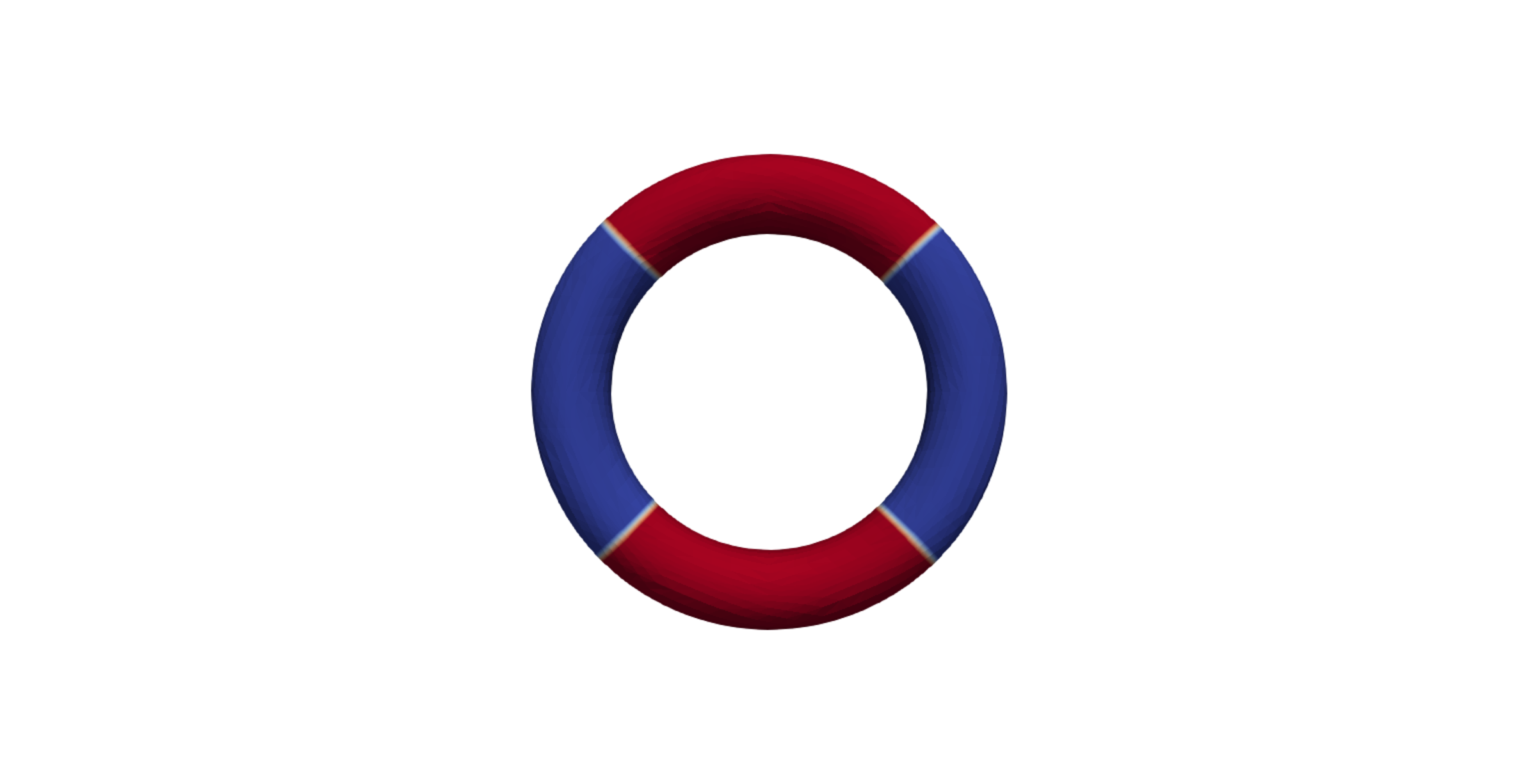}
		\caption{$t=0.375.$}
	\end{subfigure}%
	
	\begin{subfigure}{.45\linewidth}
    \centering
		\includegraphics[width=1.2\textwidth]{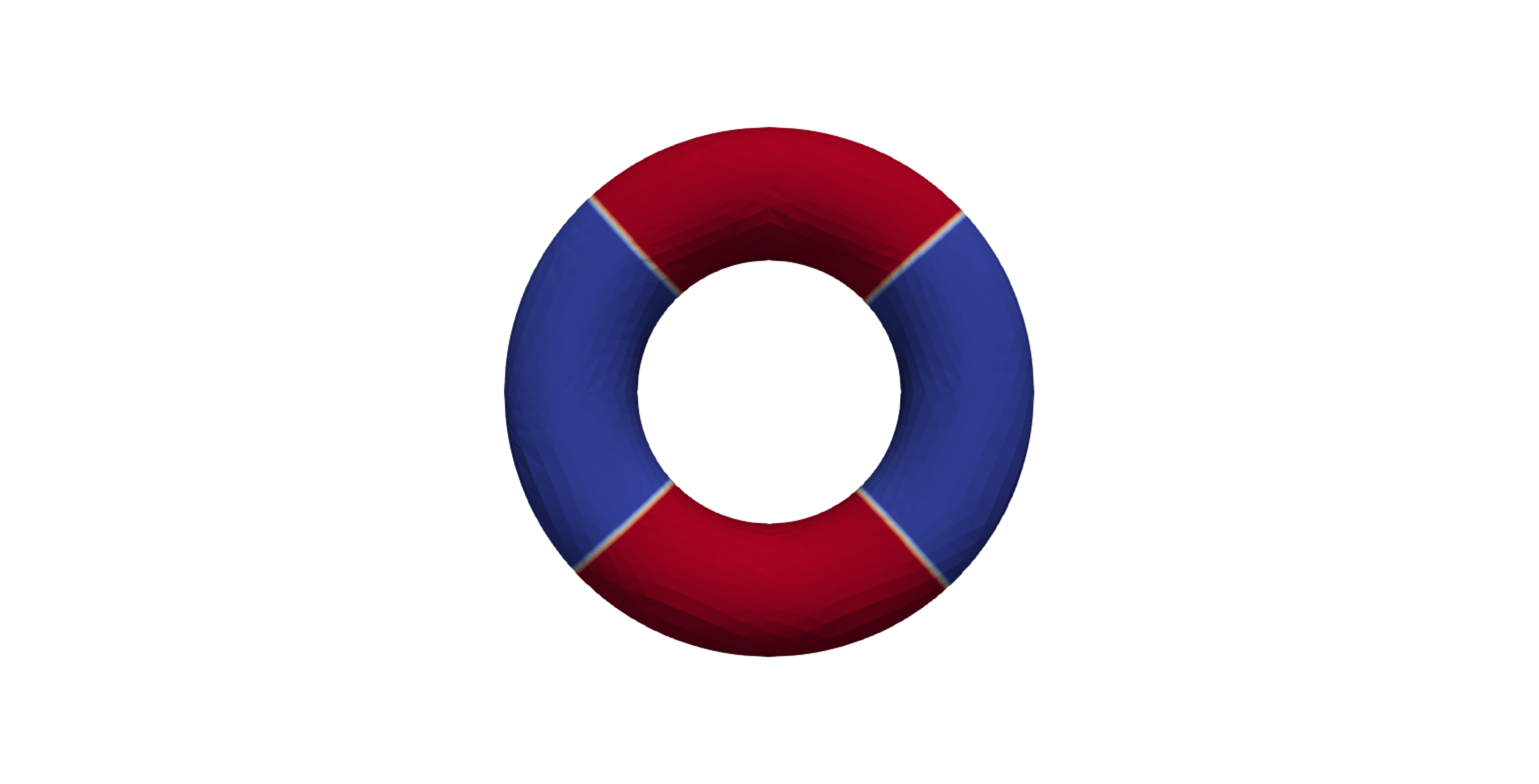}
		\caption{$t=0.5.$}
	\end{subfigure}%
	\begin{subfigure}{.45\linewidth}
    \centering
		\includegraphics[width=1.2\textwidth]{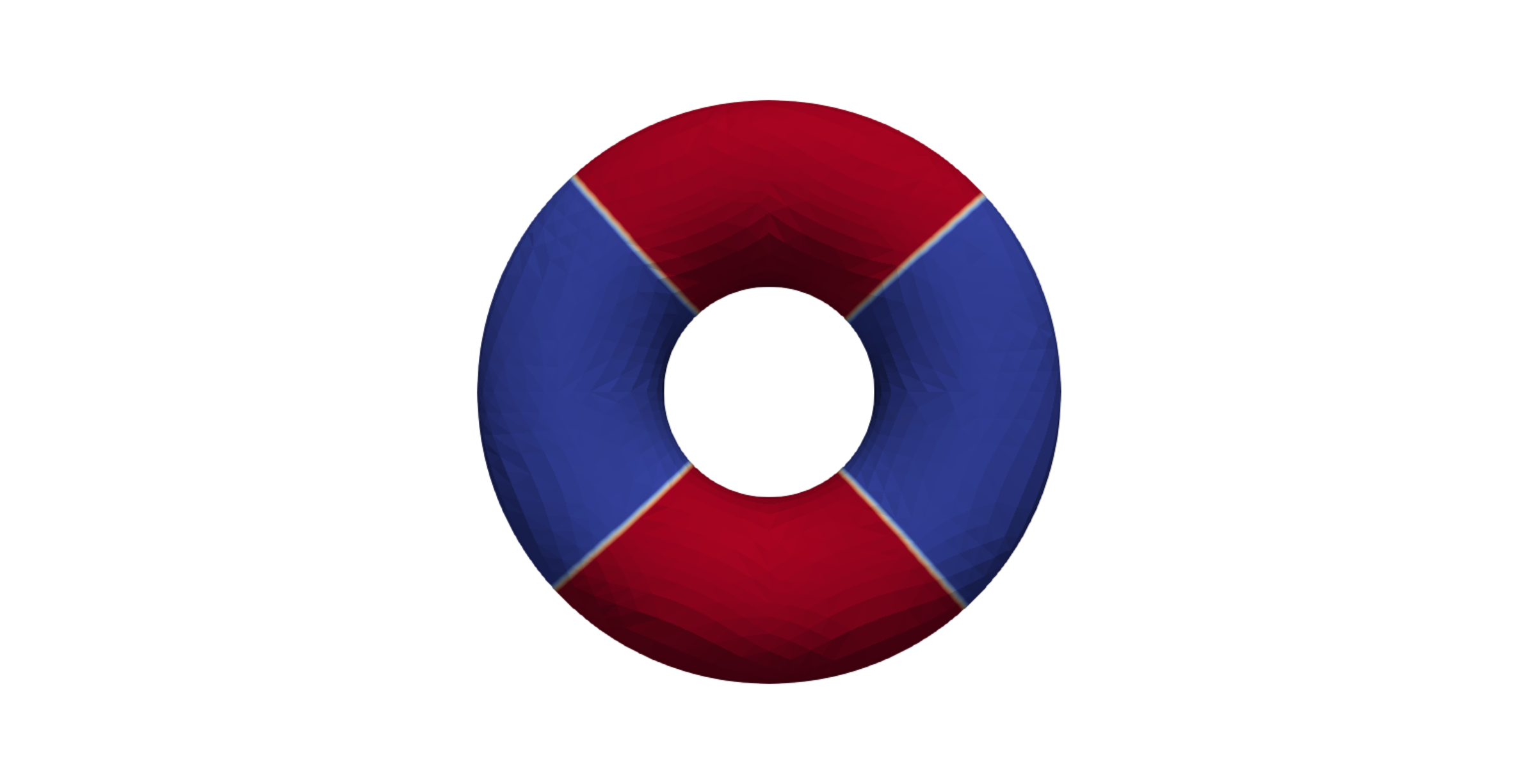}
		\caption{$t=0.625.$}
	\end{subfigure}%
	\caption{Evolution of $u$ (computed by the fully implicit scheme).
		Regions of blue correspond to a negative quantity, and red a positive quantity.
		Solution converges to an oscillatory function, as indicated by the last 3 images.}
	\label{fig:SmoothTorusEvolution}
\end{figure}

\subsection{Sharp interface limit on an evolving sphere}

In this subsection we consider the limit $\varepsilon \rightarrow 0$ on the surface given by the zero level set of the function
$$ \phi(x,y,z,t) = x^2 +y^2 + z^2 - 0.9 - 0.1 \cos(20 \pi t),$$
over a time interval $t \in [0,0.3]$.
On a stationary Euclidean domain it is known that the sharp interface limit of the Cahn-Hilliard equation yields a free boundary problem for the zero level set of the solution, $u$, known as the Mullins-Sekerka problem \cite{alikakos1994convergence,mullins1963morphological,roger2005existence}.
While this problem has been studied extensively on a stationary Euclidean domain, there are few results available for an evolving surface.
We perform some rudimentary numerical experiments to investigate the formation of a sharp interface as $\varepsilon \rightarrow 0$ for initial data given by
\[ u_0(x,y,z) = 0.5 x - 0.1,\]
on a mesh with $h \approx 3.827328 \cdot 10^{-2}$, and a timestep size $\tau = 2\cdot 10^{-5}$.
We vary the interface width $\varepsilon$ but keep these other parameters fixed.
In the limit $\varepsilon \rightarrow 0$ we observe the formation (see Figure \ref{fig:SharpInterfaceFormation}) and propagation (Figure \ref{fig:SharpInterfacePropagation}) of a sharp interface.
Moreover computing the difference of the solutions for $\varepsilon = 0.2, 0.1$ and the solution corresponding to $\varepsilon = 0.05$ we observe that halving $\varepsilon$ results in roughly half the error in the $L^2$ norm\footnote{The error in $u$ for $\varepsilon = 0.2$ was $\approx 1.11$, and the error for $\varepsilon = 0.1$ was $\approx 0.48$.
Similarly the corresponding errors for $w$ were $\approx 5.45$ for $\varepsilon = 0.2$ and $\approx 2.25$ for $\varepsilon = 0.1$.}
These numerical experiments indicate convergence to a sharp interface limit, which one would expect is given by some evolving surface analogue of the Mullins-Sekerka problem --- however the results here are quite basic and should be explored further in future works devoted to the sharp interface limit.

\begin{figure}[ht]
\centering
    \begin{subfigure}{.8\linewidth}
    \centering
		\includegraphics[width=.8\linewidth]{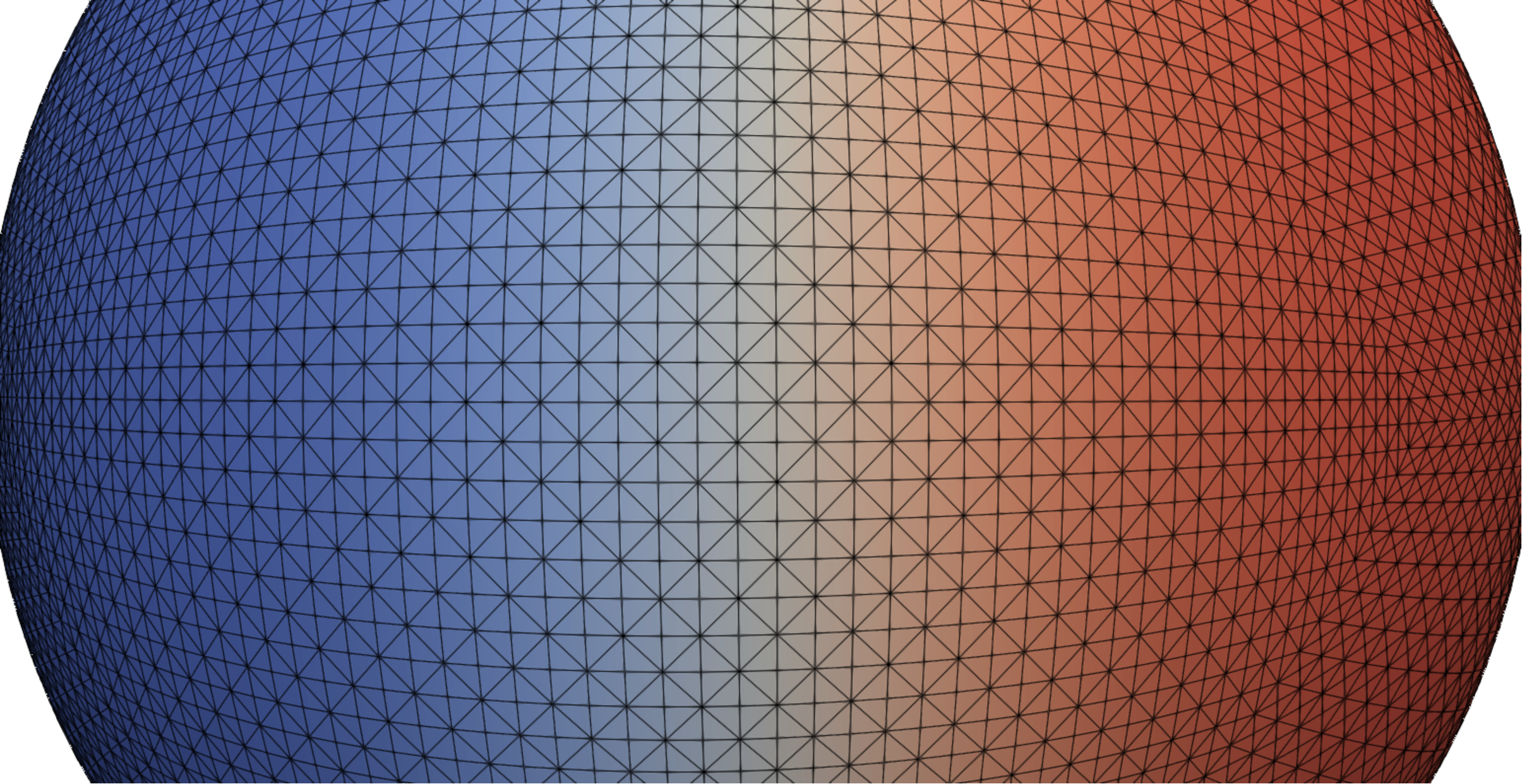}
		\caption{$\varepsilon = 0.2.$}
	\end{subfigure}%
    
	\begin{subfigure}{.8\linewidth}
    \centering
		\includegraphics[width=.8\linewidth]{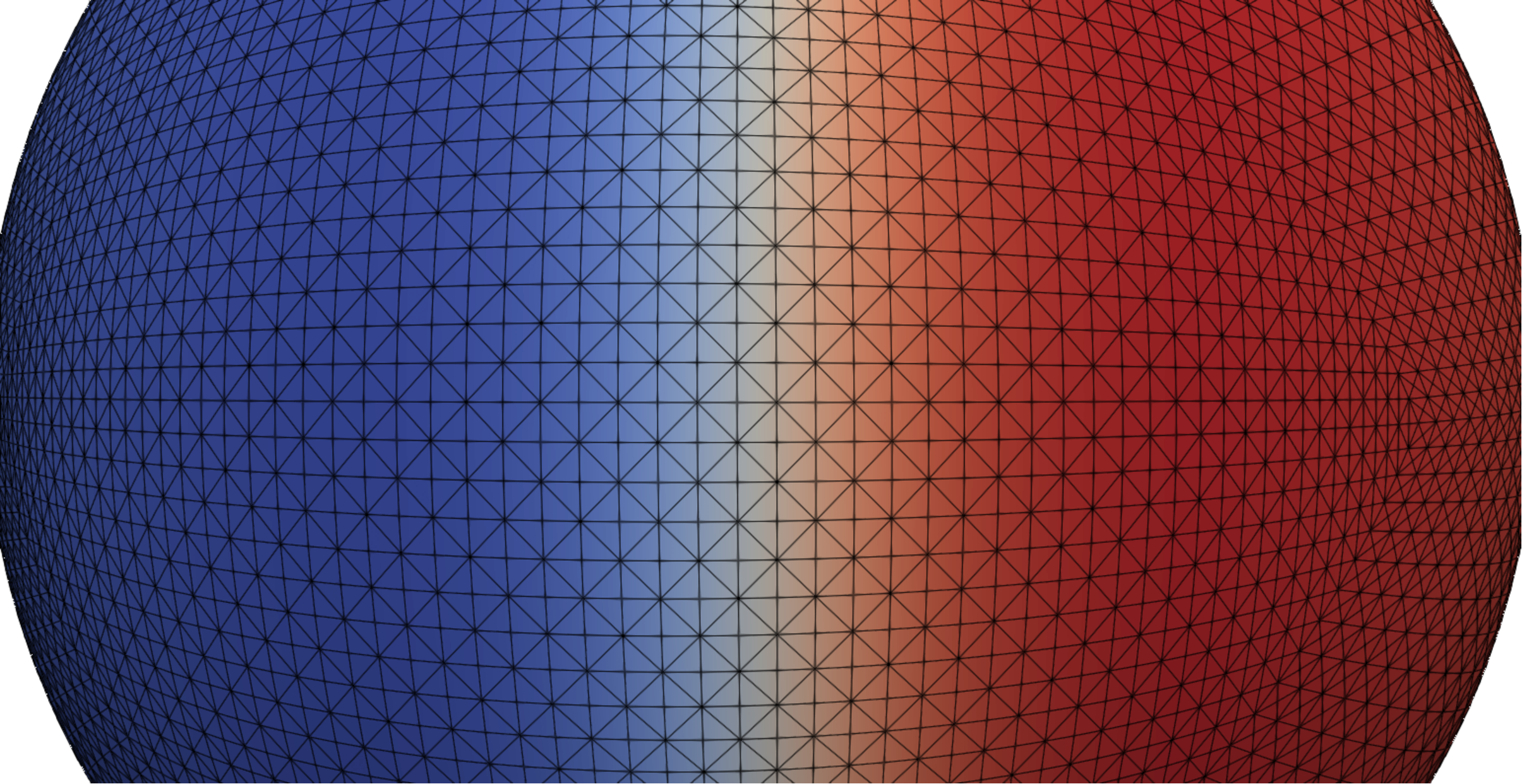}
		\caption{$\varepsilon = 0.1.$}
	\end{subfigure}%

    \begin{subfigure}{.8\linewidth}
    \centering
		\includegraphics[width=.8\linewidth]{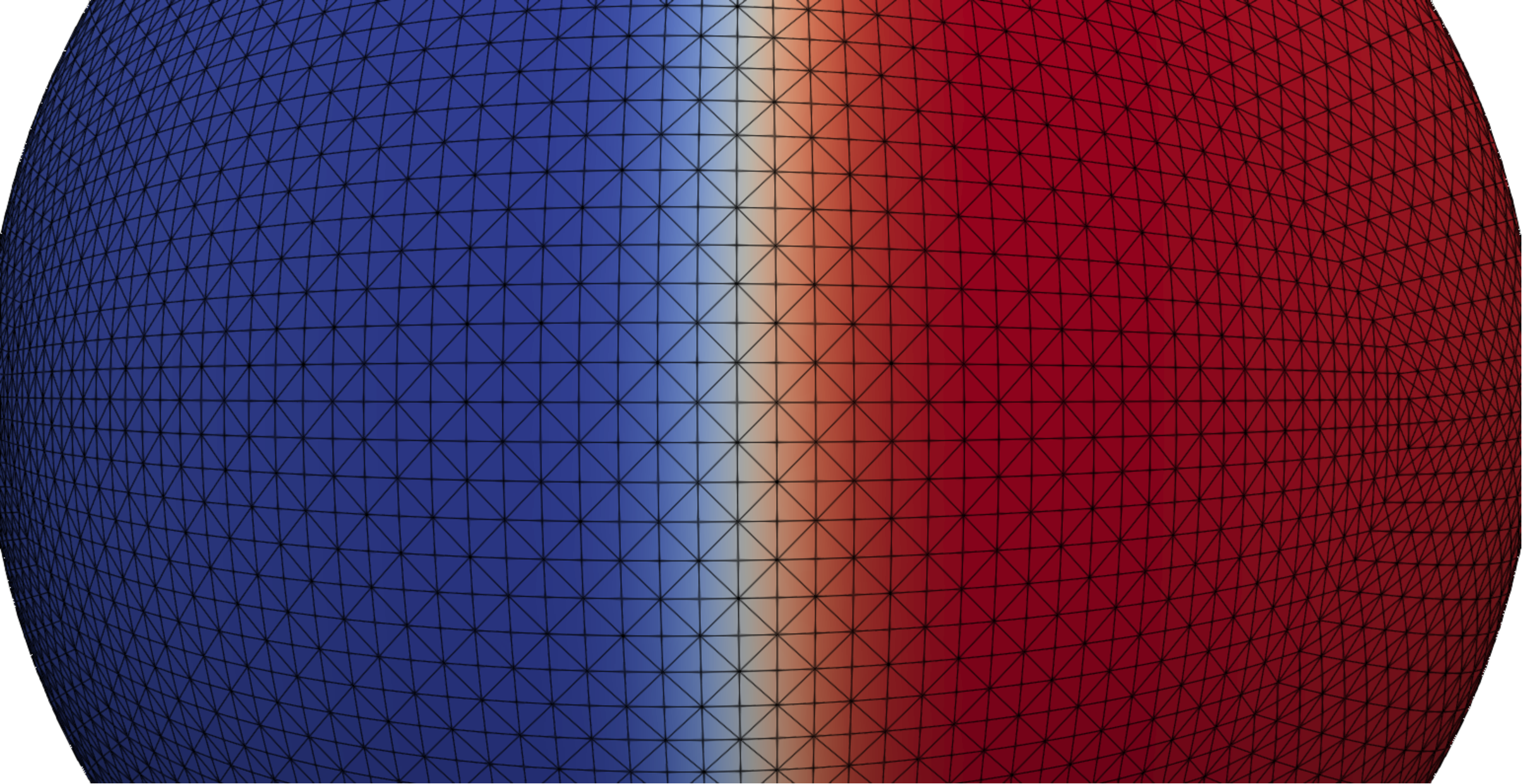}
		\caption{$\varepsilon = 0.05.$}
	\end{subfigure}%
	\caption{Convergence to a sharp interface at $t = 0.3$.}
	\label{fig:SharpInterfaceFormation}
\end{figure}

\begin{figure}[ht]
 \centering
    \begin{subfigure}{.8\linewidth}
    \centering
		\includegraphics[width=.8\linewidth]{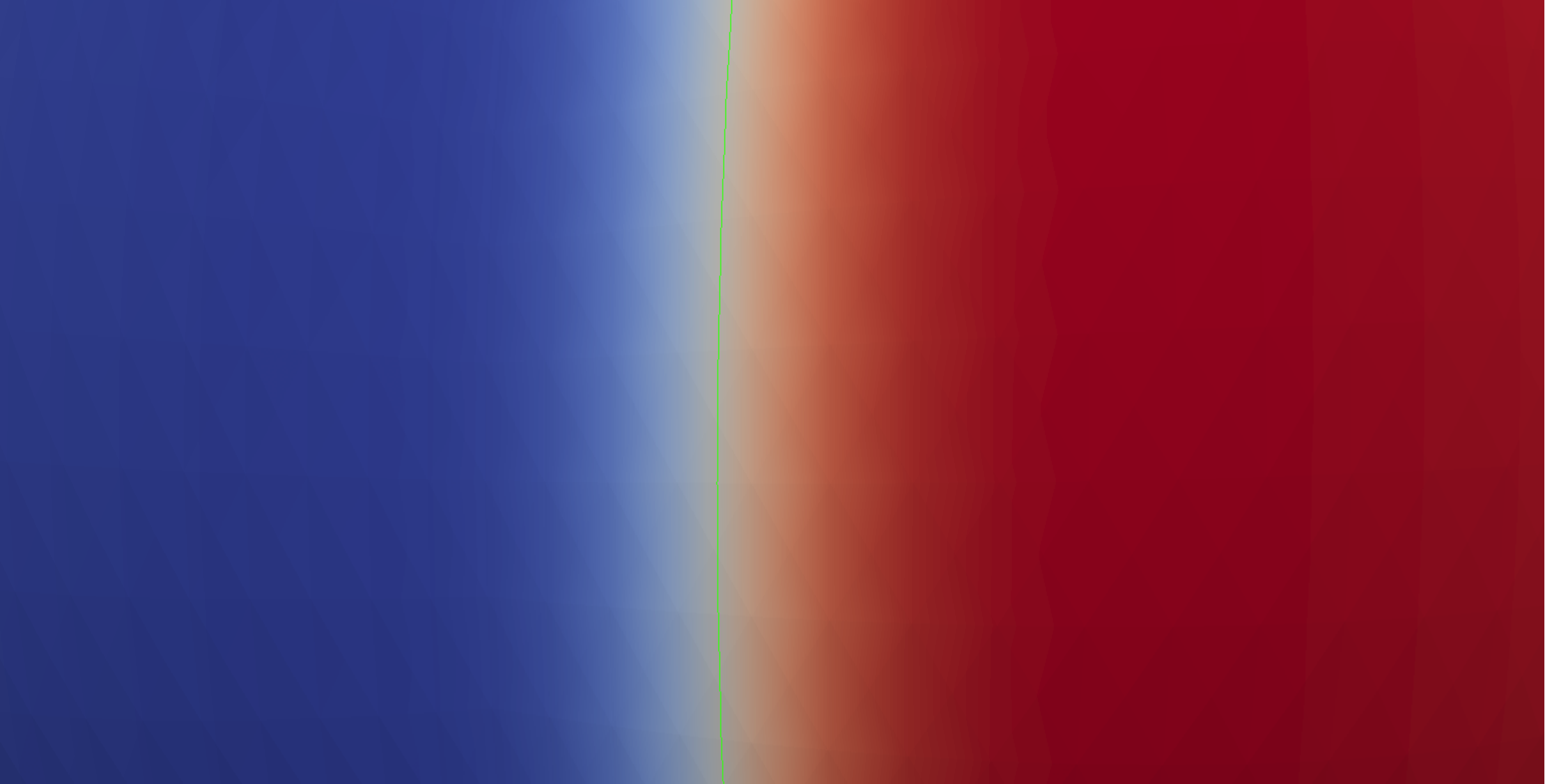}
		\caption{$t = 0.1.$}
	\end{subfigure}%
    
	\begin{subfigure}{.8\linewidth}
    \centering
		\includegraphics[width=.8\linewidth]{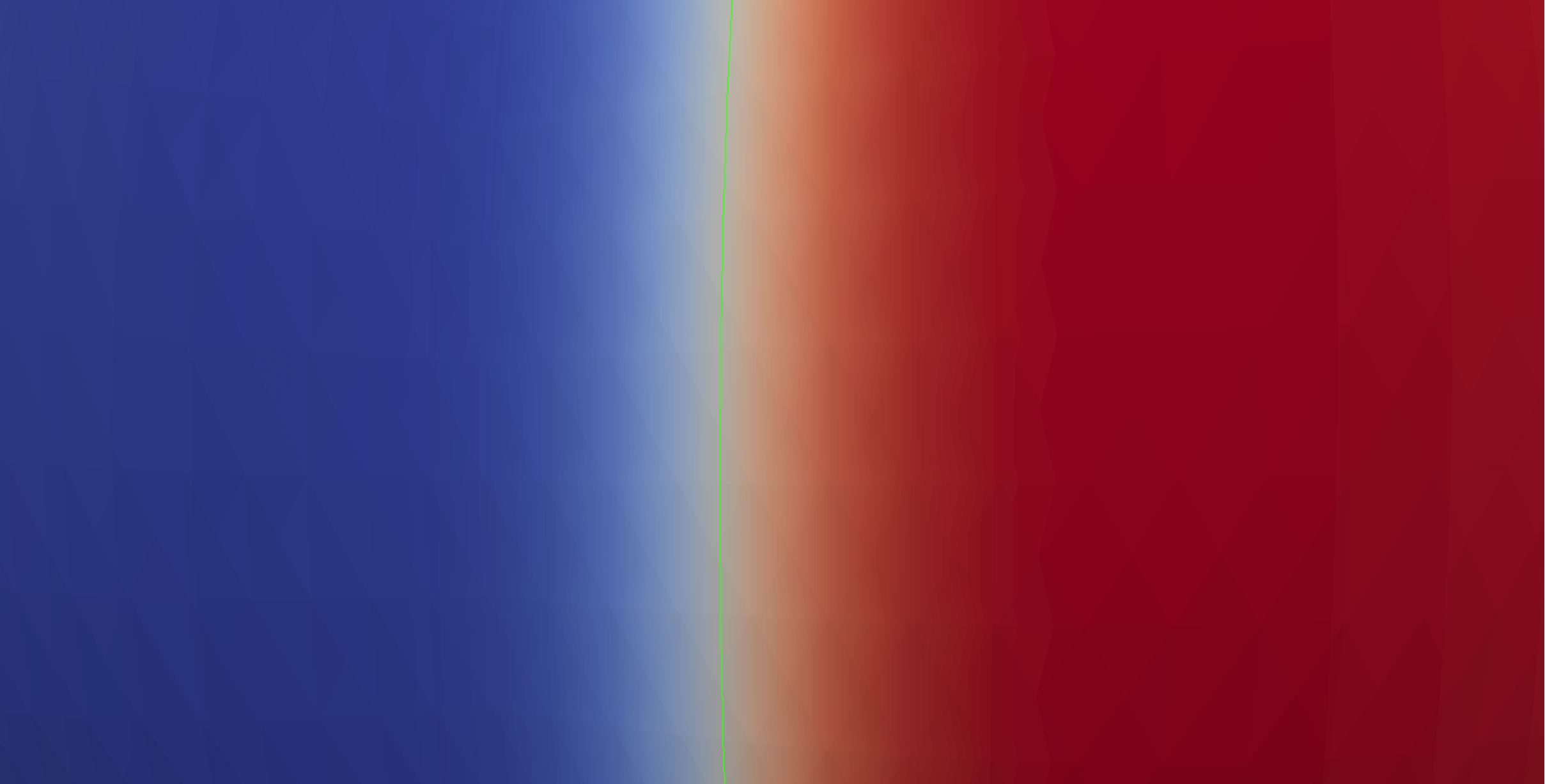}
		\caption{$t = 0.2.$}
	\end{subfigure}%

    \begin{subfigure}{.8\linewidth}
    \centering
		\includegraphics[width=.8\linewidth]{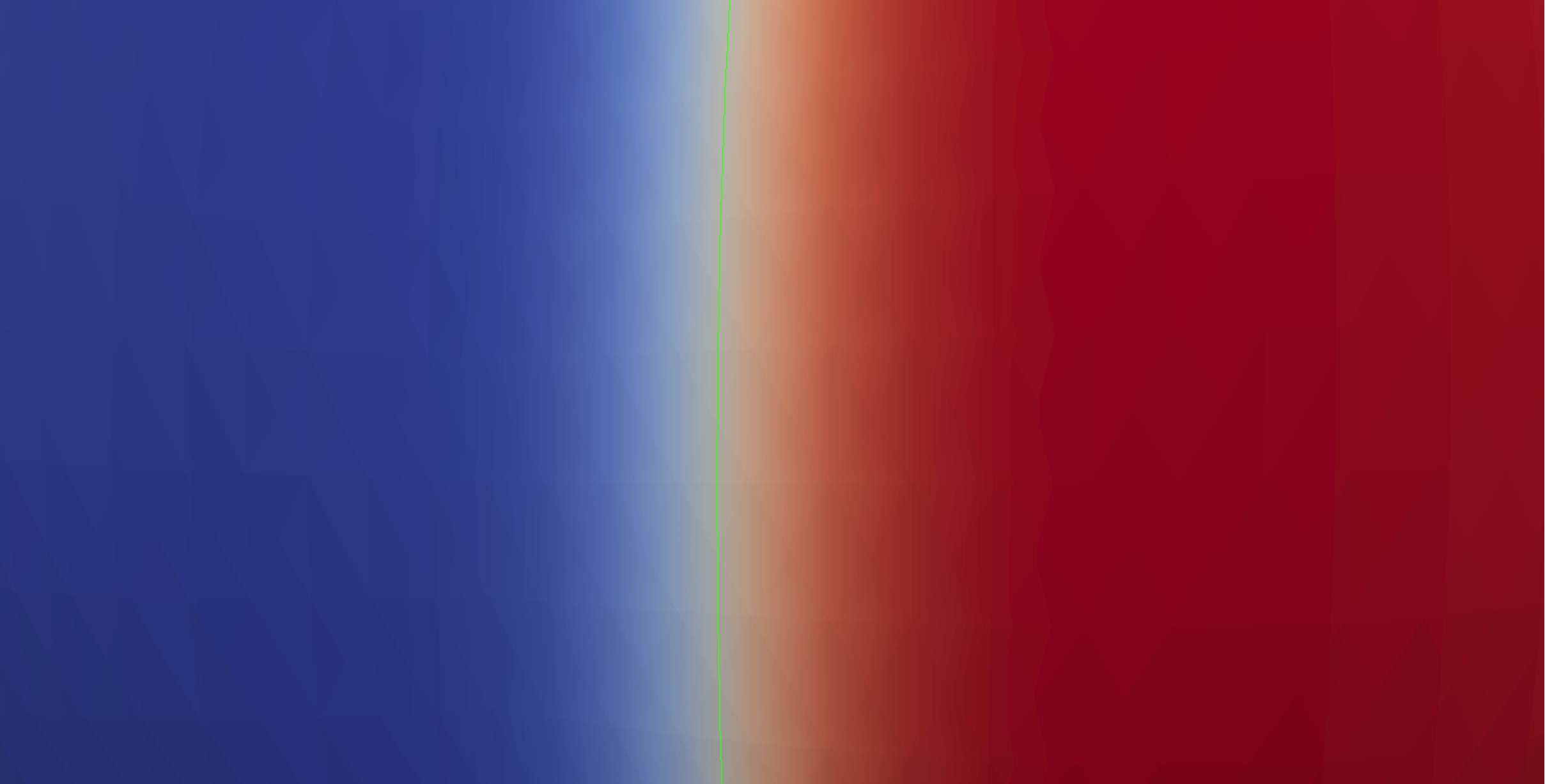}
		\caption{$t = 0.3.$}
	\end{subfigure}%
	\caption{Propagation of a sharp interface on an evolving sphere for $\varepsilon = 0.05$, computed by the fully implicit scheme.
    The centre of the interface is given in green.}
	\label{fig:SharpInterfacePropagation}
\end{figure}

\subsection*{\textbf{Acknowledgments}}
The authors would like to thank Andreas Dedner for helpful discussions regarding computing the experimental order of convergence.
Thomas Sales is supported by the Warwick Mathematics Institute Centre for Doctoral Training, and gratefully acknowledges funding from the University of Warwick and the UK Engineering and Physical Sciences Research Council (Grant number:EP/TS1794X/1).
For the purpose of open access, the author has applied a Creative Commons Attribution (CC BY) licence to any Author Accepted Manuscript version arising from this submission.

\bibliographystyle{acm}
\bibliography{bibliography.bib}

\end{document}